\documentclass[11pt]{amsart}

\usepackage{amsmath,amsthm,amssymb,amsfonts}
\usepackage[dvipsnames]{xcolor}
\usepackage[hidelinks]{hyperref}

\def\mainfile
\usepackage{amscd}
\usepackage[latin1]{inputenc}
\DeclareMathAlphabet{\mathpzc}{OT1}{pzc}{m}{it}

\newcommand{\Bad}{{\mathsf{Bad}}}

\numberwithin{equation}{section}

\swapnumbers
\newtheorem{thm}[subsection]{Theorem}
\newtheorem{coro}[subsection]{Corollary}
\newtheorem*{cor*}{Corollary}
\newtheorem{lemma}[subsection]{Lemma}

\newtheorem{propos}[subsection]{Proposition}

\newtheorem*{thm*}{Theorem}
\newtheorem*{thma*}{Theorem A}
\newtheorem*{thmb*}{Theorem B}

\newtheorem{defn}[subsection]{Definition}
\newtheorem{cor}{Corollary}[subsection]

\newtheorem{conj}{Conjecture}

\newcounter{consta}
\renewcommand{\theconsta}{{A_{\arabic{consta}}}}
\newcommand{\consta}{\refstepcounter{consta}\theconsta}

\newcounter{constb}

\newcounter{constc}

\newcounter{constD}

\newcommand\absolute[1]{{\left|{#1}\right|}}
\newcommand\norm[1]{{\left\|{#1}\right\|}}

\newcommand{\elon}[1]{{\color{Fuchsia}#1}}
{
  \color{Fuchsia}%
}%
{}
\newcommand{\newelon}[1]{{\color{ForestGreen}#1}}
{
  \color{ForestGreen}%
}%
{}
\DeclareMathOperator{\exceptional}{Exc}

\def\bbz{\mathbb{Z}}
\def\bbq{\mathbb{Q}}

\def\bbr{\mathbb{R}}

\def\bbc{\mathbb{C}}

\def\bbn{\mathbb{N}}
\def\Gbf{\mathbb{G}}

\def\Q{\bbq}
\def\Z{\bbz}
\def\R{\bbr}

\def\fcal{\mathcal{F}}

\def\bcal{\mathcal{B}}

\def\ncal{\mathcal{N}}

\def\hcal{\mathcal{H}}
\def\Hcal{\mathcal{H}}

\def\hfrak{\mathfrak{h}}

\def\ufrak{\mathfrak{u}}

\def\gfrak{\mathfrak{g}}

\def\lfrak{\mathfrak{l}}

\def\Gbf{\mathbf{G}}
\def\Hbf{\mathbf{H}}
\def\Lbf{\mathbf{L}}

\def\Wbf{{\mathbf{W}}}
\def\Mbf{\mathbf{M}}

\def\vbf{\mathbf{v}}

\def\H{\Hbf}
\def\G{\Gbf}

\def\apz{\mathpzc{a}}
\def\bpz{\mathpzc{b}}

\def\dpz{\mathpzc{d}}

\def\vpz{\mathpzc{v}}
\def\upz{\mathpzc{u}}

\def\wpz{\mathpzc{w}}

\def\zpz{\mathpzc{z}}

\def\SL{{\rm{SL}}}

\def\GL{{\rm{GL}}}
\def\Lie{{\rm Lie}}
\DeclareMathOperator\height{ht}

\DeclareMathOperator\Ad{Ad}

\def\sl{{\mathfrak{sl}}}
\def\cfun{{\mathsf c}}
\def\gadd{{\rm G}_{\rm a}}

\def\qpl{{\mathcal S}}
\def\places{{\Sigma}}

\def\la{\lambda}

\def\vare{\varepsilon}

\def\zg0{Z_{G_\omega}(s)}

\def\zg{Z_G(s)}

\def\be{\begin{equation}}
\def\ee{\end{equation}}

\def\lplus{\Mbf}

\def\llplus{M}

\def\vsr{{r}}       
\def\vsballr{{B_{V_H}(\vpz_H,\vsr)}}

\def\RMH{{m}}

\def\rad{\operatorname{R}}

\def\dotll{{\ll}}

\def\rank{{\rm rk}}
\def\zeds{\bbz_\places}

\def\vsdist{{\rm d}}

\begin{document}
\title[Quantitative behavior of unipotent flows]{Quantitative behavior of unipotent flows and an effective avoidance principle}
\date{\today}

\author{E.~Lindenstrauss}
\address{E.L.: The Einstein Institute of Mathematics, Edmond J.\ Safra Campus, 
Givat Ram, The Hebrew University of Jerusalem, Jerusalem, 91904, Israel}
\email{elon@math.huji.ac.il}
\thanks{E.L.\ acknowledges support by the ISF 891/15. Parts of this work was written when E.L. was a member at the IAS in Princeton, supported by NSF grant DMS-1638352}

\author{G.~Margulis}
\address{G.M.: Mathematics Department, Yale university, New Haven, CT 06511}
\email{gregorii.margulis@yale.edu}
\thanks{G.M.\ acknowledges support by the NSF grant DMS-1265695}

\author{A.~Mohammadi}
\address{A.M.: Department of Mathematics, University of California, San Diego, CA 92093}
\email{ammohammadi@ucsd.edu}
\thanks{A.M.\ acknowledges support by the NSF grants 
DMS-1724316, 1764246, 1128155, and Alfred P.~Sloan Research Fellowship.}

\author{N.~Shah}
\thanks{N.S.\ acknowledges support by the NSF grant DMS-1700394}
\address{N.S.: Department of Mathematics, The Ohio State University, Columbus, OH 43210}
\email{shah@math.osu.edu}

\begin{abstract}
We give an effective bound on how much time orbits of a unipotent group $U$ on an arithmetic quotient $G / \Gamma$ can stay near homogeneous subvarieties of $G / \Gamma$ corresponding to $\bbq$-subgroups of $G$. In particular, we show that if such a $U$-orbit is moderately near a proper homogeneous subvariety of $G / \Gamma$ for a long time, it is very near a different homogeneous subvariety. Our work builds upon the linearization method of Dani and Margulis.

Our motivation in developing these bounds is in order to prove quantitative density statements about unipotent orbits, which we plan to pursue in a subsequent paper. New qualitative implications of our effective bounds are also given.
\end{abstract}
\maketitle

\section{Introduction}\label{sec:introduction}
A basic challenge in homogeneous dynamics is the quantitative understanding of behavior of orbits, in particular of unipotent orbits. In this paper, we give a sharper form of the Dani-Margulis linearization method~\cite{DM}, that allows to control the amount of time a unipotent trajectory spends near invariant subvarieties of a homogeneous space; related techniques were also considered by Shah in~\cite{Shah}. 

One important use of this technique is to be able to relate the behavior of individual unipotent (or unipotent-like, see e.g. \cite{EMS}) orbits with Ratner's landmark measure classification result \cite{Ratner-Ann}. This result says that any measure invariant and ergodic under a connected unipotent group $U$ on a homogeneous space $G/\Gamma$ has to be in one of countably many families; for the cases of $G / \Gamma$ and unipotent group $U$ we will consider, the group $U$ acts ergodically on  $G / \Gamma$ with respect to the uniform measure on $G/\Gamma$, hence this uniform measure is one of the countably many possibilities. All other ergodic measures will be supported on proper homogeneous subvarieties of $G / \Gamma$. If one is able to show, using linearization or a different technique, that a given collection of orbits under consideration of increasing size do not spend much time next to any one of these countably many families of not fully supported invariant measures, then one is able to conclude using the measure classification that this collection of orbits tends to become equidistributed in $G / \Gamma$. We note that for the special case when one looks at a single orbit of a one parameter unipotent group Ratner was able to establish such an avoidance by a different argument in \cite{Ratner-Top}.

However, the linearization technique of \cite{DM} is interesting in its own sake, and in fact originated in work of Dani and Margulis before the proof of measure classification such as \cite{DM-SL3} in order to give a purely topological proof of Raghunathan Conjecture for the action of a generic one parameter unipotent group on $\SL_3 (\bbr) / \SL_3 (\bbz)$. Notably, unlike the techniques of Ratner used to prove the measure classification result in \cite{Ratner-Ann, Ratner-Duke} or the techniques used to give a related but different proof of this result by Margulis and Tomanov in~\cite{MT}, which in particular rely on results such as the pointwise ergodic theorem and Luzin's theorem which are hard to make effective, the linearization technique relies essentially only on the polynomial nature of the action: not only are the elements of the unipotent group (considered as a subgroup of some $\SL_N$) polynomial, but the same holds for any linear representation of~$G$.

In a subsequent paper we plan to make essential use of the results of this paper in order to provide a fully effective orbit closure classification theorem for unipotent flows on arithmetic homogenous spaces (albeit with very slow rates). We provide some other applications of independent interest here.

Somewhat surprisingly, many of the most striking applications of the theory of unipotent flows to number theory require working in the $S$-arithmetic context, i.e. for products of real and $p$-adic groups (here we prefer to use $\Sigma$ for the set of places instead of the more traditional $S$, so we refer to this case as the $\Sigma$-arithmetic case). Ratner's measure classification result was generalized to this context by Ratner \cite{Ratner-Duke} and by Margulis and Tomanov~\cite{MT}; the linearization techniques of Dani and Margulis were adapted to this context by Tomanov and by Gorodnik and Oh in~\cite{T, GrOh-AdelEq}. In view to potential applications, our paper is written for $\Sigma$-arithmetic quotients. For simplicity we state here in the introduction the main results in the special case where we consider the action of a one-parameter unipotent group and consider only real algebraic groups, deferring stating the slightly more technical general statements to~\S\ref{sec:S-arith}.
We emphasize that in order to get a fully explicit and effective result, we assume that the lattice is arithmetic. By Margulis Arithmeticity Theorem this assumption automatically holds for a large class of groups $G$, and in any case arithmetic quotients are the only type of quotients $G/\Gamma$ that seem to appear in number theoretic applications.

The nondivergence result of Margulis \cite{Margulis-Nondiv}, which were sharpened by Dani in \cite{Dani} are effective and have been given a very explicit and effective form by Kleinbock and Margulis in \cite{KM}. The technique of linearization is related, but we are not aware of an effective treatment of the main results in~\cite{DM}, and doing so in this paper relies on employing an effective Nullstellensatz by Masser and W\"ustholz~\cite[Thm. IV]{EffNul-Mass-Wus} as well as some local non-vanishing theorems related to Lojasiewicz inequality by Brownawell and Greenberg~\cite{Brwell-Loj, GrBerg-Loj-1, GrBerg-Loj-2}. Moreover, since we are not content with analysing what happens in the limit, we need to be able to analyse trajectories that are somewhat near a subvariety for a long time, which is an issue that has not been discussed in previous works on the linearization method.

Let $\Gbf$ be a connected $\bbq$-group and put $G=\Gbf(\bbr)$.
We assume $\Gamma$ is an arithmetic lattice in $G$. 
More specifically, fix an embedding $\iota:\Gbf\rightarrow{\SL}_N,$ 
defined over $\bbq$ so that $\iota(\Gamma)\subset\SL_N(\bbz)$. 
Using $\iota$, we identify $\G$ with $\iota(\G)\subset\SL_N$ and hence $G\subset\SL_N(\bbr)$.  
Note that using the restriction of scalars from number fields to $\bbq$, our results are applicable also in the case of groups defined over a general number field.

Let $U=\{u(t):t\in\bbr\}\subset G$ be a one parameter unipotent subgroup of $G$, and put $X=G/\Gamma$.

Define the following family 
\[
\mathcal{H}= \Bigl\{\Hbf\subset\G:\text{$\H$ is a connected $\bbq$-subgroup and $\rad(\H)=\rad_u(\H)$}\Bigr\} 
\]
where $\rad(\H)$ (resp.\ $\rad_u(\H)$) denotes the solvable (resp.\ unipotent) radical of $\H.$ 
Alternatively, $\Hbf\in\hcal$ if and only if $\Hbf$ is a connected $\bbq$-subgroup 
which is generated by unipotent subgroups over the algebraic closure of $\bbq.$ 
By a theorem of Borel and Harish-Chandra, $\H(\bbr)\cap\Gamma$ is a lattice in $\H(\bbr)$
for any $\H\in\mathcal H$.  

\medskip
\begin{center}
\textbf{Our standing assumption is that $\G\in\Hcal$ and that $U$ is not contained in $\H(\bbr)$ for any proper normal $\H \lhd \G$}.
\end{center}
\medskip

For any $\H\in\hcal$ put $H=\H(\bbr)$, we also write $H\in\hcal$. 
Define
\[
N_G(U,H):=\{g\in G: Ug\subset gH\}.
\]
Note that $N_G(U,H)$ is an $\bbr$-subvariety of $G.$ 
Moreover, if $H\lhd G$ and $U\subset H$, then $N_G(U,H)=G$.

Put
\[
\mathcal{S}(U)=\Bigl(\bigcup_{\substack{{H}\in\mathcal{H}\\ {H}\neq{G}}}\; N_G(U,H)\Bigr)/\Gamma
\hspace{2mm}\mbox{and}\hspace{2mm}\mathcal{G}(U)=X\setminus\mathcal{S}(U).
\]

Following Dani and Margulis,~\cite{DM}, points in $\mathcal{S}(U)$ are called {\it singular} with respect to $U,$
and points in $\mathcal{G}(U)$ are called {\it generic} with respect to $U$ 
--- these are, a priori, different from the measure theoretically generic  points in the sense of Furstenberg for
the action of $U$ on $X$ equipped with the the $G$-invariant probability measure (see e.g.\ \cite[p.~98]{Glasner} for a definition);  
however, any measure theoretically generic point is generic in this explicit sense as well. 
In the early 1990's Ratner proved the remarkable result~\cite{Ratner-Top}, previously conjectured by Raghunathan, that for every $x\in\mathcal G(U)$
we have $\overline{Ux}=X$. Prior to Ratner's proof of the general case of Raghunathan's Conjecture in~\cite{Ratner-Top}, important cases of Raghunathan's Conjecture were proven in~\cite{Margulis-Oppenheim, DM-SO21, DM-SL3}.

Roughly speaking our main theorems guarantee that unless there is an explicit obstruction, {\em most} points on 
a unipotent orbit are generic. We begin with the following statement which follows from our main effective theorems in this paper. 

\begin{thm}\label{thm:linearization-noneff}
Let $\eta>0$. Let $\{H_i:1\leq i\leq r\}\subset\Hcal$ be a finite subset consisting of proper subgroups, and for each $1\leq i\leq r$
let $\mathcal C_i\subset N_G(U,H_i)$ be a compact subset. 
There exists an open set $\mathcal O=\mathcal O(\eta,\{H_i\}, \{\mathcal C_i\})$ so that $X \setminus \mathcal O$ is compact and disjoint from $\cup_i \mathcal C_i\Gamma/\Gamma$, and so that for every $x\in\mathcal G(U)$ there exists some $T_0=T_0(\eta,\{H_i\}, \{\mathcal C_i\}, x)$ 
so that for all $T\geq T_0$ we have 
\[
|\{t\in[-T,T]: u(t)x\in\mathcal O\}|< \eta T 
\] 
\end{thm}

We note that this theorem can also be deduced by combining Ratner's measure classification theorem,~\cite{Ratner-Ann}, and results in~\cite{DM}; however this would only give a non-effective proof of the above statement.
Without appealing to~\cite{Ratner-Ann} and only utilizing statements in~\cite{DM} (where the proof is essentially effective), one does not get uniformity as in Theorem~\ref{thm:linearization-noneff}:
indeed, from the argument in \cite{DM} the set $\mathcal O$ above will depend on the initial point $x$. 
This distinction is similar to the difference between the non-divergence statement given by Dani in~\cite{Dani} and the dependence on the base point in Margulis'~\cite{Margulis-Nondiv}.

\subsection{Effective versions of linearization}\label{sec:eff-linearization}
The main theorems in this paper yield a 
more precise and \emph{effective} information about the compact set $X \setminus \mathcal O$ appearing above, with a \emph{polynomial} dependence on the relevant parameters. 
We need some preliminary notation before we can state our main results.

\medskip


Let $\|\;\|_{\infty}$ (or simply $\|\;\|$) denote the max norm on $\sl_N(\bbr)$ with respect to the standard basis;
this induces a family of norms on $\wedge\sl_N(\bbr)$, which we continue to denote by $\|\;\|_{\infty}$ (or simply $\|\;\|$). We also let $\|\;\|$ be a norm on $\SL_N(\bbr)$ fixed once and for all. 
For every $g\in\SL_N(\bbr)$, in particular for any $g\in G$, we let 
\[
|g|=\max\{\|g\|,\|g^{-1}\|\}.
\]

Let $\gfrak=\Lie(G)$ and put $\gfrak(\bbz):=\gfrak\cap{\sl}_N(\bbz)$.

For every $\eta>0$, we define 
\[
X_\eta=\Bigl\{g\Gamma\in X: \min_{0\neq v\in\gfrak(\bbz)}\|\Ad(g)v\|\geq \eta\Bigr\}.
\]
For every $\eta > 0$ the space $X _ \eta$ is compact (see \S\ref{sec:space-lattices} and Lemma~\ref{prop:eff-red-theory}), and $\bigcup_ {\eta > 0} X _ \eta = G / \Gamma$.

Recall that $U$ is a one parameter unipotent subgroup of $G$. Fix a $\zpz\in\gfrak$ with $\|\zpz\|=1$ so that
\be\label{def of z in intro}
U=\{u(t)=\exp(t\zpz): t\in\R\}.
\ee

Let $\Hbf\in\Hcal$ be a nontrivial proper subgroup of $\Gbf$ and put  
\[
\text{$\rho_H:=\wedge^{\dim\H}\Ad\quad$ and 
$\quad V_H:=\wedge^{\dim\H}\gfrak.$}
\] 
The representation $\rho_H$ is defined over $\bbq.$ 
 
Let ${\vbf_\H}$ be a primitive integral vector in $\wedge^{\dim\H}\Lie(\G)$ 
corresponding to the Lie algebra of $\H$, i.e., we fix a $\Z$-basis for $\Lie(H)\cap\sl_N(\Z)$, and let ${\vbf_\H}$ be the corresponding wedge product. 

We also view $\vbf_\H$ as an element in $\wedge^{\dim\H}\gfrak$; 
in order to put an emphasis on the local nature of this vector, we will denote it by $\vpz_H$. 
Define 
\[
\eta_H(g):=\rho_H(g)\vpz_H\;\text{ for every $g\in G$}.
\]


With this notation, for an element $\Hbf\in\hcal,$ we have
\[
N_G(U,H)= \{g\in G: \zpz\wedge \eta_H(g)=0\}.
\]
Note that $N_G(U,H)$ is a variety and could change drastically under small perturbations of $U$. 
However, {\em effective} notions must be stable under small perturbations. 
We will use the above finite dimensional representations to give an effective notion of generic points. 
The integer vector $\vbf_\H$ also allows us to give a notion of arithmetic complexity for subgroups in $\hcal$ by defining the height of the group $\Hbf$ to be
\be
\height(\Hbf ):=\|\vbf_\H\|_{\infty}.
\ee
Thus the height of a $\bbq$-group $\Hbf$ is given by the height of the correspoding point in the Grassmanian of $\Lie(\G)$; Cf.~\cite[\S1.5]{Bombieri-Gubler}.

The following definition will play a crucial role in this paper.

\begin{defn}\label{def:Diophantine-intro-real} 
Let $\vare : \bbr^+ \to (0,1)$ be a monotone decreasing function, and let $t\in\bbr^+$. Let $\zpz$ be as in~\eqref{def of z in intro}.
A point $g\Gamma$ is called \textbf{$(\vare,t)$-Diophantine} for the action of $U$ if for all $\Hbf \in \hcal$ with $\{e\}\neq\H\neq\G$ 
\be\label{eq:dioph-cond-intro-real}
\norm{\zpz\wedge {\eta_H(g)}}\geq\vare(\norm{\eta_H(g)}) \qquad\text{if}\qquad \norm{\eta_H(g)}< e^t
.\ee 
A point is \textbf{$\vare$-Diophantine} if
it is $(\vare,t)$-Diophantine for all $t>0$.
\end{defn}

\noindent
Note that this is a condition on the \emph{pair} $(U,g\Gamma)$. Unless $U < \H(\R)$ for some (proper) $\H \lhd \Gbf$, the set $\mathcal G(U)$ is nonempty, and moreover any $x\in\mathcal G(U)$ is $\vare$-Diophantine for some $\vare$ as above.
In most interesting examples the singular set $\mathcal S(U)$ is a dense subset of $X$
.
Therefore, $\mathcal G(U)$ is usually a $G_\delta$-set without any interior points.
For any $t \in \bbr^+$, on the other hand, the set of $(\vare,t)$-Diophantine points in Definition~\ref{def:Diophantine-intro-real} is a nice closed set with interior points (indeed, is the closure of its interior points).

We can now state our main theorem (in slightly simplified form, see Theorem~\ref{thm:eff-linearization-general} below for the full version with all the features):
\begin{thm}\label{thm:eff-linearization-general-real}
There are constants $A, D >1$ depending only on $N$, and $E_1>1$ depending on $N$, $G$ and $\Gamma$, so that the following holds. Let $g\in G$, $t>0$, $k\geq 1$, and $0<\eta<1/2$. Assume $\vare:\bbr^+ \to (0,1)$ satisfies 
for any $s >0$ that
\[
\vare(s) \leq  \eta^A s^{-A} /E_1.
\]
Then at least one of the following three possibilities holds.
\begin{enumerate}
\item 
\begin{equation*}
\biggl|\biggl\{\xi \in [-1,1]:
\begin{array}{c} 
u(e^k \xi)g\Gamma\not\in X_\eta\text{ or }\\
 \text{$u(e^k \xi)g\Gamma$ is not $(\vare,t)$-Diophantine}
\end{array}
\biggr\}\biggr|
< E_1\eta^{1/D}
\end{equation*} 

\item There exist a nontrivial proper subgroup $\Hbf\in\Hcal$ of \[\height(\Hbf)\leq E_1 ( |g|^A+ e^{At})\eta^{-A}\] 
so that the following hold for all $\xi \in [-1,1]$:
\begin{align*}
\norm{\eta_{H}(u(e^k\xi)g)}&\leq E_1 ( |g|^A+ e^{At})\,\eta^{-A}\\
\norm{\zpz\wedge{\eta_{H}(u(e^k\xi)g)}}&\leq E_1 e^{-k/D} ( |g|^A+ e^{At})\,\eta^{-A}
\end{align*}
where $\zpz$ is as in~\eqref{def of z in intro}.
\item There exist a nontrivial proper normal subgroup $\Hbf \lhd \Gbf$ of  \[\height(\Hbf)\leq E_1  e^{At}\eta^{-A}\]
so that 
\[
 \Bigl\|\zpz\wedge\vpz_{H}\Bigr\|\leq \vare(\height(\Hbf)^{1/A} \eta/ E_1 )^{1/A}.
\]

\end{enumerate}
\end{thm}


Similar to~\cite{DM}, the proof of Theorem~\ref{thm:eff-linearization-general-real}, and its $\places$-arithmetic analogue, relies on the polynomial like behavior of the unipotent orbits. However, in addition to being {\em polynomially} effective, our results here also differ from~\cite{DM} in the following sense.
They provide a compact subset of $\mathcal G(U)$ which is {\em independent} of the base point 
and to which a unipotent orbit returns unless there is an algebraic obstruction;
this uniformity is used essentially in Theorem~\ref{thm:linearization-noneff} and Theorem~\ref{thm:linearization-noneff-eff}. Regarding nondivergence properties of unipotent orbits, such uniformity is well known and is due to Dani (see~\cite{Dani, DM-Nondiv}), but in this context it is new.

These features have been made possible using two main ingredients. 
First is the use of an effective notion of a generic point, Definition~\ref{def:Diophantine-intro-real}. 
The second ingredient is the use of a group ${\bf M}_{\bf H}$, see \S\ref{sec:LH-MH}, to control the {\em speed} of unipotent orbits in
the representation space $V_H$; this group does not feature in the analysis in~\cite{DM}.

Using Theorem~\ref{thm:eff-linearization-general-real} one can give a topological analogue of a result of Mozes and Shah \cite{Mozes-Shah}. To deal with groups with infinitely many normal $\Q$-subgroups we need the following definition: 

For any $T>0$, put
\[
\sigma(T)=\min\left (\{1\}\cup \biggl\{\|\zpz\wedge\vpz_H \|:\begin{array}{c} \Hbf\in\Hcal, \Hbf\lhd\Gbf,\\  
\height(\H)\leq T, \{1\}\neq\Hbf\neq \G\end{array}\biggr\}\right ).
\]

\begin{thm}\label{thm:linearization-noneff-eff-explicit}\label{thm:linearization-noneff-eff}
There exists some $D>1$ depending on $N$ and $E_1>0$ depending on $N$, $G$, and $\Gamma$ 
so that the following holds. Let $0<\eta<1/2$.

Let $\{x_m\}$ be a sequence of points in $X$, and let $T_m\to\infty$ be a sequence of real numbers. 
For each $m$ let $I_m\subset[-T_m,T_m]$ be a measurable set with measure $>\eta T_m$.  
Let
\[
Y=\bigcap_{k\geq 1} \overline{\bigcup_{m\geq k} \{u(t)x_m: t\in I_m\}}.
\]
Then exactly one of the following holds.
\begin{enumerate}
\item $Y$ contains an $\vare$-Diophantine point for 
\[
\vare(s)=\Bigl(\eta s^{-1} \sigma(E_1^A\eta^{-A}s^A)/2E_1\Bigr)^A.
\]
\item There exists a countable (or finite) collection  
\[
\fcal = \{(\mathbb H_i, L_i): i \in I\} \subset \hcal\times\R^+
\]
so that if
\[
Y_i = \left\{g \in N(U,H_i):\norm{\eta_{H_i}(g)}\leq L_i\right\}\Gamma/\Gamma
\]
then
\begin{enumerate}
    \item $Y \subset \bigcup_{i \in I} Y_i$
\item for any $\beta>0$ 
\[
\#\{i \in I : Y\cap  X_\beta\cap Y_i\neq\emptyset\}<\infty.
\] 
\end{enumerate} 
\end{enumerate}
\end{thm}

\medskip


\noindent
As we shall see in Corollary~\ref{lem:tube-closed} below, for any $\Hbf \in \hcal$ and $L>0$ the set 
\[
Y = \{g \in N(U,H):\norm{\eta_{H}(g)}\leq L\}\Gamma/\Gamma
\]
is a closed (though in general not compact) subset of $X$. For instance, for $G=\SL_2(\bbr)$ and $\Gamma=\SL_2(\bbz)$, if we take $\Hbf$ to be the stabilizer of the vector $\biggl(\begin{matrix} 0\\[-2pt]1 \end{matrix}\biggr) \in \Z^2$, \ $U=H=\Hbf(\bbr)$ and define $Y$ as above then $Y$ is the union of all periodic $U$-orbits of period $\leq L$.

\medskip

Theorem~\ref{thm:linearization-noneff-eff-explicit} is related to~\cite[Thm.~4]{DM}.
Specifically in that paper it is proved that if one assumes that sequence $\{x_m\}$ converges to a point in $\mathcal G(U)$ then a less precise form of (1) of Theorem~\ref{thm:linearization-noneff-eff-explicit}, namely that $Y$ contains a point in $\mathcal G(U)$, holds.

\subsection{Friendly measures}\label{sec:friendly-intro}

In this section we discuss generalizations of Theorem~\ref{thm:linearization-noneff-eff-explicit} 
to the class of friendly measures which were studied in~\cite{KLW}.

Let $(Y,d)$ be a $\sigma$-compact metric space; for every $y\in Y$ and $r>0$, let $B(y,r)$ denote the open ball of radius $r$ centered at $y$.
Let $\mu$ be a locally finite Borel measure on $Y$.
If $A>0$ and $O\subset Y$ is an open subset, the measure $\mu$ is called $A$-{\em Federer} on $O$ 
if for all $y\in{\rm supp}(\mu)\cap O$ one has
\[
\frac{\mu(B(y,3r))}{\mu(B(y,r))}<A
\]
whenever $B(y,3r)\subset O$. 


Let $Y=\bbr$ be equipped with the standard metric.
Given a point $a\in \bbr$ and $\delta>0$, we let $I_\delta(a)=(a-\delta,a+\delta)$.
Given $c,\alpha>0$ and an open subset $O\subset\bbr$, we say $\mu$ is $(c,\alpha)$-{\em absolutely decaying} on $O$ if for every non-empty open
interval $J\subset O$ centered in ${\rm supp}(\mu)$, every point $a\in \bbr$, and every $\delta>0$ we have 
\be\label{eq:mu-decaying}
\mu(J\cap I_\delta(a))< c \Bigl(\frac{\delta}{r}\Bigr)^\alpha\mu(J)
\ee  
where $J$ has length $2r$, see~\cite[Lemma 2.2]{KLW}.

%


We will say a measure $\mu$ on $\bbr$ is {\em uniformly friendly}
if $\mu$ is  $A$-Federer and $(c,\alpha)$-absolutely decaying for some $A,c,\alpha>0$.

Let the notation be as in \S\ref{sec:eff-linearization}; in particular,
\[
U=\{u(t)=\exp(t\zpz): t\in\R\}.
\]
for some nilpotent element $\zpz\in\mathfrak{g}$ with $\|\zpz\|=1$.

\begin{thm}\label{thm:linearization-noneff-eff-friendly}
Let $\mu$ be a uniformly friendly measure on $\R$.
There exists some $D>1$ depending on $N$ and $\mu$, and $E_1>0$ depending on $N$, $G$, $\Gamma$, and $\mu$ so that the following holds.

Let $0<\eta< 1/2$.
Let $\{x_m\}$ be a sequence 
of points in $X$, \ $0<\eta<1/2$, and let $k_m\to\infty$ be a sequence of real numbers. 
For each $m$ let $I_m\subset[-1,1]$ be a measurable set with 
\(
\mu(I_m)>\eta\,\mu([-1,1]).
\)  
Let
\[
Y=\bigcap_{\ell\geq 1} \overline{\bigcup_{m\geq \ell} \{u(e^{k_m}t)x_m: t\in I_m\}}.
\]
Then exactly one of the following holds.
\begin{enumerate}
\item $Y$ contains an $\vare$-Diophantine point for 
\[
\vare(s)=\Bigl(\eta s^{-1} \sigma(E_1^A\eta^{-A}s^A)/2E_1\Bigr)^A.
\]
\item There exists a countable (or finite) collection  
\[
\fcal = \{(\mathcal H_i, L_i): i \in I\} \subset \hcal\times\R^+
\]
so that if
\[
Y_i = \{g \in N(U,H):\norm{\eta_{H_i}(g)}\leq L_i\}\Gamma/\Gamma,
\]
then
\begin{enumerate}
    \item $Y \subset \bigcup_{i \in I} Y_i$
\item for any $\beta>0$ 
\[
\#\{i \in I : Y\cap  X_\beta\cap Y_i\neq\emptyset\}<\infty.
\] 
\end{enumerate}
\end{enumerate}
\end{thm}

See \S\ref{sec:friendly} for a more detailed discussion of this generalization.

\subsection*{Acknowledgements}
We would like to thank M.~Einsiedler, H.~Oh, and A.~Wieser for their helpful comments on earlier drafts of this paper.

\section{Notation}\label{s;nottaion}\label{sec:notation}

\subsection{}\label{sec:group}
Let $\qpl=\{\infty\}\cup\{p:p\text{ is a prime}\}$ denote the set of places of $\Q$.
We let $\qpl_f=\qpl\setminus\{\infty\}$ denote the set of finite places in $\qpl$.
For every $v\in\qpl$ let $\bbq_v$ be the completion of $\bbq$
at $v$; we often write $\R$ for $\bbq_\infty$. 

For every $p\in\qpl_f$, we let $\bbc_p$ be the completion of the algebraic closure, $\overline{\bbq}_p$, 
of $\bbq_p$ with respect to the $p$-adic norm. The field $\bbc_p$ is a complete and algebraically closed field.

Given a finite subset $\places\subset\qpl$,  
we put $\places_f:=\places\setminus\{\infty\}$; 
also set $\bbq_\places=\prod_{v\in\places}\bbq_v$. 
Given an element $r\in\Q_\places$, we put $|r|_\places=\max_{v\in\places}|r|_v$.

For any $p\in\qpl_f$, let $\bbz_p$ 
denote the ring of $p$-adic integers in $\bbq_p.$ 
The ring of $\places$-integers in $\bbq$ is as usual denoted by $\bbz_\places.$ 

\medskip

Given a $\Q$-variety $\bf Y$, we put $Y_v={\bf Y}(\Q_v)$; given a finite subset $\places\subset\qpl$, 
we also write $Y_\places$, or simply $Y$ if there is not confusion, for $\prod_\places Y_v$.  

For any $\Q$-variety $\mathbf{Y}$, we denote by $\dim\mathbf Y$ the dimension in the algebro geometric sense. 
In particular $\dim\G$ is the dimension of $\G$ as an algebraic group. Note that
$\dim_{\bbq_v} Y_v,$ the dimension of $Y_v$ as a $\bbq_v$-manifold, equals $\dim{\bf Y},$ see e.g.~\cite[Ch.~I, \S2.5]{Margulis-Book}. 
We also put
\[
\dim Y_\places:=\textstyle\sum_{\places} \dim_{\bbq_p} Y_p=\Bigl(\#\places\Bigr)\dim{\bf Y}.
\] 

Given a $\bbq$-group, $\H$, we denote by $\Lie(\H)$
the Lie algebra of $\H$. We will use lower case gothic letters to denote the Lie algebra of $\H$ 
over various local fields, e.g., $\hfrak_p=\Lie\bigl(\H(\bbq_p)\bigr)$;
similarly, we write $\hfrak_\places$, or simply $\hfrak$, for $\oplus_{\places}\hfrak_p$.

The space $\hfrak_\places$ is a $\Q_\places$-module; and the notation $r\wpz$ 
for $r\in\Q_\places$ and $\wpz\in\hfrak_\places$ in the sequel refers to this module structure.

Given a natural number $m\leq \dim(\Lie(\H))$, we write $\wedge^m\hfrak$ or 
$\wedge^m\hfrak_\places$ to denote $\oplus_{\places}(\wedge^m\hfrak_v)$.

\medskip

For any (compact) subset $\mathsf K\subset \bbq_v^m$ and any $\delta>0$, 
we let $\ncal_\delta(\mathsf K )$ denote the $\delta$-neighborhood of $\mathsf K.$ 
Also let $|\mathsf K|$ denote the Haar measure of $\mathsf K.$ 

Let $\mathbf H$ be a $\bbq_\places$-group and put $H=\mathbf H(\bbq_\places).$
Given a subset $B\subset H$, we define 
\[
Z_H(B)=\{g \in H: gb=bg\;\text{for all}\;b\in B\}.
\]
Given two subsets $B_1,B_2\subset H$, we define 
\[
N_H(B_1,B_2)=\Bigl\{g\in H: g^{-1}B_1g\subset B_2\Bigr\},
\]
and put $N_H(B):=N_H(B,B)$ for any $B\subset H.$

\subsection{}\label{sec:norms}
For any place $v\in\qpl$, let $\|\;\|_{v}$ denote the max norm, with respect to the standard basis, 
on $\sl_N(\bbq_v)$  and on $\wedge \sl_N(\bbq_v)$. 
Given a finite subset $\places\subset\qpl$, the norm $\|\;\|_\places$ (or simply $\|\;\|$) is defined by 
\[
\|\;\|_\places=\max_{v\in\places}\|\;\|_{v}.
\]

We let $\vsdist$ denote the induced metric on the exterior algebra $\wedge\sl_N(\bbq_\Sigma)$  induced from $\|\;\|.$

We also fix norms, which we continue to denote by $\|\;\|_v$, on $\SL_N(\bbq_v)$ for all $v$; 
and put $\|\;\|=\max_{\places}\|\|_v$. 
For every $g\in\SL_N(\bbq_\places)$, in particular for every $g\in G$, we set 
\be\label{eq:def-|g|}
|g|:=\max\{\|g\|,\|g^{-1}\|\}.
\ee 
Note that
\be\label{e;norm-like-est}
\text{$|g|=|g^{-1}|\;$ and $\;|g_1g_2|\ll|g_1||g_2|,$}
\ee

Fix $\consta\label{k:adjoint}$ and $\consta\label{k:adjoint-mult}$ both depending only on $N$ so that
\be\label{eq:def-k-adjoint}
\|\wedge^r\Ad (g) \zpz\|\leq \ref{k:adjoint-mult} |g|^{\ref{k:adjoint}}\|\zpz\|\quad\text{for all $\zpz\in\sl_N(\bbq_\places)$ and $1\leq r\leq N^2.$}
\ee 

\subsection{}\label{ss;height}
Let ${\bf W}\subset\sl_N$ be a rational subspace,
then $\wedge^{\dim{\bf W}}{\bf W}$ defines a rational line in $\wedge^{\dim\Wbf}\sl_N.$
This line is diagonally embedded in $\wedge^{\dim{\bf W}}\sl_N(\bbq_\Sigma)$, and we do not 
distinguish between this diagonal embedding and the line.

Fix a $\bbz$-basis for ${\bf W}(\R)\cap\sl_N(\Z)$.
Let $\vbf_{\bf W}$ denote the corresponding primitive integral vector on $\wedge^{\dim{\bf W}}{\bf W}$, and  
define
\[
\height(\Wbf )=\|\vbf_\Wbf\|.
\]
Note that we used the max norm in the above definition, in particular, we have: $\height(\Wbf)$
is an integer.

Alternatively, $\height(\Wbf)$ may be defined as follows. Let $\{{\bf e}_1,\ldots,{\bf e}_{\dim\Wbf}\}$
be a $\bbq$-basis for $\Wbf.$ Then
\[
\height(\Wbf)=\prod_{v\in\qpl}\|{\bf e}_1\wedge\cdots\wedge{\bf e}_{\dim\Wbf}\|_{v}
\]
In view of the product formula, the above is independent of our choice of the rational basis for $\Wbf$, see~\cite[\S1.5]{Bombieri-Gubler}.  

Given a $\bbq$-subgroup $\Hbf$ of $\SL_N$, we put $\vbf_\H:=\vbf_{\Lie(\H)}$ and define
\be\label{e;def-height}
\height(\H):=\height\Bigl(\Lie(\H)\Bigr)=\|\vbf_\H\|.
\ee

\subsection{}\label{sec:notation-exponent}
For the rest of this paper, fix a finite subset $\places\subset\qpl$ containing $\infty$. 

The exponents in this paper are denoted by $A$ with numerical indecies. 
These constants depend only on $N$. 
The understanding is that $A_{\cdot}>1$. 

Similarly, the constants $C, D,$ and $F$ in the sequel depend only on $N$, and are implicitly assumed to be $>1$.  

We use the notation $T\ll R$ to denote $T\leq c R$ where the multiplicative constant $c$ is allowed to depend on $N$, the number of places $\#\places$, and polynomially on the finite primes in $\places$ and on $\height(G)$. Similarly we define $T\gg R$.
It will also be convenient to use $\star$ to denote a constant. More precisely,  
we write $T\ll R^\star$ if $T\leq cR^{A}$ or $T\leq cR^\alpha$ where $c$ is allowed to depend on $N$, $\#\places$, polynomially on the finite primes in $\places$ and $\height(\Gbf)$, and the exponent is either a ``big enough'' constant or a ``small enough'' constant depending only on $N$; hopefully the context will make it clear if the exponent needs to be large or small.

\subsection{}\label{sec:norms-gfrak}
For all $v\in\places$, let $\|\;\|_v$ denote the max norm with respect to the standard basis on $\bbq_v^m$; we put
$\|\zpz\|=\|\zpz\|_\places=\max\|\zpz_v\|_v$ for all $\zpz=(\zpz_v)\in\bbq_\places^m$.

Define
\be\label{def:cFun-intro}
\cfun(\zpz)=\prod_{v\in\places}\|\zpz\|_v\text{ for all $\zpz\in\bbq_\places^m$}.
\ee
Note that $\cfun(r\zpz)=\cfun(\zpz)$ for all $r\in\bbz_\places^\times$ and all $\zpz\in\bbq_\places^m$.

\begin{lemma}\label{lem:c-fun-norm}
There exists $\consta\label{k:exp-cfun}$ and some $C_{m,\places}\ge 1$ so that the following holds.   
Let $\zpz\in\bbq_\places^m$ be a vector so that $\cfun(\zpz)\neq0$. 
\begin{enumerate}
\item There exists some $r_0\in\bbz_\places^\times$ so that \[
C_{m,\places}^{-1}\|r_0\zpz\|_\places\leq \|r_0\zpz\|_v\leq C_{m,\places}\|r_0\zpz\|_\places
\]
for all $v\in\places$, in particular, we have 
\be\label{eq:norm-prod}
\min_{r\in\bbz_\places^\times}\|r\zpz\|_\places\leq C_{m,\places}\|r_0\zpz\|_\places\leq C_{m,\places}\cfun(r_0\zpz)^{1/{\#\places}}.
\ee
\item Let $\|\zpz\|_\places=1$, and let $T>0$. Then 
\be\label{eq:number-ZS-log}
\#\{r\in\bbz_\places^\times: \|r\zpz\|_\places\leq T\}\leq C_{m,\places}\Bigl(\log\frac{T}{\cfun(\zpz)}\Bigr)^{\ref{k:exp-cfun}}. 
\ee
\end{enumerate}
\end{lemma}

\begin{proof}
The claim in part (1) is proved in~\cite[Lemma 8.6]{KT:Nondiv}.

We now turn to the proof of part (2). Let $\ell=\#\places$ and for every $\apz>0$ put 
\[
\mathcal E_\apz=\Bigl\{(w_1,\ldots,w_{\ell})\in\bbr_+^{\ell}:\textstyle\prod w_i=\apz\Bigr\}.
\]
Note that $\mathcal E_\apz$ is invariant under multiplication by positive diagonal matrices in $\SL_\ell(\bbr)$.

Let $D_\places$ denote the group of positive diagonal matrices in $\SL_\ell(\bbq)$ whose entries are in $\zeds$.
Let $\apz=\cfun(\zpz)$. Then $(|\zpz|_v)\in\mathcal E_\apz$, and for every $r\in\bbz_\places^\times$ we have ${\rm Diag}(|r|_v)\in D_\places$.

Let $\|\;\|_{\rm op}$ denote the operator norm on $\SL_\ell(\bbr)$ and let $\|\;\|_{\rm m}$ denote the
max norm on $\bbr^\ell$. We have
\be\label{eq:number-diag-log}
\#\{A\in D_\places: \|A\|_{\rm op}\leq S\}\ll_{m,\places} (\log S)^\star.
\ee
Further, if $\wpz=(w_1,\ldots,w_\ell)\in \mathcal E_\apz$ is so that $|w_i|\ll_{m,\places}\|\wpz\|_{\rm m}\ll_{m,\places} |w_i|$
for all $i$, then $\|A\|_{\rm op}\|\wpz\|_{\rm m}\ll_{m,\places}\|\wpz A\|_{\rm m}\ll_{m,\places} \|A\|_{\rm op}\|\wpz\|_{\rm m}$.

Hence, the claim follows from~\eqref{eq:number-diag-log} 
if we replace $\zpz$ by $r_0\zpz$ so that 
\[
\|r_0\zpz\|_\places\ll_{m,\places}\|r_0\zpz\|_v\ll_{m,\places}\|r_0\zpz\|_\places
\]
for all $v\in\places$.
\end{proof}

Similar to~\eqref{def:cFun-intro}, we define $\cfun(\wpz)=\prod\|\wpz\|_{v}$ for all $\wpz\in\wedge\sl_N(\bbq_\places)$.

\subsection{}\label{sec:space-lattices}
Let $\G$ be a connected $\Q$-group of class $\Hcal$. 
Fix an embedding $\iota:\Gbf\rightarrow{\SL}_N$ defined over $\bbq$. Put $G=G_\places$ and $\gfrak=\gfrak_\places$.

We identify $\G$ with $\iota(\G)\subset\SL_N$, hence, $G\subset\SL_N(\bbq_\places)$.
Let $\gfrak(\bbz_\places):=\gfrak\cap\sl_N(\bbz_\places)$, then $[\gfrak(\bbz_\places),\gfrak(\bbz_\places)] \subset\gfrak(\bbz_\places)$.

\medskip

We fix a $\Z$-basis, $\mathcal B_\G=\{\zpz_1,\ldots,\zpz_d\}$, for $\gfrak\cap \SL_N(\Z)$ 
so that $\cfun(\zpz_i)\dotll\height(\G)^\star$ for all $1\leq i\leq d$. 
Using this basis, we identify $\Lie(\G)$ with a $d$-dimensional vector space with a $\Q$-structure.
We also identify the $\Z$-span of $\mathcal B_\G$ with $\Z^d$;
hence we get a representation $\Ad:\G\to\SL_d$.  

Let $\Gamma\subset G \cap \SL_N(\zeds)$ be a lattice. Then $\Gamma$ fixes $\gfrak(\bbz_\places)$, which implies that $\Ad(\Gamma)\subset\SL_d(\bbz_\places)$ --- recall that we are using $\mathcal B_G$ to define $\Ad$ over the ring~$\bbz_\places$ (and not just as a $\bbq$-representation).

\medskip

Let $X:=G/\Gamma$. For any $\eta>0$, set 
\be\label{eq:def-XR}
X_\eta:=\Bigl\{g\Gamma\in X: \min_{0\neq v\in\gfrak(\bbz_\places)}\cfun(gv)\geq \eta\Bigr\}
\ee
where here and in what follows we often simply write $gv$ for $\Ad(g)v$; similarly, for $\wpz\in\wedge\gfrak$,
we simply write $g\wpz$ to denote the corresponding wedge power of the adjoint representation.

For any $\eta > 0$, the set $X _ \eta$ is a compact subset of $G / \Gamma$, and $G / \Gamma = \bigcup_ {\eta > 0} X _ \eta$. 

We will need a quantitative version of the former statement:

\begin{lemma}\label{prop:eff-red-theory}
There exist some $E_\Gbf$ (depending on the geometry of $G/\Gamma$) and $F$ (depending only on $N$) so that the following holds. 
Let $g\in G$ be so that $g\Gamma\in X_\eta$. There exists some $\gamma\in\Gamma$ so that 
\[
|g\gamma|\leq E_\Gbf \eta^{-F}
\]
\end{lemma}
\subsection*{Remark} For a point $g \Gamma$ to be $X _ \eta$ essentially means that the local injectivity radius for $G / \Gamma$ at $g \Gamma$ is $\gg \eta ^{\star}$. Thus Lemma~\ref{prop:eff-red-theory} can be viewed as an estimate of the diameter of the part of $G / \Gamma$ which has injectivity radius greater than $\eta ^{\star}$. In particular, for $G / \Gamma$ compact, $F$ is essentially meaningless, and $E _ {\mathbf G}$ is the diameter of the smallest norm ball in $G$ needed to cover $G / \Gamma$, see~\cite[Thm.~1.7 and Thm.~6.9]{MoSa-Eff-Red} for more explicit estimates. 

For future convenience, we set $\tilde E_\Gbf = \height(\Gbf ) \cdot E_\Gbf$. In the proof of this lemma implicit constants are allowed to depend on $G$ and $\Gamma$ where indicated.

\begin{proof}
By \cite[Prop.\ 3.1]{MoSa-Eff-Red} there exists a Levi subgroup $\bf L$
so that $\height(\bf L)$ is bounded by $\height(\G)^\star$, in particular, $L\cap\Gamma$ is a lattice\footnote{Note that without the estimate on the the height, the existence of a Levi subgroup so that $L\cap\Gamma$ is a lattice in $L$ is a theorem of Mostow~\cite{Most-Levi}.} in $L$. 

For any $g\in G$, define
\[
\alpha _ {\mathfrak g}(g):=\max \{ \cfun(\Ad(g)\zpz)^{-1}:0 \neq \zpz \in\gfrak(\zeds)\};
\]
define similarly $\alpha_{\mathfrak l}$ for any $g\in L = \mathbf L(\bbq_\Sigma)$.

Let $g\in G$ and write $g=g^0g^u$ where $g^0\in L$ and $g^u\in\rad_u(G)$.
Then $\alpha _ {\mathfrak l}(g^0)\leq \alpha _ {\mathfrak g}(g^0)$ and by~\cite[Lemmas~4.4 and~6.8]{MoSa-Eff-Red} we have
$
\alpha _ {\mathfrak g}(g^0)\ll_{\Gbf,\Gamma}\alpha _ {\mathfrak g}(g)^\star.
$
It follows from reduction theory for $L$,
(\cite[Thm.~4.8]{PR} and 
\cite[Thm 4.17]{PR}) that there exists some $\gamma_0\in L\cap\Gamma$ so that
\[
|g^0\gamma_0|\ll_{\Gbf,\Gamma} \alpha_\lfrak(g^0)^\star,
\]
and combining the above we get $|g^0\gamma_0|\ll \alpha_\gfrak(g)^\star$.
Moreover, $\gamma_0^{-1}g^u\gamma_0\in\rad_u(G)$ and $\height(\rad_u(\G))\ll1$, see Lemma~\ref{lem:ht-rad-u-rad}.
Therefore, there exists some $\gamma_1\in \rad_u(G)\cap\Gamma$ so that
\[
|\gamma_0^{-1}g^u\gamma_0\gamma_1|\ll_{\Gbf,\Gamma}1,
\]
see, e.g.,~\cite[Lemma 5.6]{MoSa-Eff-Red}.

Put $\gamma=\gamma_0\gamma_1$. Then 
\begin{align*}
|g\gamma|&=|g^0g^u\gamma_0\gamma_1|\\
&\dotll_{\Gbf,\Gamma} |g^0\gamma_0||\gamma_0^{-1}g^u\gamma_0\gamma_1|\ll_{\Gbf,\Gamma} \alpha_\gfrak (g)^\star;
\end{align*}
as was claimed.
\end{proof}

\subsection{}\label{sec:lambda-n}
Let $U=\prod_{v\in\places}U_v\subset G$ where for all $v\in\places$ we have $U_v$ is a (possibly trivial) unipotent subgroup of $G$. We will refer to such groups as a $\bbq_\places$-unipotent subgroup of $G$. 
Define
\be\label{eq:def-Sigma'}
\places'=\{v\in\places: U_v\neq\{e\}\}.
\ee

Let $\ufrak$ (resp.\ $\ufrak_v$) denote the Lie algebra of $U$ (resp.\ $U_v$).
The exponential map defines an isomorphism of $\bbq_\places$-varieties from $\ufrak$ onto $U.$
We fix once and for all a basis $\mathcal B_U$ for $\ufrak$ consisting of elements which are nontrivial at only one place. 

For $\delta>0$, let $\mathsf B_\ufrak(0,\delta)=\Bigl\{\sum_{\zpz\in\mathcal B_U} r_\zpz\zpz: |r_\zpz|_\places\leq \delta\Bigr\}$
where $r_\zpz\in\Q_\places$ and $|r|_\places:=\max_{\places}|r_v|_v$; put 
\[
\mathsf B_U(e):=\exp\Bigl(\mathsf B_\ufrak(0,1)\Bigr).
\]
Thus $\mathsf B_U(e)$ is a product of neighborhoods of $1$ in $U_v$ for $v \in \Sigma$.
A subset $\mathsf B\subset \mathsf B_U(e)$ will be called a {\em ball} if it is the image, under
the exponential map, of a norm ball in $\ufrak.$

%

Let $\la:\ufrak\to\ufrak$ be a $\bbq_\places$-diagonalizable expanding linear map, and for all $k\in\bbn$ let $\la_k:\ufrak\to\ufrak$ denote the $k$-fold composition of $\la$ with itself, i.e., $\la_k=\la\circ\cdots\circ\la$, $k$-times. We will throughout make the assumption that for some fixed $\kappa>0$ and all $k \geq 1$
\be \label{assumption on Greek lambda}
\exp\Bigl(\lambda _ {k-\kappa} (\mathsf B_\ufrak(0,1))\Bigr) \cdot \exp\Bigl(\lambda _ {k-1} (\mathsf B_\ufrak(0,1))\Bigr) \subset \exp\Bigl(\lambda _ k (\mathsf B_\ufrak(0,1))\Bigr)
.\ee

We now explicate two examples of $\lambda$ which satisfy the required conditions. One may take $\lambda$ to be an expanding automorphism of the Lie algebra $\ufrak$ as Margulis and Tomanov did in \cite{MT}; more explicitly, we may embed $\mathbf G$ in a larger group in which one can find an element $h$ so that $\lambda=\Ad(h)$ expands~$\ufrak$. 

The following is an alternative construction for a $\lambda$ which satisfies the required assumptions:
Let $v\in\places$, and consider the lower central series for $\ufrak_v$. That is:
\[
\ufrak_v=\ufrak_{v,0}\supset\ufrak_{v,1}\supset\cdots\supset\ufrak_{v,n_v}=\{0\}
\]
where $\ufrak_{v,i+1}=[\ufrak_v,\ufrak_{v,i}]$ for all $0\leq i<n_v$.

For each $i$, let $\ufrak_{v}^i$ denote an orthogonal complement\footnote{Recall that for a finite prime $p$, a set of unit vectors in $\bbq_p^m$ is called orthonormal if it can be extended to a $\bbz_p$-basis for $\bbz_p^m$.} 
of $\ufrak_{v,i+1}$ in $\ufrak_{v,i}$. In particular, we have
\[
\ufrak_{v,i}=\ufrak_{v,i+1}\oplus\ufrak_{v}^i.
\]

Fixing an orthonormal basis of $\ufrak_{v}^i$ for all $0\leq i<n_v$, we obtain an orthonormal 
basis of $\ufrak_v$. 
Let $a_p=p^{-3}$ if $p$ is a finite prime and $a_\infty=e^3$. 
For each $k\in\bbn$, define $\lambda_k:\ufrak_v\to\ufrak_v$ by
\[
\la_k(\zpz)=a_v^{(i+1)k}\zpz\;\;\text{ for all $\zpz\in\ufrak_v^i$}.
\]
We leave the verification that this example does indeed satisfy~\eqref{assumption on Greek lambda} to the reader.

Abusing the notation, for an element $u=\exp(\zpz)\in U$ we set $\lambda(u):=\exp(\lambda(\zpz))$; that is: $\lambda$ and $\lambda_k$ are also considered as function on $U$.

In this paper, we assume that a linear expanding map $\lambda$ satisfying \eqref{assumption on Greek lambda} is fixed; moreover, we assume that the parameters $\kappa$, $\frac {\absolute {\lambda _ 1 (\mathsf B _ U (e))} }{ \absolute {\mathsf B _ U (e))}}$, etc.\ depend only on $N$ and polynomially on $\height (\G)$. E.g., the examples above satisfy these properties, and the reader may take $\lambda$ to be one of these examples. 

However, if for some reason, the reader is keen on taking some particularly wild expanding linear map $\lambda$ satisfying \eqref{assumption on Greek lambda}, the only adverse effect would be that the implicit multiplicative constants need to be allowed to depend polynomially on the parameters of $\lambda$.


\section{Statements of the main theorems: $\places$-arithmetic}\label{sec:S-arith}
Let $\Gbf\subset \SL_N$ be a $\bbq$-group.
Recall the family 
\[
\mathcal{H}= \Bigl\{\Hbf\subset\G:\text{$\H$ is a connected $\bbq$-subgroup and $\rad(\H)=\rad_u(\H)$}\Bigr\}, 
\]
where $\rad(\H)$ (resp.\ $\rad_u(\H)$) denotes the solvable (resp.\ unipotent) radical of $\H.$
We always assume that $\G\in\hcal$.
Recall also our notation $G_v=\Gbf(\bbq_v)$ for all $v\in\places,$ and 
$G=\prod_{v\in\places} G_v$.

Let $\Hbf\in\Hcal$ be a proper subgroup and put  
\[
\text{$\rho_H:=\oplus_{\places}(\wedge^{\dim\H}\Ad)\quad$ and 
$\quad V_H:=\wedge^{\dim\H}\gfrak=\oplus_{\places}(\wedge^{\dim\H}\gfrak_v).$}
\] 
We shall identify between the the representation $\rho_H$ of $G$ and the $\bbq$-representaiton $ \wedge^{\dim\H}\Ad$ of $\G$.
 
Let ${\vbf_\H}$ be a primitive integral vector in $\wedge^{\dim\H}\Lie(\G)$ 
corresponding to the Lie algebra of $\H.$ 
Recall from~\eqref{e;def-height} that
\be\label{eq:def-height-again}
\height(\Hbf )=\|\vbf_\H\|_{\places}=\|\vbf_\H\|.
\ee

The vector $\vbf_\H$ is diagonally embedded in $V_H$ (which is a product of local factors); 
in order to put an emphasis on the local nature of this diagonally embedded vector, we will denote it by $\vpz_H$. 
Define 
\[
\eta_H(g):=\rho_H(g)\vpz_H\;\text{ for every $g\in G$}.
\]


Throughout, $U=\prod_{v\in\places} U_v\subset G$ is a $\bbq_\places$-unipotent subgroup. 
We will use the notation from~\S\ref{sec:lambda-n}. In particular, $\mathcal B_U$ is an orthonormal 
basis for $\ufrak$, and 
\[
\mathsf B_U(e)=\exp\Bigl(\Bigl\{\textstyle\sum_{\zpz\in\mathcal B_U} r_\zpz\zpz: |r_\zpz|_\places\leq 1\Bigr\}\Bigr).
\]
Recall also from~\S\ref{sec:lambda-n} the notion of an admissible expanding map $\la_k:U\to U$ for all $k\in\bbn$. 
 
 \medskip

The following is a $\places$-arithmetic version of Definition~\ref{def:Diophantine-intro-real}, and plays a crucial role in this paper.



\begin{defn}\label{def:Diophantine-intro} 
Let $\vare : \bbr^+ \to (0,1)$ be a monotone decreasing function, $t\in\bbr^+$, and $\fcal\subset \mathcal H$ a subcollection that is $\Gamma$-invariant with respect to conjugation.
A point $g\Gamma$ is called \textbf{$(\vare,t,\mathcal F)$-Diophantine} for the action of $U$ if for all $\Hbf \in \fcal$ with $\{e\}\neq\H\neq\G$ and $\cfun(\eta_H(g))< e^t$
\be\label{eq:dioph-cond-intro-2}
\max_{\zpz\in\bcal_U}\|\zpz\wedge {\eta_H(g)}\|\geq\vare(\cfun(\eta_H(g)))
.\ee 
A point is \textbf{$(\vare,t)$-Diophantine} if it is $(\vare,t,\hcal)$-Diophantine. 
A point is \textbf{$\vare$-Diophantine} if
it is $(\vare,t)$-Diophantine for all $t>0$.
\end{defn}  
\medskip

\noindent
Note that if there exists some nontrivial $\Hbf \lhd \Gbf$ so that $U\subset \Hbf(\bbq_\Sigma)$, then for \emph{any} $\vare : \bbr^+ \to (0,1)$ the set of $\vare$-Diophantine points is empty.

We now state the main result of this paper.

\begin{thm}\label{thm:eff-linearization-general}
There exist constants $A$ and $D$ depending only on $N$, and constants $E$ depending on $N, \#\Sigma$ and polynomially on $\height(\mathbf G)$ and the primes in $\Sigma$, and $E_1$ depending in addition also (polynomially) on $E_\Gbf$, so that the following holds. Let $g\in G$, $t>0$, $k\geq 1$, and $0<\eta<1/2$. Assume $\vare:\bbr^+ \to (0,1)$ satisfies 
for any $s \in \bbr^+$ that
\begin{equation}
\label{condition on vare}
\vare(s) \leq  \eta^A s^{-A} /E_1.
\end{equation}
Then at least one of the following three possibilities holds.
\begin{enumerate}
\item 
\begin{equation*}
\biggl|\biggl\{u\in\mathsf B_U(e):\begin{array}{c}\la_k(u)g\Gamma\not\in X_\eta\text{ or }\\ \text{$\la_k(u)g\Gamma$ is not $(\vare,t)$-Diophantine}\end{array}\biggr\}\biggr|
< E_1 \eta^{1/D}
\end{equation*} 

\item There exist a nontrivial proper subgroup $\Hbf\in\Hcal$ with \[\height(\Hbf)\leq ( E|g|^A+ E_1 e^{At})\eta^{-A}\] 
so that the following hold for all $u\in\mathsf B_U(e)$:
\begin{align*}
\cfun(\eta_{H}(\la_k(u)g))&\leq  ( E|g|^A+ E_1 e^{At})\eta^{-A}\\
\max_{\zpz\in\mathcal B_U} \Bigl\|\zpz\wedge{\eta_{H}(\la_{k}(u)g)}\Bigr\|&\leq  e^{-k/D}( E|g|^A+ E_1 e^{At})\eta^{-A}
\end{align*}

\item There exist a nontrivial proper normal subgroup $\Hbf \lhd \Gbf$ with  
\[
\height(\Hbf)\leq E_1 ( e^{t}\eta^{-1})^A
\]
so that 
\[
\max_{\zpz\in\mathcal B_U} \Bigl\|\zpz\wedge\vpz_{H}\Bigr\|\leq  \vare(\height(\Hbf)^{1/A} \eta/ E_1 )^{1/A}.
\]

\end{enumerate}
\end{thm}

\medskip

\noindent
Of course if $\mathbf G$ is $\bbq$-simple, possibility (3) cannot hold. A typical example where there are infinitely many normal subgroups is 
\[\mathbf G = \SL _ n \ltimes \left ((\gadd)^n\times (\gadd)^n\right )\]
 with $\gadd$ denoting the the one dimensional additive group (the simplest possible algebraic group!). The group $\mathbf G$ is a perfect group, and for any $l,k \in \bbz$ the subgroup
\begin{equation*}
\mathbf H _ {l,k} = \left\{ (g, l \vpz, k \vpz):  g \in \SL _ n, \vpz \in \gadd^n \right\}
\end{equation*}
is a normal subgroup of $\mathbf G$.

We note the following interesting corollary of Theorem~\ref{thm:eff-linearization-general}. For simplicity we state it in the case where $\Gbf$ has only finitely many normal $\bbq$-subgroups (it is fairly easy to adjust the statement and the proof to accommodate general $\Gbf$, but they become a bit messier). Results of similar flavour were given by Lindenstrauss and Margulis in \cite[Prop.~4.4]{LM-EffOpp}.

\begin{coro}\label{inheritance corollary}
Let $\Gbf, G, \Gamma, U$ be as above, with $\Gbf$ having only finitely many normal $\bbq$-subgroups, and $U \not\subset \Hbf (\bbq_\Sigma)$ for all $\Hbf \lhd \Gbf$.
There are $\consta \label{inheritance constant}, \consta \label{more interesting inheritance constant} $ depending only on $N$ and $\epsilon _ 1, \epsilon_2$ depending on $N, \#\Sigma$ and polynomially on $\height(\mathbf G)$, the primes in $\Sigma$, and $E_\Gbf$, and $t_0$ that depend in addition also on $U$ and how far it is from lying in any  $\Hbf \lhd \Gbf$, so that if $\vare(s)=\epsilon_1 \eta^{\ref{inheritance constant}}s^{-\ref{inheritance constant}} $ then if $t>t_0$, and if $t' , k \geq \ref{more interesting inheritance constant} (t+\log(1/\eta)+\log(1/\epsilon_2))$, then for any 
$(\vare, t')$-Diophantine $g \Gamma \in X_\eta$,  
\begin{equation*}
\biggl|\biggl\{u\in\mathsf B_U(e):\begin{array}{c}\la_k(u)g\Gamma\not\in X_\eta\text{ or }\\ \text{$\la_k(u)g\Gamma$ is not $(\vare,t)$-Diophantine}\end{array}\biggr\}\biggr|
< E_1 \eta^{1/D}
\end{equation*}
\end{coro}

\begin{proof} [Proof of Corollary~\ref{inheritance corollary} assuming Theorem~\ref{thm:eff-linearization-general}]

Assuming that the constants in Corollary~\ref{inheritance corollary} were appropriately chosen, $\vare$ satisfies \eqref{condition on vare} and we may apply Theorem~\ref{thm:eff-linearization-general}. 

If (1) of that theorem holds there is nothing to prove. Otherwise either (2) or (3) of that theorem holds. (3) is ruled out by our assumption that $U \not\subset \Hbf (\bbq_\Sigma)$ for all $\Hbf \lhd \Gbf$ if $t_0$ is large enough.

Suppose then we are in case (2). As $g \Gamma \in X _ \eta$ it follows that there is a nontrivial subgroup $\mathbf H \in \mathcal{H}$ for which in particular\footnote{Possibly for a slightly larger $A$ than in the theorem.}
\begin{align}
\cfun(\eta_{H}(g))&\leq  2E_1 e^{At}\eta^{-A}\nonumber\\
\max_{\zpz\in\mathcal B_U} \Bigl\|\zpz\wedge{\eta_{H}((u)g)}\Bigr\|&\leq  2E_1 e^{-k/D} e^{At}\eta^{-A}\label{almost invariance in inheritance proof}
\end{align}
Chose $\ref{more interesting inheritance constant}, \epsilon_2$ so that in particular $t' > \log (2E_1 e^{At}\eta^{-A})$. Then since $g \Gamma$ is $(\vare, t ')$-Diophantine
\begin{equation*}
\max_{\zpz \in \mathcal B_U} \Bigl\| \zpz \wedge{\eta_{H}((u)g)}\Bigr\| \geq \epsilon_1 \eta^{\ref{inheritance constant}} \left (2E_1 e^{At}\eta^{-A}\right )^{-\ref{inheritance constant}} \gg \eta^{2A\ref{inheritance constant}} e^{-A\ref{inheritance constant}t}.
\end{equation*}
But this contradicts \eqref{almost invariance in inheritance proof} if $k \geq \ref{more interesting inheritance constant} (t+\log(1/\eta)+\log(1/\epsilon_2))$ for sufficiently large $\ref{more interesting inheritance constant}$.
\end{proof}

\section{The family $\Hcal$ and the Diophantine condition}\label{sec:Hcal-diophatine}\label{s;singular}

\subsection{}\label{sec:Hcal}
Recall the family 
\[
\mathcal{H}= \Bigl\{\Hbf\subset\G:\text{$\H$ is a connected $\bbq$-subgroup and $\rad(\H)=\rad_u(\H)$}\Bigr\} 
\]
where $\rad(\H)$ (resp.\ $\rad_u(\H)$) denotes the solvable (resp.\ unipotent) radical of $\H.$ 


For any subgroup $\mathbf{H}\in\Hcal$, we put $H=\Hbf(\bbq_{\places}).$
Sometimes we write $H\in\hcal.$ 

\begin{lemma}\label{lem:ht-rad-u-rad}
There exists some $\consta\label{k:commut-rad}$ so that the following holds.
Let ${\bf L}\subset\SL_N$ be a connected algebraic group defined over $\bbq$.
Then 
\[
\height([{\bf L},{\bf L}])\dotll\height({\bf L})^{\ref{k:commut-rad}};\; \height(\rad({\bf L}))\dotll\height({\bf L})^{\ref{k:commut-rad}};\;\text{and }\; \height(\rad_u({\bf L}))\dotll\height({\bf L})^{\ref{k:commut-rad}}.
\]  
\end{lemma}

\begin{proof}
Let $\mathcal B$ be a $\bbz$-basis for $\Lie(\Lbf)\cap\SL_N(\bbz)$
so that $\|\zpz\|\ll\height(\Lbf)^\star$ for all $\zpz\in\mathcal B$.
Then $\{[\zpz,\zpz']: \zpz,\zpz'\in\mathcal B\}$ generates $[\Lie(\Lbf),\Lie(\Lbf)].$
Hence, 
\[
\height([\Lie(\Lbf),\Lie(\Lbf)])\dotll\height(\Lbf)^\star.
\]

It remains to bound $\height(\rad_u(\Lbf)).$ To that end, first note that
\[
\rad(\Lie(\Lbf))=\Bigl\{\zpz\in\Lie(\Lbf): \mathsf k_{\Lbf}(\zpz,[\wpz,\wpz'])=0, \forall\wpz,\wpz'\in\Lie(\Lbf)\Bigr\}
\]
where $\mathsf k_{\Lbf}$ is the killing form of $\Lie(\Lbf)$.
Therefore, $\height(\rad(\Lie(\Lbf)))\dotll\height(\Lbf)^\star$. 

Now let $\mathcal B'$ be a $\bbz$-basis for $\rad(\Lie(\Lbf))\cap\SL_N(\bbz)$
so that $\|\zpz\|\dotll\height(\Lbf)^\star$ for all $\zpz\in\mathcal B'$. Then
\[
\rad_u(\Lie(\Lbf))=\Bigl\{\zpz\in \rad(\Lie(\Lbf)):{\rm tr}(\wpz_1\cdots\wpz_s\zpz)=0, \forall1\leq s\leq N,\wpz_i\in\mathcal B'\Bigr\}.
\]
Hence, $\height(\rad_u(\Lie(\Lbf)))\dotll\height(\Lbf)^\star$. 
\end{proof}

\subsection{Algebraic properties of subgroups in class $\Hcal$}\label{sec:algeb-Hcal}
A quantitative notion of a point satisfying a Diophantine condition was given 
in Definition~\ref{def:Diophantine-intro}.
This definition is formulated in terms of certain representations whose constructions and basic properties we now recall. 

Let $\Hbf\in\Hcal$ be a proper subgroup.
Recall that $\gfrak=\oplus_{\places}\,\gfrak_v$ where $\gfrak_v=\Lie(G_v)$. Put   
\[
\text{$\rho_H:=\wedge^{\dim\H}\Ad\;$ and 
$\;V_H:=\oplus_\places\wedge^{\dim\H}\gfrak_v.$}
\] 
The representation $\rho_H$ is defined over $\bbq.$ 


Let ${\vbf_\H}$ be a primitive integral vector in 
$\wedge^{\dim\H}\Lie(\G)$ (or $\wedge^{\dim\H}\sl_N$) corresponding to the Lie algebra of $\H$, see~\S\ref{ss;height}. 
We embed ${\vbf_\H}$ diagonally in $V_H$ and denote this vector by $\vpz_H$. 
Let $\eta_H:G\to V_H$ denote the orbit map, that is
\[
\mbox{$\eta_H(g)=\rho_H(g)\vpz_H$ for all $g\in G$.}
\] 


Note that $\rho_H$ and $V_H$ depend only on $\dim\H,$ however, $\vbf_\H$ (similarly $\vpz_H$) 
uniquely determines $\Lie(\H)$ and hence $\H.$ 

\begin{lemma}\label{lem:rep-properties}
\begin{enumerate}
 \item $N_G(H)=\Bigl\{g\in G: \rho_H(g)\vpz_H=\Bigl(\chi_H(g_p)\Bigr)_{v\in\places}\vpz_H\Bigr\},$ 
 where $\chi_H$ is a rational character.
\item The orbit $\eta_H(\Gamma)$ is discrete and closed in $V_H$.
\end{enumerate}
\end{lemma}

\begin{proof}
Property~(1) is a consequences of the definition.

In light of our assumption that $\Gamma$ is arithmetic, property~(2) also follows from the definitions. We note, however, that this {\em qualitative} result does not require arithmeticity of $\Gamma$, see~\cite[Thm.~3.4]{DM}.
\end{proof}

\begin{lemma}\label{counting lemma}
There exists some constant $\consta\label{k:hcal-ht-T}$ so that the following holds.
\be\label{eq:hcal-ht-T}
\#\{\H\in\Hcal:\height(\H)\leq T\}\ll T^{\ref{k:hcal-ht-T}}.
\ee
\end{lemma}

\begin{proof}
This follows from the definitions of $\vbf_\H$ and $\height(\H).$
\end{proof}

\begin{lemma}\label{lem:subgroup-Hcal}
There exists some $\consta\label{k:MH}>0$ so that the following holds.
Given any $\bbq$-group $\Lbf\subset\G$, there exists a normal subgroup $\Lbf^\Hcal\subset\Lbf$ 
which is maximal among all subgroups of $\Lbf$ which belong to class $\Hcal$;
moreover, 
\be\label{eq:ht-H-Hcal}
\height\Bigl(\Lbf^\Hcal\Bigr)\ll\height(\Lbf)^{\ref{k:MH}}.
\ee
\end{lemma} 

\begin{proof}
Since $\Lbf/\rad_u(\Lbf)$ is a reductive group and unipotent subgroups in $\Lbf$
map to unipotent subgroups in $\Lbf/\rad_u(\Lbf)$, we have 
\be\label{eq:Hcal-H'RH}
\Lie(\Lbf^\hcal)=[\Lie(\Lbf),\Lie(\Lbf)]+\Lie(\rad_u(\Lbf));
\ee
in particular, $\Lbf^\hcal$ exists.

By Lemma~\ref{lem:ht-rad-u-rad} we have $\height([{\bf L},{\bf L}])\dotll\height({\bf L})^\star$ and 
$\height(\rad_u({\bf L}))\dotll\height({\bf L})^\star$. The claim thus follows from~\eqref{eq:Hcal-H'RH}.
\end{proof}

\subsection{}\label{sec:LH-MH}
Let 
\[
\Lbf_\H=\Bigl\{g\in\G: \wedge^{\dim\H}\Ad (g) \vbf_\H=\vbf_\H\Bigr\}.
\] 
Then $\Lbf_\H$ is a $\bbq$-group. 
The subgroup $\Lbf_\H$ is not necessarily in $\hcal.$ 
Define 
\[
\lplus_\H:=\Lbf_\H^\Hcal,
\] 
see \eqref{eq:ht-H-Hcal} for the notation. 

Put $L_H=\Lbf_\H(\bbq_\places)$ and $\llplus_H=\lplus_\H(\bbq_\places).$ Note that 
\[
L_H=\{g\in G: \rho_H(g)\vpz_H=\vpz_H\}.
\]

We will simply denote these groups by $\Lbf$, $L$, $\lplus$, and $M$ 
when there is no confusion.

\begin{lemma}\label{lem:height-L-M}
There exist
$\consta\label{k:height-L-M}$ with the following property. For any $\H\in\Hcal$ we have 
\be\label{eq:height-L-M}
\height(\Lbf_\Hbf)\ll\height(\Hbf)^{\ref{k:height-L-M}}\quad\text{and}\quad\height(\lplus_\Hbf)\ll\height(\Hbf)^{\ref{k:height-L-M}}.
\ee
\end{lemma}

\begin{proof}
Since $\lplus_\H:=\Lbf_\H^\Hcal,$ the second inequality is a consequence of the first inequality and~\eqref{eq:ht-H-Hcal}.

Recall now that
\[
\Lie\Bigl(\Lbf_\H\Bigr)=\Bigl\{\zpz\in\Lie(\G): \wedge^{\dim\H}\operatorname{ad} (\zpz) \vbf_\H=0\Bigr\},
\]
and that $\vbf_\H$ is an integral vector with $\|\vbf_\H\|=\height(\H)$. 

The first inequality thus follows, and the proof is complete. 
\end{proof}


\begin{lemma}\label{l;proper}\label{lem:proper}
\begin{enumerate}
\item For any $\gamma\in \Gamma$ and any $\H\in\Hcal$, we have 
\[
1\leq\height(\gamma\Hbf\gamma^{-1})=\cfun\Bigl(\eta_H(\gamma)\Bigr).
\]
\item Let $r>1$ and suppose $\gamma\in\Gamma$ is so that $\cfun\Bigl(\eta_H(\gamma)\Bigr)\leq \vsr$. 
Then 
\begin{enumerate}
\item $\height(\gamma{\bf L}_\H\gamma^{-1})\ll\vsr^\star.$
\item $\height(\gamma{\lplus_\H}\gamma^{-1})\ll\vsr^\star.$
\end{enumerate}
\end{enumerate}
\end{lemma}

\begin{proof}
Recall that $\Ad(\Gamma)\subset\SL_d(\bbz_\places)$.
Recall that ${\bf v_H}$ is primitive, in particular, $\|\eta_H(\gamma)\|_p=1$ for all $p\not\in\places$.
Part (1) of the lemma thus follows from the definition of $\height(\gamma\Hbf\gamma^{-1}).$

To see parts (2)a and (2)b, note that 
\[
\text{$\gamma{\bf L}_\H\gamma^{-1}={\bf L}_{\gamma\Hbf\gamma^{-1}}\;\;$
and $\;\;\gamma\lplus_\H\gamma^{-1}={\lplus}_{\gamma\Hbf\gamma^{-1}}.$}
\] 
Hence, the claim follows from part (1) and~\eqref{eq:height-L-M}.
\end{proof} 

Let $\H\in\Hcal.$ For any $g\in G$ and any $\vsr>1$, put   
\be\label{eq:def-rmh}
\RMH_H(g,\vsr):=\lceil\log\Bigl(R_H(g,r)\Bigr)\rceil,
\ee
where $R_H(g,r):=\max\bigl\{\cfun\Bigl(\eta_{\llplus_H}(g\gamma)\Bigr):\gamma\in\Gamma,\cfun\Bigl(\eta_H(\gamma)\Bigr)\leq\vsr\bigr\}$.

\begin{cor}\label{cor:rmh-latticept} 
\begin{enumerate}
\item $R_H(g,\vsr)\ll|g|^\star\vsr^\star.$
\item $\#\Bigl(\eta_H(\Gamma)\cap\vsballr\Bigr)\ll\vsr^{\star}$.
\end{enumerate}
\end{cor}

\begin{proof}
We first prove part~(1).
For any $\gamma\in\Gamma$ so that $\cfun\Bigl(\eta_H(\gamma)\Bigr)\leq\vsr$, we have
$\height(\gamma\lplus_\H\gamma^{-1})\ll r^\star$, see Lemma~\ref{lem:proper}(2)(b).
Moreover, by Lemma~\ref{lem:proper}(1), we have
\[
1\leq\height(\gamma\lplus_\H\gamma^{-1})=\cfun\Bigl(\eta_{\llplus_H}(\gamma)\Bigr).
\]
Using~\eqref{eq:def-k-adjoint} to control the effect of $g$, the above implies the claim in part~(1). 

The second claim follows from the fact that
$\Ad(\Gamma)\subset\SL_d(\bbz_\places)$.
\end{proof}

\begin{lemma}\label{lem:almost-tube}
Let $\H\in\hcal$. Assume there exist an $L>0$, a sequence $\ell_n\to0$, and a sequence $g_n\Gamma\to g\Gamma$
satisfying the following.
\begin{enumerate}
\item $\cfun(\eta_H(g_n))\leq L$ for all $n$, and 
\item $\max_{z\in\mathcal B_U}\|\zpz\wedge \eta_H(g_n)\|\leq \ell_n$ for all $n$.
\end{enumerate}
Then $g\in \{g'\in N_G(U,H): \cfun(\eta_H(g'))\leq L\}\Gamma$.
\end{lemma}

\begin{proof}
In view of our assumption, there exists a sequence $\{\gamma_n\}$ so that
\[
g_n\gamma_n^{-1}\to g.
\]
Hence, using the assumption in (1), we get that
\be\label{eq:eta-H-L-tube}
\cfun(\eta_{H}(g_n\gamma_n^{-1}\gamma_n))=\cfun(\eta_{H}(g_n))\leq L.
\ee
Moreover, since $g_n\gamma_n^{-1}\to g$, we have $|g_n\gamma_n^{-1}|\leq 1+|g|$ 
for all large enough~$n$. This and the above imply that for some constant $A'$ depending only on $N$, we have  
\[
\cfun(\eta_{H}(\gamma_n))=\height(\gamma_n\mathbf{H}\gamma_n^{-1})\leq L(2+|g|)^{A'}
\]
for all large enough $n$. 

Using~\eqref{eq:hcal-ht-T} and passing to a subsequence, we assume that 
$\gamma_n{\bf H}\gamma_n^{-1}=\gamma{\bf H}\gamma^{-1}$ for all $n$, or equivalently that  $\eta_H(\gamma_n)=\eta_H(\gamma)$.
%
Then for any $\zpz\in\mathcal B_U$
\begin{align*}
\zpz\wedge\eta_{\gamma H\gamma^{-1}}(g_n\gamma_n^{-1})&=\zpz\wedge\eta_{H}(g_n\gamma_n^{-1}\gamma_n)\\
&=\zpz\wedge\eta_{H}(g_n). 
\end{align*}
This computation and the assumption in~(2) now imply that 
\[
\|\zpz\wedge\eta_{\gamma H\gamma^{-1}}(g\gamma_n^{-1})\|\leq \ell_n\quad\text{for all $\zpz\in\mathcal B_U$.}
\]
Passing to the limit, we get that  $\zpz\wedge\eta_{\gamma H\gamma^{-1}}(g)=0$ for all $\zpz\in\mathcal B_U$.
That is
\be\label{eq:ggamma-is-right-1-tube}
\zpz\wedge\eta_{H}(g\gamma)=0\quad\text{for all $\zpz\in\mathcal B_U$}.
\ee

Similarly, using the fact that $\eta_H(\gamma_n)=\eta_H(\gamma)$ for all $n$ and passing to the limit in~\eqref{eq:eta-H-L-tube} we get that
\be\label{eq:ggamma-is-right-2-tube}
\cfun(\eta_H(g\gamma))\leq L.
\ee
In view of~\eqref{eq:ggamma-is-right-1-tube} and~\eqref{eq:ggamma-is-right-2-tube} we obtain 
\[
g\gamma\in\{g' \in N_G(U,H):\cfun(\eta_{H}(g'))\leq L\},
\]
as we claimed.
\end{proof}

\begin{cor}\label{lem:tube-closed}
Let $\H\in\hcal$ and let $L>0$. The set
\[
\{g\in N_G(U,H): \cfun(\eta_H(g))\leq L\}\Gamma/\Gamma
\]
is a closed subset of $G/\Gamma$.
\end{cor}

\begin{proof}
Recall that $N_G(U,H)=\{g\in G: \zpz\wedge\eta_{H}(g\gamma)=0$ for all $\zpz\in\mathcal B_U\}$.
The claim thus follows from Lemma~\ref{lem:almost-tube}.
\end{proof}

\subsection{}\label{sec:Loj-inequality}
Theorems A and B below will be used in the proof of Lemma~\ref{lem:loj}.
We begin by recalling an effective versions of Hilbert's Nullstellensatz theorem;
the statement presented here is due to D.~Masser, G.~W\"{u}stholz,~\cite[Thm. IV]{EffNul-Mass-Wus}, see also~\cite{Seid, JKS} and references there.

\begin{thm*}[Effective Nullstellensatz]
Assume $f, f_1,\ldots, f_n\in\bbz[t_1,\ldots,t_m]$ have total degree at most $D_0$ and logarithmic height at most $\mathsf h$.
Suppose $f$ vanishes at all the common zeros (if any) of $\{f_i\}$ in $\bbc^m$.

Put $M=2^{m-1}.$ Then there exist
\begin{itemize}
\item some $b\in\bbn$ with $b\ll (8D_0)^{2M}$, 
\item $q_1,\ldots,q_n\in\bbz[t_1,\ldots,t_m]$ of total degree at most $(8D_0)^{2M+1}$ and logarithmic height at most
$(8D_0)^{4M-1}(\mathsf h+8D_0\log(8D_0))$, and 
\item some $a\in\bbz$ with $\log|a|\leq (8D_0)^{4M-1}(\mathsf h+8D_0\log(8D_0))$
\end{itemize}
so that
\[
af^b=\sum_i q_if_i.
\]
\end{thm*}

We need the following theorem of W.~Brownawell which can be thought of as a local version of the above theorem.

\begin{thma*}[Cf.~\cite{Brwell-Loj}]
Let $f_1,\ldots, f_n\in\bbz[t_1,\ldots,t_m]$ have total degree at most $D_0$ and logarithmic height at most $\mathsf h$.
If $f_1,\ldots,f_n$ have no common zero within $0<\bpz\leq1$ of some $w\in\bbc^m$, then
\[
\max\{|f_j(w)|:1\leq j\leq n\}\geq C_1\,e^{-\ref{k:loj-brw}\mathsf h}\Bigl(\frac{\|w\|^2}{\bpz}\Bigr)^{-\ref{k:loj-brw}}
\]
where $C_1$ and $\consta\label{k:loj-brw}$ are explicit constants depending only on $n, m$, and $D_0$.
\end{thma*}

In the $p$-adic setting, we have the following theorem. This theorem is proved by M.~Greenberg, we reconstruct Greenberg's proof 
in Appendix \ref{sec:thmB} to make the dependence on the height of the polynomials in question explicit.

\begin{thmb*}[Cf.~\cite{GrBerg-Loj-1} and~\cite{GrBerg-Loj-2}]
Let $f_1,\ldots, f_n\in\bbz[t_1,\ldots,t_m]$ have total degree at most $D_0$
and logarithmic height at most $\mathsf h$.
There exists $\consta\label{k:GrBr-Loj}$ depending only on $m,$ $n$, and $D_0$ so that the following holds.

Suppose $w_1,\ldots, w_m\in\bbz_p$ and $C_2> 2\ref{k:GrBr-Loj}\mathsf h$ are such that 
\[
\text{$f_j(w_1,\ldots,w_m) \equiv 0\; {\rm (mod}\; p^{C_2}{\rm )}$ for all $j.$}
\]
Then, there exist $y_1,\ldots,y_m \in\bbz_p$ such that 
\[
y_i \equiv w_i\quad {\rm\biggl(mod} \;p^{\lceil\frac{ C_2-\ref{k:GrBr-Loj}\mathsf h}{\ref{k:GrBr-Loj}}\rceil}{\rm \biggr)}
\] 
and $f_j(y_1,\ldots,y_m)=0$ for all $j.$ 
\end{thmb*}

The following lemma is a crucial ingredient for our inductive argument in the proof of Theorem~\ref{thm:eff-linearization-general}.

\begin{lemma}\label{l;lojas-ineq}\label{s;loj-ineq}\label{lem:loj}
There exist $\consta\label{k:epsion-loja}$, $\consta\label{k:Loja}$, and $C_0$ where $C_0$ depends on
$N$, the number of places $\#\places$, and polynomially on the finite primes in $\places$ and on $\height(G)$
so that the following holds.

Let $r>1$, $\epsilon>1$, and $g\in G$ be fixed.  
Suppose $\Hbf_1,\Hbf_2 < \Gbf$ are two $\bbq$-subgroups of class-$\mathcal H$  
with $\cfun\Bigl(\eta_{H_i}(g)\Bigr)\leq r$ for $i=1,2.$
Assume that 
\be\label{eq:assump-loj}
\max_{\zpz\in\bcal_U}\Bigl\|\zpz\wedge{}{\eta_{H_i}(g)}\Bigr\|\leq \epsilon\quad
\text{ for $i=1,2.$}
\ee
Let $\Hbf_{1,2}:=\bigl(\Hbf_1\cap\Hbf_2\bigr)^{\Hcal}$.
Then if $\epsilon\leq C_0 |g|^{-\ref{k:epsion-loja}} r^{-\ref{k:epsion-loja}}$, the group $\H_{1,2}$ is not trivial, $\height(\Hbf_{1,2})\ll |g|^\star r^\star$, and 
\be\label{eq:concl-loj}
\max_{\zpz\in\bcal_U}\Bigl\|\zpz\wedge{}{\eta_{H_{1,2}}(g )}\Bigr\|
\ll |g|^{\ref{k:Loja}}r^{\ref{k:Loja}}\epsilon^{1/\ref{k:Loja}}.
\ee
\end{lemma}

\begin{proof}
First note that~\eqref{eq:assump-loj} and~\eqref{eq:def-k-adjoint} imply the following:
\[
\bigl\|\Ad(g^{-1})\zpz\wedge \vpz_{H_i}\bigr\|\dotll |g|^\star\epsilon\quad
\text{ for $i=1,2$ and all $\zpz\in\mathcal{B}_U$}. 
\]
Rewriting this at the level of the Lie algebra, we have
\be\label{eq:loj-dist-Lie}
\vsdist(\Ad( g^{-1})\zpz,\hfrak_i)\ll |g|^\star r^\star\epsilon\quad
\mbox{ for $i=1,2$ and all $\zpz\in\mathcal{B}_U,$} 
\ee 
where $\mathfrak h _ i$ denotes the Lie algebra of $H _ i = \mathbf H _ i (\bbq _ \Sigma)$.

Now~\eqref{eq:def-k-adjoint} and $\cfun\Bigl(\eta_{H_i}(g)\Bigr)\leq r$ imply 
\[
\height(\Hbf_i)=\cfun\Bigl(\vpz_{H_i}\Bigr)\dotll r|g|^\star\quad\text{ for $i=1,2$.}
\]
Therefore, $\height (\Hbf _ 1 \cap \Hbf _ 2) \ll \height (\Hbf _ 1) \cdot \height (\Hbf _ 2)$ and hence by~\eqref{eq:ht-H-Hcal} we have $\height(\Hbf_{1, 2}) \ll |g|^\star r^\star$.

As $\hfrak_1$ and $\hfrak_2$ 
are rational subspaces of $\gfrak$ with height $\ll |g|^\star r^\star$,
the estimates~\eqref{eq:loj-dist-Lie} thus imply that
\be\label{e;dist-intersection}
{\vsdist(\Ad( g^{-1})\zpz,\hfrak_1\cap\hfrak_2)\ll |g|^\star r^\star\epsilon\quad 
\text{ for all $\zpz\in\mathcal{B}_U,$}}
\ee
see, e.g., \cite[\S13.4]{EMV}.

For every finite place $p\in\places$ let $\Omega_p=\bbq_p$, 
and let $\Omega_\infty=\bbc$. Set $\Omega_\places=\prod_{\places}\Omega_p$. 
Let $\mathcal{N}$ denote the cone of ad-nilpotent elements in 
$\gfrak\otimes \Omega_\places.$ 
Then 
\be\label{eq:z-ad-nil}
\Ad( g^{-1})\zpz\in\mathcal N.
\ee

There are $n,m\gg1$ so that the subspace $\hfrak_1\cap\hfrak_2$ and the cone $\mathcal N$
are $\bbq$-varieties defined by $\{f_{{\rm sp},j}:1\leq j\leq n\}\subset\bbz[t_1,\ldots, t_m]$ and 
$\{f_{{\rm cn},j}:1\leq j\leq n\}\subset\bbz[t_1,\ldots, t_m]$, respectively\footnote{The subscript sp stands for subspace and cp stands for cone.}; further, 
the logarithmic heights $\mathsf h$ of these polynomials are bounded by 
\[
B_0+\log r
\] 
for some $B_0$ depends on
$N$, the number of places $\#\places$, and polynomially on the finite primes in $\places$ and on $\height(G)$.

In particular, conditions of Theorems A and B are satisfied for $\{f_{{\rm sp},j}\}\cup\{f_{{\rm cn},j}\}$.
In view of Theorems A and B, thus,~\eqref{e;dist-intersection} and~\eqref{eq:z-ad-nil} imply the following estimate
\[
\vsdist\biggl(\Ad( g^{-1})\zpz,\mathcal{N}\cap\Bigl(\Bigl(\hfrak_1\cap\hfrak_2\Bigr)\otimes{\Omega_\places}\Bigr)\biggr)\ll
|g|^\star r^\star\epsilon^\star.
\]

Let $\hfrak_{1,2}=\Lie(\H_{1,2})$. By the definition of the family $\hcal$, see~\S\ref{sec:Hcal},
we have $\hfrak_{1,2}$ contains the Lie algebra generated by 
$\mathcal{N}\cap\Bigl(\Bigl(\hfrak_1\cap\hfrak_2\Bigr)\otimes\Omega_\places\Bigr)$.
Therefore, the above estimate implies that 
\be\label{e;dist-h'}
{\vsdist\Bigl(\Ad( g^{-1})\zpz,\hfrak_{1,2}\otimes\Omega_\places\Bigr)
\ll |g|^\star r^\star\epsilon^\star\quad
\mbox{ for all $\zpz\in\mathcal{B}_U$}}.
\ee
Now since $\Ad( g^{-1})\zpz\in\gfrak$, we get  the following from~\eqref{e;dist-h'}.
\be\label{eq:dist-h'-Q-S}
{\vsdist(\Ad( g^{-1})\zpz,\hfrak_{1,2})
\ll |g|^\star r^\star\epsilon^\star\quad
\mbox{ for all $\zpz\in\mathcal{B}_U$}}.
\ee
Equations~\eqref{eq:dist-h'-Q-S} implies $\hfrak_{1,2}\neq\{0\}$  
so long as right hand side of~\eqref{eq:dist-h'-Q-S} is a sufficiently high power of $|g|^{-1}$; 
this is satisfied if $\epsilon\ll |g|^\star r^\star.$
Equation~\eqref{eq:concl-loj} is now an immediate consequence of~\eqref{eq:dist-h'-Q-S}.
\end{proof}

\section{Non-divergence of unipotent flows in $\SL_N(\bbq_\Sigma)/\SL_N(\bbz_\Sigma)$ with an application to almost invariant Lie algebras}\label{sec:non-div-I}


In this section we recall the basic nondivergence results regarding the action of unipotent groups on $\SL_N(\bbq_\Sigma)/\SL_N(\bbz_\Sigma)$, and deduce some important corollaries that will play a central role in the following sections. The basic reference for this section is the paper \cite{KT:Nondiv} by Kleinbock and Tomanov, which can be viewed as a $\Sigma$-arithmetic adaptation of \cite{KM} by Kleinbock and Margulis (which itself relies on the nondivergence result of Margulis \cite{Margulis-Nondiv}, perhaps the first general result regarding dynamics of unipotent groups on arithmetic quotients, and Dani \cite{Dani}).

Some of the implicit multiplicative constants in this section satisfy a stricter requirement, i.e., they depend on  
$N$, $\#\places$, and polynomially on the finite primes in $\places$, {\em but not on $\height(\Gbf)$}. We will explicate these  by an index, i.e., we write $\ll_{N,\places}$ or $\gg_{N,\places}$ for these implicit multiplicative constants.


\medskip

\subsection{} Let $\GL _ N ^ 1 (\bbq _ \Sigma)$ denote the group
\begin{equation*}
\GL _ N ^ 1 (\bbq _ \Sigma) = \left\{ (g _ v) \in \GL (\bbq _ \Sigma): \prod_ {v \in \Sigma} \det (g _ v) = 1 \right\}
.\end{equation*}
Then we can identify $\GL _ N ^ 1 (\bbq _ \Sigma) / \GL (\bbz _ \Sigma)$ with the space of
discrete $\bbz_\Sigma$-modules in $\bbq _ \Sigma ^ N$ of covolume 1, and there is a natural injective proper map  from $
\SL_N(\bbq_\Sigma)/\SL_N(\bbz_\Sigma)$ to $\GL _ N ^ 1 (\bbq _ \Sigma) / \GL (\bbz _ \Sigma)$ obtained by assigning to $(g_v)_{v \in \Sigma}$ the  $\bbq_\Sigma$-module spanned by the elements in $\bbq _ \Sigma ^ N$ formed by taking the $i$th column of all $g_v$ for $i=1$, \dots, $N$.
In view of this, we will view $
\SL_N(\bbq_\Sigma)/\SL_N(\bbz_\Sigma)$ as embedded in $\GL _ N ^ 1 (\bbq _ \Sigma) / \GL (\bbz _ \Sigma)$.

Let $\Gamma _ 1 = \GL _ N (\bbz _ \Sigma)$ and $G _ 1 = \GL _ N ^ 1 (\bbq _ \Sigma)$. For $x = g \Gamma _ 1 / \Gamma _ 1 \in G _ 1 / \Gamma _ 1$, let
\begin{equation*}
\alpha (x) = \max  \left\{ 1/\cfun (\zpz): \zpz \in g \bbz _ \Sigma ^ N \setminus \left\{ 0 \right\} \right\}; \end{equation*}
this function is a proper map from $G _ 1 /\Gamma_1$ to $\bbr ^ +$ (as well as from
the quotient space $\SL_N(\bbq_\Sigma)/\SL_N(\bbz_\Sigma)$ to $\bbr^+$) and any compact subset of $G _ 1 /\Gamma_1$ is contained in the compact subset of the form $\{ x: \alpha (x) < M\}$ for some $M > 0$.
Let $\Delta$ be a $\bbz _ \Sigma$-submodule  of rank $k$ in a discrete $\bbz _ \Sigma$-module $g \bbz _ \Sigma ^ N$, say generated over $\bbz _ \Sigma$ by $v_1, \dots, v_k \in \bbq _ \Sigma ^ N$.  Then while $v _ 1, \dots, v _ k$ are not uniquely defined, the wedge $v _ 1 \wedge \dots \wedge v _ k$ in $\wedge ^ k \bbq _ \Sigma ^ N$ is, 
and we define $\cfun (\Delta) = \cfun (v _ 1 \wedge \dots \wedge v _ k)$. A $\bbz _ \Sigma$-submodule $\Delta$ of $g \bbz _ \Sigma ^ N$ is said to be \emph{primitive} in $g \bbz _ \Sigma ^ N$ if it is maximal with respect to finite-index extensions, i.e.\ it is not a proper $\bbz _ \Sigma$-submodule of finite index in any $\bbz _ \Sigma$-submodule of $g \bbz _ \Sigma ^ N$.

The results of \cite{KT:Nondiv} are more general in that they deal with general ``$(c, \alpha)$-good'' maps from a convex $B$ in a product of parameter spaces over $\bbq _ v$ for $v$ in some subset of $\Sigma$ to $G_1/\Gamma _ 1$, but the basic nondivergence estimate of the paper \cite[Thm. 9.4]{KT:Nondiv} gives the following:

\begin{thm} [cf. \cite{KT:Nondiv}]\label{Kleinbock-Tomanov theorem}
Let $U=\prod_{v\in\places} U_v$ be a $\bbq _ \Sigma$-unipotent subgroup, $\mathsf B_U(e)$ an open ball in $U$ and $\lambda _ k $ as in \S\ref{sec:lambda-n}. Let $g \in \GL _ N ^ 1 (\bbq _ \Sigma)$ and assume that for every primitive $\bbz _ \Sigma$-submodule $\Delta$ of $g \bbz _ \Sigma ^ N$ of rank $1 \leq k \leq N-1$
\begin{equation}\label{everything gets big inequality}
\max_ {u\in\mathsf B_U(e)} \cfun (\la_k(u) \Delta) \geq \eta
.\end{equation}
Then
\[
\left |\left \{u\in\mathsf B_U(e):\alpha(\la_k(u)g\Gamma_1) >\epsilon^{-1}\right \}\right|< E\left (\frac\epsilon\eta \right )^{1/D} \absolute{\mathsf B_U(e)},
\] 
with $D$ depending only on $N$ and $E$ depending on $N$, $\#\places$, and polynomially on finite primes in $\places$.
\end{thm}

In fact, the basic inductive argument used to prove Theorem \ref{Kleinbock-Tomanov theorem}, specifically \cite[Thm.~6.1]{KT:Nondiv} can be used to provide a more precise result that would be important for us in the sequel. This result does not seem to appear in the literature. One can view Kleinbock's~\cite[Thm.~0.2]{K-Ext} as a step in this direction, and a result very close to what we give below can be found in a draft by Breuillard and de Saxce \cite{BdS}. 

For $g \Gamma _ 1 \in G _ 1 / \Gamma _ 1$ and $1 \leq i \leq N -1$ let
\begin{equation*}
\alpha _ i (g \Gamma _ 1) = 1 / \min \left\{ \cfun (\Delta): \text{$\Delta$ is a primitive $\bbz _ \Sigma$-submodule of $g \bbz _ \Sigma ^ N$ of rank $i$} \right\}
.\end{equation*}

\begin{thm}\label{fancy nondivergence theorem} With the notations of Theorem~\ref{Kleinbock-Tomanov theorem} (but without the assumption~\eqref{everything gets big inequality}), there are $0=k_0 < k _ 1 < k _ 2 < \dots < k _ \ell < k_{\ell+1}=N $, and primitive $\bbz _ \Sigma$-submodules $\Delta _ {k_1} < \Delta _ {k_2} < \dots < \Delta _ {k_\ell}$  of $g \bbz _ \Sigma ^ N$ of rank corresponding to their index  so that if  $\eta(k_0),\dots,\eta(k_{\ell+1}) \in (0,1]$ is defined by
\begin{align}
\eta(0)&=\eta(N)=1 \nonumber\\
\eta(k_i) &= \max_ {u \in\mathsf B_U(e)} \cfun (\la_k(u) \Delta_{k_i}) \qquad\text{for $1\leq i \leq \ell$}
\label{eta equation}\end{align}
 then $\eta(\bullet)$ can be extended to a function $[1,N]\to(0,1]$ so that  $-\log \eta: [1,N]\to \R^+$ is concave and linear on each interval $[k_0, k_{1}]$, \dots, $[k_{\ell},k_{\ell+1}]$  and
\[
\left |\left \{ u \in\mathsf B_U(e): \exists i \text{ s.t.\ }\frac{\alpha_i(\la_k(u)g \Gamma_1)^{-1}}{\eta(i)} < \epsilon^i \right \}\right|< E \epsilon ^{1/D} \absolute{\mathsf B_U(e)},
\]
with $D$ depending only on $N$ and $E$ depending on $N$, $\#\places$, and polynomially on finite primes in $\places$. Moreover, given a primitive $\bbz _ \Sigma$-submodule $\tilde\Delta < g \bbz _ \Sigma ^ N$, we can choose $\Delta _ {k_1} < \Delta _ {k_2} < \dots < \Delta _ {k_\ell}$  so that 
\[\eta(\rank(\tilde\Delta)) \leq \max_ {u \in\mathsf B_U(e)} \cfun (\la_k(u) \tilde\Delta).
\]
\end{thm}
\medskip

\noindent
Note that it easily follows from the $\Sigma$-arithmetic version of Minkowski's second theorem, \cite[\S C.2, specifically Thm.~C.2.11]{Bombieri-Gubler}, that under the assumption \eqref{eta equation} for any $u \in \mathsf B_U(e)$ one can complete the partial flag $\Delta _ {k_1} < \Delta _ {k_2} < \dots < \Delta _ {k_\ell}$ of submodules of $ g \bbz _ \Sigma ^ N $to a full flag of primitive $\bbz _ \Sigma$-modules $\Delta _ 1 < \dots < \Delta _ {N -1}$ so that if $k_i < r  <k_{i+1}$ and $\tau = (k_{i+1}-r)/(k_{i+1}-k_{i})$ then
\begin{equation}\label{using Minkowski's 2nd} 
\begin{aligned}
 \cfun (\la_k(u)\Delta _ r) &< A \cfun \left (\la_k(u)\Delta _ {k_i}\right)^{\tau}\cfun \left(\la_k(u)\Delta _ {k_{i+1}}\right)^{1-\tau}\\
 &\leq A\eta(k_{i})^\tau \eta(k_{i+1})^{1-\tau}= A\eta(r)
,
\end{aligned} 
\end{equation}
 with $A$ depending only on $N$ and $\places$. 
 Hence for all $u \in \mathsf B_U(e)$
\begin{equation*}
\frac{\alpha_r(\la_k(u)g \Gamma_1)^{-1}}{\eta(r)} < A
.\end{equation*}

\begin{proof}
Consider the (finite) collection of all primitive $\bbz _ \Sigma$-submodules $\Delta < g \bbz _ \Sigma ^ N$ so that 
\begin{equation}\label{max lt one}
\max_ {u \in\mathsf B_U(e)} \cfun (\la_k(u) \Delta)<1,
\end{equation}
and for each such $\Delta$, let
\[\eta_\Delta = \max_ {u \in\mathsf B_U(e)} \cfun (\la_k(u) \Delta)
.\]
From all the possible partial flags of primitive $\bbz _ \Sigma$-submodules $\Delta _ {k_1} < \Delta _ {k_2} < \dots < \Delta _ {k_\ell}$ with all $\Delta _ {k _ i}$ in this sub collection, choose one for which the convex hull of the pairs of points
\begin{equation}\label{set to be taken convex hull of}
\left\{ (0,0), (k _ 1, - \log \eta _ {\Delta _ {k _ 1}}), \dots, (k _ 1, - \log \eta _ { \Delta _ {k _ \ell}}), (N,0) \right\}
\end{equation}
is maximal (with respect to the usual partial order by inclusion on subsets of $\bbr^2$).
There could be more than one possible choice, but any one of these choices would be good enough for us, and if $\tilde\Delta$ satisfies~\eqref{max lt one} we can choose such a $\Delta _ {k_1} < \Delta _ {k_2} < \dots < \Delta _ {k_\ell}$ so that the convex hull of the points in~\eqref{set to be taken convex hull of} contains the point 
$(\rank(\tilde\Delta), -\log \eta_{\tilde\Delta})$.

Fix the choice of primitive $\bbz _ \Sigma$-submodules $\Delta _ {k_1} < \Delta _ {k_2} < \dots < \Delta _ {k_\ell}$ and let $\eta: [0, N] \to \bbr ^ +$ be as in the statement of the theorem. Then the graph of $- \log \eta (\bullet)$ forms the upper half of the boundary of the convex hull of the set in \eqref{set to be taken convex hull of}, and 
$\eta(\rank(\tilde\Delta)) \leq \eta_{\tilde\Delta}$.

By the choice of the $\Delta _ {k _ i}$ and definition of $\eta (\bullet)$, it follows that for any $1 \leq r \leq N-1$ and any $\bbz _ \Sigma$-primitive submodule $\Delta$ of rank $r$ of $g \bbz _ \Sigma ^ N$ compatible with $\Delta _ {k_1} < \Delta _ {k_2} < \dots < \Delta _ {k_\ell}$, 
\[
\max_ {u \in\mathsf B_U(e)} \cfun (\la_k(u) \Delta)\geq \eta(r)
.\]

Applying \cite [Thm.~6.1]{KT:Nondiv} similarly to the way it is used to prove \cite[Thm.~9.3]{KT:Nondiv}, but with the poset used in \cite[Thm.~6.1]{KT:Nondiv} being the collection of $\bbz _ \Sigma$-submodules of $g \bbz _ \Sigma ^ N$ compatible with the chosen partial flag  $\Delta _ {k_1} < \Delta _ {k_2} < \dots < \Delta _ {k_\ell}$ one obtains that outside a subset $\mathsf C \subset \mathsf B_U(e)$ of measure $\absolute {\mathsf C} \ll_{N,\places} \epsilon ^{\star}$ we can find for every $u \in \mathsf B _ U (e) \setminus {\mathsf C}\,$ a completion $\Delta _ 1 < \dots < \Delta _ {N -1}$ (depending on $ u$) of the fixed partial flag $\Delta _ {k_1} <  \dots < \Delta _ {k_\ell}$ so that for every~$i$
\begin{equation}\label{good flag}
\epsilon \eta (i) \leq \cfun (\la_k(u)\Delta _ i) \leq A' \eta(i) 
,\end{equation}
with $A'$ depending only on $N$ and $\places$.
To be precise, we apply a variant of \cite [Thm.~6.1]{KT:Nondiv} where the marking equations (M1) and (M2) on \cite [p.~540]{KT:Nondiv} for a partial flag $\mathfrak G_u$ (compatible with our fixed flag $\Delta _ {k_1} <  \dots < \Delta _ {k_\ell}$) are replaced by (in the notations of this paper)
\begin{enumerate}
\item[(M1)] $\eta (\rank \Delta) \geq \cfun (\la_k(u) \Delta) \geq \epsilon \eta (\rank \Delta)$ for every $\Delta \in \mathfrak G_u$
\item [(M2)] $\cfun (\la_k(u) \Delta) \geq \eta (\rank \Delta)$ for every $\Delta$ compatible with $\mathfrak G_u$ and $\Delta _ {k_1} <  \dots < \Delta _ {k_\ell}$ but not in $\mathfrak G_u$.
\end{enumerate}
The argument of \cite [Thm.~6.1]{KT:Nondiv} would give us that for $u$ outside the set ${\mathsf C}$ as above there exists a partial flag $\mathfrak G_u$ for which (M1), (M2) holds.
Subsequently applying Minkowski's 2nd theorem (cf.~note following the statement of Theorem~\ref{fancy nondivergence theorem}, particularly \eqref{using Minkowski's 2nd}) we can complete the flag $\mathfrak G_u$ to a full flag so that \eqref{good flag} holds.

Such a marking was used in \cite{KT:Nondiv} (and \cite{KM}) to show that there is no primitive $v \in \la_k(u)g \bbz _ \Sigma ^ N$ with small $\cfun (v)$, i.e. to control $\alpha (g \Gamma _ 1) = \alpha _ 1 (g \Gamma _ 1)$, but in fact can be used to show $\alpha _ i (\la_k(u)g \Gamma _ 1) \ll_{N,\places} \epsilon^{-i}\eta (i)^{-1}$, as we now show.

The proof is by induction on the rank of the submodule $\Delta < g \bbz _ \Sigma ^ N$, and all implicit constants may depend on the step in the induction.
Note that since $- \log \eta (i)$ is a concave function,
\begin{equation}\label{monotenicity of eta}
\frac {\eta (i)} {\eta (i-1)} \leq \frac {\eta (i+1)} {\eta (i)} \qquad \text{for all $1 \leq i \leq N-1$}
.\end{equation}
We also recall the following important inequality for any primitive $\Delta, \Delta ' < g \bbz ^ N _ \Sigma$ and any $u \in U$
\begin{equation}\label{basic sum-intersection inequality}
\cfun (u\Delta) \, \cfun (u\Delta ') \geq \cfun (u\Delta \cap u\Delta ') \,\cfun (u\Delta' + u\Delta)/A
,\end{equation}
with $A$ depending only on $N,\Sigma$.

We start induction with rank one primitive submodules $ \bbz _ \Sigma v< g \bbz _ \Sigma ^ N$. Let $i$ be such that $v \in \Delta_ {i+1}$ but not in $\Delta _i$ (where for this purpose we take $\Delta _0=\{0\}$ and $\Delta _ N = g \bbz _ \Sigma ^ N$). Then by \eqref{basic sum-intersection inequality} and~\eqref{monotenicity of eta},
\begin{equation*}
\cfun (\la_k(u) \bbz _ \Sigma v) \geq \frac{\cfun (\la_k(u) \Delta _{i+1})}{A \cfun (\la_k(u) \Delta _{i})}\geq \frac{ \epsilon \eta (i+1) }{AA' \eta (i)} \geq \frac{\epsilon \eta(1)}{AA'}
.\end{equation*}

Consider now a rank-$r$ primitive submodule $\Delta < g \bbz _ \Sigma ^ N$, let $i$ be such that $\Delta < \Delta _ {i+1}$ and $i$ is minimal such (clearly, $i+1 \geq r$). Applying \eqref{basic sum-intersection inequality} once again, we obtain
\begin{equation*}
\cfun (\la_k(u) \Delta )\,\cfun (\la_k(u) \Delta _{i}) \geq \cfun (\la_k(u) (\Delta _{i}\cap \Delta))\,\cfun (\la_k(u) \Delta _{i+1})/A
.\end{equation*}
By induction $\cfun (\la_k(u) (\Delta _{i}\cap \Delta)) \gg_{N,\places} \epsilon ^ {r-1} \eta (r-1)$ hence
\begin{align*}
\cfun (\la_k(u) \Delta ) &\gg_{N,\places} \epsilon ^ {r-1} \eta (r -1) \frac {\cfun (\la_k(u) \Delta _{i+1}) }{\cfun (\la_k(u) \Delta _{i})}\\
& \gg_{N,\places} \epsilon ^ {r} \eta (r -1) \frac {\eta (i +1) }{ \eta (i)} \\
&\gg_{N,\places} \epsilon ^ {r} \eta (r -1) \frac {\eta (r) }{ \eta (r-1)} = \epsilon ^ {r} \eta (r)
,\end{align*}
and we are done.
\end{proof}

A key ingredient in the works of Margulis, Dani, Kleinbock-Margulis, and Kleinbock-Tomanov quoted above is an estimate on the size of the set where a polynomial function is small. The result needed, at least for the real case (i.e.~$\Sigma = \left\{ \infty \right\}$) is known as Remez inequality, and is used in \cite{KM} and \cite{KT:Nondiv} to verify the ``$(C, \alpha)$-good'' property. Since we will also use it in the sequel, we quote it below (in a slightly sharper form, though this is not relevant to us; Cf.\ e.g.\ \cite[Prop.~3.2]{KM}).

\begin{lemma}\label{l;polygood0}
Let $\mathfrak F$ be a local field with absolute value $|\;|$. 
Let $\mathsf B$ be a compact convex subset of $\mathfrak F^{\,r}$, and let $f\in \mathfrak F\,[t_1,\cdots,t_r]$ be a nonzero polynomial of degree $d$. Then for any $\delta>0$ we have
\be\label{eq:Good-Remez}
\Bigl|\Bigl\{z\in{\mathsf B}:|f(z)|< \delta\sup_{{z}\in\mathsf{B}}|f(z)|\Bigr\}\Bigr|\leq c \delta^{1/d}|{\mathsf B}| ,
\ee
where $|\mathsf K|$ denotes the Haar measure of $\mathsf K$ for any subset $\mathsf K\subset \mathfrak F^{\,r}$, with $c$ depending only on $d$ and $r$.
\end{lemma}

See \cite{Brudnyi-Ganzburg-Remez-inequality} for a proof for $k=\bbr$; the general case is essentially identical.

\begin{proof}[Sketch of proof]
Let $\delta' = \delta\sup_{{z}\in\mathsf{B}}|f(z)|$. For $r=1$ this follows from Lagrange's interpolation formula. For higher dimension, let $x \in \mathsf B$ be such that $f(x)=\sup_{{z}\in\mathsf{B}}|f(z)|$. Then there is a line $\ell$ through $x$ where 
\[
\frac{|\Bigl\{z\in{\mathsf B}:|f(z)|< \delta'\Bigr\}\cap \ell| }{ |\mathsf B\cap \ell|}>
c_1 \frac {|\Bigl\{z\in{\mathsf B}:|f(z)|< \delta'\Bigr\}|}{|\mathsf B|}.
\] 
Since $x \in \ell$ by the choice of $x$ we have  
\[
\sup_{{z}\in\mathsf{B}}|f(z)|=\sup_{{z}\in\mathsf{B}\cap \ell}|f(z)|;
\]
now apply the one dimensional result.
\end{proof}

\begin{lemma}\label{l;polygood}
Let $\Sigma' \subset \Sigma$.
For all positive integers $r$ and $d$, there exist explicit constant
$c=c(r,d,\places')$ with the following property.
For every $v \in \Sigma'$ and every $1\leq j\leq r'_v$ let 
$f_{v,j} \in\bbq_v[t_1,\cdots,t_{r_v}]$ be a nonzero polynomial of degree $\leq d$.
Define 
\[
f_v(t_1,\cdots,t_{r_v})=\|(f_{v,1}(t),\ldots, f_{v,r'_v}(t))\|_v=\max\{|f_{v,j}(t)|_v: 1\leq j\leq r'_v\}.
\]
Let $\mathsf B = \prod_ {v \in \Sigma'} \mathsf B _ v$ where $\mathsf B_v$ is a convex set in $\Q_v^{r_v}$ for each $v$, 
and set
\[
F(t_{vi}: v \in \Sigma ', 1 \leq i \leq r _ v) = \prod_ {v \in \Sigma'} f_v(t_1,\cdots,t_{r_v}).
\]
Then for any $\delta>0$ we have
\[
\Bigl|\Bigl\{ z \in{\mathsf B}:F(z)< \delta \sup_{{z}\in\mathsf{B}}F(z)\Bigr\}\Bigr|\leq c |\log \delta|^{\#\Sigma'-1}
\delta^{1/d}|{\mathsf B}|.
\]
Similarly, if we put $F(t_{vi})=\max_{\places'}f_v(t_{vi})$, then 
\[
\Bigl|\Bigl\{ z \in{\mathsf B}:F(z)< \delta \sup_{{z}\in\mathsf{B}}F(z)\Bigr\}\Bigr|\leq c \delta^{1/d}|{\mathsf B}|.
\] 
\end{lemma}

\begin{proof}
We first prove the first claim. Hence, let 
$F= \prod_ {v \in \Sigma'} f_v$ be as in that statement; note that $\max F=\prod\max f_v$.
Moreover,~\eqref{eq:Good-Remez} holds true for $f_v$ in place of $|f|$, see e.g.~\cite[Lemma 3.1]{KT:Nondiv}.

Note also that it suffices to prove the lemma for $\delta=2^{-m}$ where $m$ is a non-negative integer.
For all nonnegative integers $m'$ and any $v\in\places'$, put
\[
\mathsf B_v^{f_v,m'}=\Bigl\{z\in\mathsf B_v: f_v(z)\leq 2^{-m'}\max_{\mathsf B_v} f_v\Bigr\}.
\]
Then we have 
\[
\Bigl\{ z \in{\mathsf B}:F(z)< 2^{-m} \sup_{{z}\in\mathsf{B}}F(z)\Bigr\}=\bigcup \textstyle\prod_{\places'}\mathsf B_v^{f_v, m_v}
\]
where the union is taken over all partitions $m=\sum_{\places'} m_v$ with $m_v$ nonnegative integer for all $v\in\places$.

Now by~\eqref{eq:Good-Remez} applied for $f_v$ implies that $|\mathsf B_v^{f_v, m_v}|\leq C 2^{-m_v/d}|\mathsf B_v|$ for all $v\in\places'$ and $m_v$.
The claim follows from this as the number partitions $m=\sum_{\places'} m_v$ is $\leq m^{\#\places'-1}$.

To see the second claim, let $v$ be so that $\max_{\mathsf B_v}f_v=\max F$. The claim then follows from
the fact that~\eqref{eq:Good-Remez} holds for $f_v$.
\end{proof}

Note that replacing $\frac{1}{d}$ with $\frac{1}{d}-\epsilon$, for a small enough $\epsilon$ depending only on $d$ 
and the constant $c$ by a bigger constant depending on $\places'$ if necessary, we have the following. 
There exists some $\alpha=\alpha(d)$ so that for all $F$ as in Lemma~\ref{l;polygood} we have
\be\label{eq:C-alpha-degree}
\Bigl|\Bigl\{ z \in{\mathsf B}:F(z)< \delta \sup_{{z}\in\mathsf{B}}F(z)\Bigr\}\Bigr|\leq c 
\delta^{\alpha}|{\mathsf B}|
\ee
where $c=c(r,d,\places').$

In the sequel, we will deal with functions defined on $U$ of the form $u\mapsto\cfun(\eta_H(\la_k(u)g))$ and $u\mapsto\|\zpz\wedge\eta_H(\la_k(u)g)\|$, see~\S\ref{sec:Hcal-diophatine} for the notation. 
We let $\alpha$ 
be so that~\eqref{eq:C-alpha-degree} holds true for all of these functions; note that $\alpha$ depends only on $N$.

\begin{lemma}\label{c;max-speed}\label{lem:U-speed}
There exists some $\consta\label{k:U-speed-1}\label{k:cor:U-moves}$ so that the following holds.
Let $\H\in\Hcal$ and $g\in G$. 
Put 
$
\epsilon_g=\max\{\|\zpz\wedge\eta_{M_H}(g)\|: \zpz\in\mathcal B_U\}.
$
Assume $\epsilon_g >0$, i.e.\ that $g^{-1}Ug$ does not normalize $H$.  
Then
\[
\Bigl|\Bigr\{u\in \mathsf B_U( e):\cfun\Bigl(\eta_H(\la_k(u)g)\Bigr)\leq R\Bigr\}\Bigr|\ll \Bigl(R|g|\height(\H)/\epsilon_g\Bigr)^{\ref{k:cor:U-moves}}e^{-k/\ref{k:cor:U-moves}}.
\]
\end{lemma} 



We need the following lemma for the proof of Lemma~\ref{lem:U-speed}.

\begin{lemma}\label{lem:dist-fixu}
There exists some $\consta\label{k:dist-fixu-1}$ so that the following holds.
Let the notation and assumptions be as in Lemma~\ref{lem:U-speed}. 
Moreover, let ${\rm Der}\rho_H$ denote the derivative of $\rho_H$.
Then 
\[
\max\bigl\{\|{\rm Der}\rho_H(\zpz)\eta_H(g)\|: \zpz\in\mathcal B_U\bigr\}\gg\delta
\]
where 
$\delta=\Bigl(\epsilon_g\height(\H)^{-1}|g|^{-1}\Bigr)^{\ref{k:dist-fixu-1}}$.
\end{lemma}

\begin{proof}
Let ${b}>0$ and assume that
$
\max\bigl\{\|{\rm Der}\rho_H(\zpz){}{\eta_H(g)}\|: \zpz\in\mathcal B_U\bigr\}\leq {b}.
$
Using~\eqref{eq:def-k-adjoint}, then we have
\be\label{eq:U-MH-Der-Close}
 \max\Bigl\{\|{\rm Der}\rho_H(\Ad(g^{-1})\zpz)\vpz_H\|: \zpz\in\mathcal B_U\Bigr\}\ll |g|^{\star}{b}.
\ee

Recall from the definition of ${\bf L}_{\H}$ that
\[
\Lie( L_{H})=\{w\in\gfrak: {\rm Der}\rho_H(w)\vpz_H=0\}.
\]
That is: $\Lie(L_{H})$ is the kernel of the linear map 
$w\mapsto{\rm Der}\rho_H(w)\vpz_H$ from $\gfrak$ to $V_H.$
The vector $\vpz_H$ is an integral vector of size $\height(\H)$.
Therefore, the map $w\mapsto{\rm Der}\rho_H(w)\vpz_H$ 
can be realized by an integral matrix whose entries are bounded by $\height(\H)^\star$. 

Now by~\eqref{eq:U-MH-Der-Close}, for all $\zpz\in\ufrak$ with $\|\zpz\|=1$ the vector $\Ad(g^{-1})\zpz$ almost belongs to the kernel of this map, in view of the above bound we get that
\be\label{eq:dist-etaHM-Der}
 \vsdist\bigl(\Ad(g^{-1})\zpz,\Lie(L_H)\bigr)\ll |g|^\star \height(\H)^\star{b}^\star,
\ee
see e.g.~\cite[\S13.4]{EMV}.

Recall that $\ufrak$ is a nilpotent Lie algebra and $\lplus_{\H}={\bf L}_{\H}^\Hcal$. 
Hence, arguing as in the proof of Lemma~\ref{l;lojas-ineq}, i.e.\ using Theorems A and B, 
we get the following from~\eqref{eq:dist-etaHM-Der}.
\[
 \vsdist(\Ad(g^{-1})\zpz,\Lie(\llplus_{ H}))\ll |g|^\star \height(\H)^\star{b}^\star\quad
 \text{for all $\zpz\in\mathcal B_U.$}
\]
The above estimate thus implies that
\[
\Bigl\|\zpz\wedge\eta_{\llplus_H}(g)\Bigr\|\ll |g|^\star r^\star{b}^\star\quad\text{for all $\zpz\in\mathcal B_U$;}
\]
as we wanted to show.
\end{proof}

\begin{proof}[Proof of Lemma~\ref{lem:U-speed}]
In view of Lemma~\ref{l;polygood}, it suffices to prove that
\be\label{eq:max-laku-etaH}
\max\{\cfun(\eta_H(\la_k(u)g)): u\in\mathsf B\}\gg \Bigl(\frac{\epsilon_g}{\height(\H)|g|}\Bigr)^\star e^{\star k}.
\ee

To see this, for any $\zpz\in\mathcal B_U$ define 
\[
f_{\zpz}(t)=\rho_H(\exp(t\zpz)){\eta_H(g)}.
\]
Then $f_{\zpz}$ is a polynomial map from $\Q_v$ into $V_H$.
Let us write $f_\zpz=c_{\zpz,0}+\hat f_{\zpz}$ where $c_{\zpz,0}\in V_H$
and $\hat f_{\zpz}(0)=0$.

Let $\delta$ be as in the previous lemma. By the conclusion of that lemma, there exists some 
$\zpz_0\in\mathcal B_U$ so that 
\be\label{eq:coef-est}
\max\{|c|_v:c\text{ is a coefficient of $\hat f_{\zpz_0}$}\}\gg \delta^\star.
\ee

For any nonzero $T\in\Q_v$, define the renormalized polynomial 
\[
{\hat f_{\zpz_0,T}}(t):=\tfrac{1}{T}\hat f_{\zpz_0}({T}t).
\]
Then by~\eqref{eq:coef-est}, we have 
$
\sup_{|t|_v\leq 1}\|{\hat f_{\zpz_0,T}}(t)\|\gg\delta^\star.
$

Hence, there exists some $v\in\places$ so that 
\[
\max\{u_v\in \mathsf B_{U_w}( e):\|\eta_H(\la_k(u_v)g)\|_v\}\gg \delta^\star e^{\star k};
\]
we also used the fact that for all $w\in\places$ we have
$
\|\eta_H(g)\|_w\gg |g|^{-\star}\height(\H)^{-\star};
$
this lower bound follows as $\vpz_H$ is an integral vector whose $\infty$-norm is $\height(\H)$.

Altogether, we get that
\begin{align*}
\max\{\cfun(\eta_H&(\la_k(u)g):u\in\mathsf B_{U}( e)\}\\
&\geq \max\{\|\eta_H(\la_k(u_v)g\|_v\textstyle\prod_{w\neq v}\|\eta_H(g)\|_w:u_v\in \mathsf B_{U_v}( e)\}\\
&\geq \delta^\star|g|^{-\star}\height(\H)^{-\star} e^{\star k};
\end{align*}
this completes the proof of~\eqref{eq:max-laku-etaH} and hence the lemma.
\end{proof}

%

\begin{propos}\label{prop:almost-inv-Lie}
There is a constant $D'$ depending only on $N$
so that the following holds. Let $\mathbf H \in \mathcal H$ and $r > 1$. 
Suppose that $k>1$ and
\begin{equation}\label{H almost fixed}
\cfun (\eta _ H (\lambda _ k (u) g))<r \qquad \text{for all $u \in \mathsf B _ U (e)$}
.\end{equation}
Then
\begin{equation*}
\cfun (\eta _ {M_H} (\lambda _ k (u) g))\ll r^\star \absolute g ^{\star} \qquad \text{for all $u \in \mathsf B _ U (e)$,}
\end{equation*}
moreover, for all $u \in \mathsf B _ U (e)$ and $\zpz \in \mathcal{B}_U$ we have
\begin{equation*}
\norm {\zpz \wedge \eta _ {M_H} (\lambda _ k (u) g)} \ll r^\star \absolute g ^{\star} e^{-k/D'}
.\end{equation*}
\end{propos}

\begin{proof}
Let $l = \dim (\mathbf H)$. Recall that $\vpz_H$ denotes the integer vector corresponding to $\Lie (\mathbf H)$ in  $\wedge ^ l \mathfrak g \subset \wedge ^ l \sl_N (\bbq _ \places)$ --- here and in what follows we view $\wedge ^ r \mathfrak g$ as a rational subspace of $\wedge ^ r \sl_N (\bbq _ \places)$ of height $\ll 1$ (recall from \S\ref{sec:notation-exponent} that the implicit constants for $\ll$ and $\gg$ are allowed to depend polynomialy on $\height (\G)$).

In the notations of \S\ref{sec:lambda-n}, let $\mathsf B=\mathsf B_U(e)$ 
and set
\[
\vartheta = 0.1 \frac {\absolute {\mathsf B} }{ \absolute {\lambda _ 1 (\mathsf B)}}.
\]
For any primitive $\bbz _ \Sigma$-submodule $\Delta$ of $g \bbz _ \Sigma ^ N$, it holds that $\cfun (\Delta) \gg \absolute g ^ {-\rank \Delta}$, hence
by Theorem~\ref{Kleinbock-Tomanov theorem} there exists a subset $\mathsf B_g \subset \la_k(\mathsf B)$
with 
\[
|\la_k(\mathsf B)\setminus \mathsf B_g|< \vartheta |\la_k(\mathsf B)|
\]
so that
\[
\alpha(ug \Gamma) \ll \absolute g^{\star}\qquad\text{ for  all $u \in\mathsf B_g$}.
\]
This implies that
for every $u \in\mathsf B_g$ there exists some $\gamma_u \in \SL_N (\bbz _ \places)$ so that
\be\label{eq:gamma-u-ht-prop}
|ug \gamma_u^{-1}|\ll \absolute g ^\star.
\ee
Now~\eqref{eq:gamma-u-ht-prop} and~\eqref{H almost fixed} imply that
\be\label{eq:gamma-u-ht-wpz}
\cfun(\gamma_u\vpz_H)\ll |ug \gamma_u^{-1}|^\star \cdot \cfun(ug\vpz_H) \ll \absolute g^{\star} r.
\ee
Applying a similar argument to the integral vector $\wpz\in\wedge ^{\dim\G} \sl_N (\bbq _ \places)$ 
corresponding to $\wedge^{\dim\G}\Lie(\G)$ and using the fact that
\[
\wpz=ug\wpz=ug\gamma_u^{-1}\gamma_u\wpz,
\] 
we have that $\Ad(\gamma_u) \mathfrak g$ is a rational subspace of $\sl_N(\bbq _ \places)$ 
of height $\ll \absolute g ^{\star}$.

Define 
$
{\bf L}'=\{g \in \SL_N(\bbq _ \places): g\vpz_H=\vpz_H\}.
$ 
It follows from~\eqref{H almost fixed} applied with $u=e$ that
$\cfun(\vpz_H)\ll |g|^\star r$; hence, $\height({\bf L}')\ll |g|^\star r$.

Moreover, the definitions imply that ${\bf L}_\H=\G\cap{\bf L}'$ 
and that ${\bf M}_\H={\bf L}_\H^\hcal$, see~\S\ref{sec:LH-MH}. Further, in view of~\eqref{eq:height-L-M} we have 
$\height({\bf M}_\H)\ll|g|^\star r^\star$. 


Similarly, for each $u\in\mathsf B_g$ define 
\[
{\bf L}'_u=\{g \in \SL_N(\bbq _ \places): g\gamma_u\vpz_H=\gamma_u\vpz_H\}=\gamma_u{\bf L}'\gamma_u^{-1};
\]
then $\height({\bf L}'_u)\ll |g|^\star r$.
Put ${\bf L}_u=\gamma_u{\bf L}_\H\gamma_u^{-1}$, and 
let ${\bf M}_u={\bf L}_u^\hcal=\gamma_u{\bf M}_{\bf H}\gamma_u^{-1}$. 
Then $\height({\bf M}_u)=\cfun(\gamma_u\vpz_{M_H})\ll |g|^\star r^\star$.

For every $u\in\mathsf B_g$ we have 
\be\label{eq:su-Mu-M-prop}
\cfun(\eta_{\llplus_H}(ug))\ll |ug\gamma_u^{-1}|^\star \cfun(\gamma_u\vpz_{M_H})\ll |g|^\star r^\star.
\ee
Since $u\mapsto\eta_{M_H}(ug\gamma)$ is a polynomial, the estimate 
in~\eqref{eq:su-Mu-M-prop} and Lemma~\ref{l;polygood} imply that 
\be\label{eq:cfun-su-Mu-prop}
\cfun(\eta_{\llplus_H}(\la_k(u)g))\ll |g|^{\star} r^\star\;\;\text{for all $u\in\mathsf B$.}
\ee
In particular, the first claim in the proposition holds.

We now turn to the proof of the second claim. 
Let $u\in\mathsf B_g \cap \la_{k-1}(\mathsf B)$. By the choice of $\vartheta$, this set has measure $\geq 0.9 \absolute{\la_{k-1}(\mathsf B)}$, in particular is nonempty. 
Let $\gamma_u\in\SL_N(\Z_\places)$ be as in~\eqref{eq:gamma-u-ht-prop}.

By \eqref{assumption on Greek lambda}, $\lambda _ {k -\kappa} (\mathsf B) \lambda _ {k -1} (\mathsf B) \subset \lambda _ k (\mathsf B)$; hence by \eqref{H almost fixed} we have 
\begin{equation*}
\cfun (\eta _ H (\lambda _ {k -\kappa} (v) u g )) < r \quad\text{for all $v\in\mathsf B$}.
\end{equation*}
Therefore, by Lemma~\ref{lem:U-speed} for every $u\in\mathsf B_g \cap \la_{k-1}(\mathsf B)$,
\begin{equation*}
\max\{\|\zpz\wedge\eta_{M_H}(ug)\|: \zpz\in\mathcal B_U\} \ll r \absolute {ug \gamma _ u ^{-1}} \cdot \height (\gamma _ u \mathbf H \gamma _ u ^{-1}) e ^ {-k/\ref{k:cor:U-moves} ^2}
.\end{equation*}
For $u \in \mathsf B _ g$, \ $\absolute {ug \gamma _ u ^{-1}} \ll \absolute g ^{\star}$ and $\height (\gamma _ u \mathbf H \gamma _ u ^{-1})  = \cfun (\gamma _ u \vpz_H) \ll \absolute g ^{\star} r$ hence for $u\in\mathsf B_g \cap \la_{k-1}(\mathsf B)$
\begin{equation*}
\max\{\|\zpz\wedge\eta_{M_H}(ug)\|: \zpz\in\mathcal B_U\} \ll \absolute g ^{\star} r ^{\star} e ^ {- k / \star}
.\end{equation*}
Since $u\mapsto\zpz\wedge{\eta_{\llplus_H}(\la_{k-1}(u)g)}$ is a polynomial, the above estimate together with Lemma~\ref{l;polygood} implies that
\[
\max\{\|\zpz\wedge\eta_{M_H}(\la_k(u)g)\|: \zpz\in\mathcal B_U\}
\ll |g|^{\star}r^{\star }e^{-{k}/\star}\;\;\text{for all $u\in\mathsf B$.}
\]

This finishes the proof of the second claim and the proposition.
\end{proof}

\section{Non-divergence of unipotent flows for general algebraic groups}\label{sec:non-div}
Consider now $\mathbf G$ a $\bbq$-group of class-$\mathcal{H}$ and $G = \mathbf G (\bbq _ \Sigma)$ as in \S\ref{sec:notation}. Let $d = \dim \mathbf G$. Recall that $\gfrak(\bbz_\places)=\gfrak\cap\mathfrak{sl}_N(\bbz_\places)$, see~\S\ref{sec:space-lattices}. Let $U=\prod_{v\in\places}U_v\subset G$ be a $\bbq _ \Sigma$-unipotent group as in \S\ref{sec:lambda-n}. By assumption $\mathbf G$ is equipped with an embedding $\iota: \mathbf G \to \SL _ N$, and a lattice $\Gamma$ commensurable to $ G \cap \SL _ N(\bbz_\Sigma)$.
Taking a finite index subgroup if necessary, we assume that $\Gamma < G \cap \SL _ N (\bbz _ \Sigma)$, and that $\Ad(\Gamma)$ preserves $\gfrak(\bbz_\places)$. Hence we get a finite to one map 
\[G / \Gamma \to \SL _ N (\bbq _ \Sigma) / \SL _ N (\bbz _ \Sigma),\]
or using the adjoint representation a different map 
\[G / \Gamma \to \SL _ d (\bbq _ \Sigma) / \SL _ d (\bbz _ \Sigma).\]
This latter map is in general not finite to one, but has compact fibers, since if $\mathbf C_\mathbf G$ denotes the connected component of the center of $G$ (necessarily a unipotent group as $\mathbf G$ is of class-$\mathcal{H}$) the fibers would be finite-to-one extension of the compact space $\mathbf C_\mathbf G(\bbq_\Sigma)/\mathbf C_\mathbf G(\bbz_\Sigma)$.  
 Therefore we may apply Theorem~\ref{Kleinbock-Tomanov theorem} to the image of $G / \Gamma$ to either of these quotient spaces to deduce that for every $g \in G$ and $\delta > 0$, there is a compact $K \subset G / \Gamma$ so that for every $k$ for all $u \in \mathsf B _ U (e)$ outside a set of measure $< \delta$ we have that $\lambda _ k (u) g \Gamma \in K$.

Theorem~\ref{Kleinbock-Tomanov theorem} gives more: it also says that for a compact set $K$ that does not depend on the point $g\Gamma$, if $\lambda _ k (u) g \Gamma \not \in K$ for a large set of $u \in \mathsf B _ U (e)$ then there would be a $\bbz _ \Sigma$-submodule in $g \SL _ N (\bbz _ \Sigma)$ which is not changed much by the action of $U$, at least not when we act by $\lambda _ k (\mathsf B _ U (e))$, and Theorem~\ref{fancy nondivergence theorem} gives somewhat finer information. However both of these theorems relate the properties of $U$ orbits in $G / \Gamma$ to the structure of the ambient $\SL_\bullet (\bbq _ \Sigma)/\SL_\bullet (\bbz _ \Sigma)$ and not some intrinsic algebraic structure of $G / \Gamma$.

In \cite{DM-Nondiv} Dani and Margulis prove (in the real case) that given a one parameter unipotent subgroup $u _ t$ of $G$ and $\delta > 0$, one can find an (fixed) compact subset $K \subset G / \Gamma$ so that if a trajectory of the one-parameter unipotent group $u _ t$ starting from $g \Gamma$ does not eventually spend $1 - \delta$ of its time in $K$ then there is a $\bbq$-parabolic subgroup  $\mathbf P < \mathbf G$ so that $g \in \mathbf P (\bbq _ \Sigma)$. This information is intrinsic for $G / \Gamma$.

The purpose of this section is to provide an effective version of \cite{DM-Nondiv}, where the existence of many  $u \in \mathsf B _ U (e)$ for which $\lambda _ k (u) g \Gamma$ is outside a suitable fixed compact region is used to imply some Diophantine conditions at appropriate scale for $g \Gamma$. We note that in addition to \cite{DM-Nondiv}, understanding intrinsically behavior of orbits near the cusp in arithmetic quotients $G / \Gamma$, this time for certain diagonalizable groups, was studied by Tomanov and Weiss in \cite{TW}.

Recall from \S\ref{sec:space-lattices} the definition 
\[
X_\eta=\Bigl\{g\Gamma\in X: \min_{0\neq \zpz\in\gfrak(\bbz_\places)}\cfun(\Ad(g)\zpz)\geq \eta\Bigr\}.
\]
It follows from the discussion at the beginning of this section that
for any $\eta > 0$, the set $X _ \eta$ is a compact subset of $G / \Gamma$.

\begin{lemma}\label{lem:in-cusp-small-par}
There exists some $0<\kappa(N,\places)<1$ with the following property.
Let $\wpz\in\gfrak(\zeds)$ and suppose that there exists some $g\in G$ so that $\cfun(\Ad(g)\wpz)\leq \kappa(N,\places)$. 
Then $\wpz$ is a nilpotent element.
\end{lemma}

\begin{proof}
Let $\bar\sigma(\wpz)=\prod\sigma$ where the product is taken over all the nonzero eigenvalues of $\wpz$;
if the product is empty, i.e., $\wpz$ is nilpotent, put $\bar\sigma(\wpz)=0$.
Then $\bar\sigma(\wpz)\in\bbq$ --- indeed $\bar\sigma(\wpz)$ is invariant under the Galois group of the splitting field of $\wpz$.
Further, since $\wpz\in\gfrak(\zeds)$, the product formula implies that
either $\cfun(\bar\sigma(\wpz))\geq 1$ or $\bar\sigma(\wpz)=0$.

Let $\kappa>0$ and assume that $\cfun(\Ad(g)\wpz)\leq \kappa$ for some $g\in G$. 
Using Lemma~\ref{lem:c-fun-norm}, there exist some $r\in\zeds^\times$ and a constant $A=A(N,\Sigma)$ 
so that 
\[
A^{-1}\cfun(\Ad(g)\wpz)^{1/\#\places}\leq \|r\Ad(g)\wpz\|_v\leq A\cfun(\Ad(g)\wpz)^{1/\#\places}
\] 
for all $v\in\places$. 
Therefore, all the eigenvalues of $r\Ad(g)\wpz$ have $v$-norm $\ll_{N,\places}\kappa^{\star/\#\places}$.

Since $\cfun(r)=1$ and $\Ad(g)\wpz$ has the same eigenvalues as $\wpz$, we get that $\cfun(\bar\sigma(\wpz))\geq 1$ cannot hold
for small enough $\kappa$; hence, $\wpz$ is nilpotent.   
\end{proof}

\begin{lemma}\label{lem:nilp-subalg}
There exists some 
$\kappa'(N,\places)$ with the following property. 
Let $V\subset\gfrak$ be a nonzero rational subspace, and let 
$\vpz\in\wedge\gfrak(\zeds)$ be a primitive integral vector corresponding to $V$.
Assume that there is some $g\in G$ so that 
\[\max_{u \in \mathsf B_U(e)}\cfun(\la_k(u)g\vpz)\leq \rho < \kappa'(N,\places).\] 
Then there exists a unipotent $\Q$-group $\mathbf W < \mathbf G$ so that
\begin{equation}\label{max size of unipotent group}
\max_{u \in \mathsf B_U(e)} \cfun(\la_k(u)g\vpz_W)\ll \rho^{\dim(W)/\dim(V)}
\end{equation}
where $\vpz_W\in\wedge^{\dim (\mathbf W)}\gfrak(\zeds)$ is the primitive integer vector corresponding to~$\mathbf W$ as in \S\ref{ss;height}.
\end{lemma}

\begin{proof}
Let $d = \dim (\mathbf G)$. We apply Theorem~\ref{fancy nondivergence theorem} on the image of $G / \Gamma$ in $\SL _d (\bbq _ \Sigma) / \SL _ d (\bbz _ \Sigma)$ obtained via $\Ad$, with $\tilde \Delta$ the $\bbz _ \Sigma$-submodule of $\Ad(g)\gfrak(\bbz_\places)$ corresponding to $V$ (or more precisely $gV(\bbq)$); let $r$ denote the dimension of the $\bbq$-subspace $V$ (equivalently, $r=\rank (\tilde \Delta)$).

Let $\eta (\bullet)$ and $\Delta _ {k_1} < \Delta _ {k_2} < \dots < \Delta _ {k_\ell}$ be as in that theorem. Then
$\eta (r) \leq \rho$ hence by concavity of $- \log \eta (\bullet)$ we have that $\eta (1) \leq \rho ^{1/ r}$. Let $s$ be maximal so that $\frac {\eta (s)}{\eta(s-1)} \leq \rho ^{1/ r}$; clearly $1 \leq r \leq d-1$, and because $- \log \eta (\bullet)$ changes its slope at $s$ this implies that there is some $1 \leq j \leq \ell$ for which $k_j = \rank (\Delta _ {k_j})=s$.

We claim that (assuming $\kappa ' (N,\places)$ is small enough) the rational subspace of $\mathfrak g$ corresponding to $\rank (\Delta _ {k_j})$ is the Lie algebra of a unipotent $\bbq$-subgroup $\mathbf W <\mathbf G$. Let us denote by $\vpz_W$ the vector in $\wedge^{\rank (\Delta _ {k_j})}\gfrak(\zeds)$ corresponding to this rational subspace. 
By the choice of $\Delta _ {k_j}$, we have
\[
\cfun(\la_k(u)g\vpz_W) \equiv \cfun(\la_k(u)\Delta _ {k_j}) = \eta(s) \leq \rho^{s/r}
\]
for all $u\in\mathsf B_U(e)$; so \eqref{max size of unipotent group} is satisfied. 

It remains to show that $\Delta _ {k_j}$ does indeed correspond to a rational nilpotent Lie algebra.
Fix an $\epsilon > 0$ (depending only on $N$) so that for a set of $u \in \mathsf B_U(e)$ of size $\geq 0.5 \absolute{\mathsf B_U(e)}$ we can find a completion $\Delta _ 1 < \dots < \Delta _ {N -1}$ (depending on $ u$) of the fixed partial flag $\Delta _ {k_1} <  \dots < \Delta _ {k_\ell}$ so that for every~$i$
\be\label{good flag 2}
\epsilon \eta (i) \leq \cfun (\la_k(u)\Delta _ i) \ll_{N,\places} \eta(i)
.\end{equation}
Indeed, we will only use the existence of one such $u$.

Using \eqref{basic sum-intersection inequality}, we can deduce from \eqref{good flag 2} that for every $v \in \la_k(u) \Delta _ {i+1}$ which is not in $\la_k(u) \Delta _ {i}$
\begin{equation*}
\cfun (v) \gg \frac {\cfun (\la_k(u) \Delta _ {i+1}) }{ \cfun (\la_k(u) \Delta _ {i})}\gg \frac {\eta (i+1)}{\eta (i)}
\end{equation*}
(since $\epsilon$ depends only on $N$, we absorbed it in the implicit constant). 
	
Moreover by induction one easily shows that we can pick $v_i \in \la_k(u) \Delta _ i$ (in particular, $v_i \in \Ad(\la_k(u)g)\gfrak (\bbz _ \Sigma)$) so that $v_1, \dots,v_i$ generate $\la_k(u) \Delta _ i$ and $\cfun (v_i) \ll \frac {\eta (i)}{\eta (i-1)}$. 

We conclude that there is some $A$ depending only  $N$ and $\places$ so that if $v \in \la_k(u) \Delta _ {i}$ but not in $\la_k(u) \Delta _ {i-1}$
then $\cfun (v) \geq \frac {\eta (i)}{A\eta (i-1)}$ and $\cfun (v_i) \leq \frac {A\eta (i)}{\eta (i-1)}$.

Recall that for $i \leq s$ we have that $\frac {\eta (i)}{\eta (i-1)} \leq \rho ^{1/ r}$.
As $\cfun ([z,z']) \ll \cfun (z) \cfun(z')$, it follows that if $\kappa ' (N,\places)$ (hence also $\rho$) is small enough, for $i < i' \leq s$ we have that $\cfun ([v_i,v_{i'}])$ is so small it forces $[v_i,v_{i'}]$ to belong to $\Delta _ {i-1}$. 

It follows that $\Delta _ {k_j}$ is closed under $[\cdot,\cdot]$. Since by Lemma~\ref{lem:in-cusp-small-par} if $\kappa ' (N,\places)$ is small enough, all the $v_i$ for $ i\leq s=k_j$ are nilpotent, it follows that all $v \in \Delta _ {k_j}$ are nilpotent. Hence $\Delta _ {k_j}$ corresponds to the Lie algebra of a unipotent $\Q$-subgroup of $\mathbf G$.
\end{proof}

\begin{thm}\label{thm:non-div-par}
There exists a constant $F$ depending on $N$ and a constant $E$ depending on $N, \#\Sigma$ and polynomialy on $\height(\G)$ and the primes in $\Sigma$
so that for any $g\in G$, $k\geq 1$, and any $0<\eta\leq 1/2$ at least one of the following holds.
\begin{enumerate}
\item 
\[
|\{u\in \mathsf B_U(e): \la_k(u )g\Gamma\not\in X_{\eta}\}|\leq E \eta^{1/F}.
\]

\item There exists a unipotent $\bbq$-subgroup ${\bf W}$ of height $\height(\mathbf W) \leq E |g|^{F} \eta^{1/F}$ so that
\be\label{eq:nondiv-v-almost-fixed-statement}
{\cfun\Bigl(\eta_W(\la_k(u)g)\Bigr)\leq E \eta^{1/F}\quad\text{ for all $u\in\mathsf B_U(e)$}}.
\ee
Moreover, if we put ${\bf M}={\bf M}_{\bf W}$, then ${\bf M}\neq \G$,
\[
\height(\lplus)\leq E |g|^{F}\eta^{1/F},
\] 
and we have: 
\begin{enumerate}
\item For all $u\in\mathsf B_U(e)$ we have
\[
\cfun(\eta_{M}(\la_{k}(u)g))\leq E |g|^F\eta^{1/F}. 
\]
\item For all $u\in\mathsf B_U(e)$ we have
\[
\max_{\zpz\in\mathcal B_U} \|\zpz\wedge{}{\eta_{M}(\la_{k}(u)g)}\|\leq E |g|^F\eta^{1/F}e^{-k/F}.
\]
\end{enumerate}

\end{enumerate}
\end{thm}

\begin{proof}
We may assume $\eta < \kappa ' (N,\places)$ with $\kappa ' (N,\places)$ as in Lemma~\ref{lem:nilp-subalg} since otherwise for sufficiently large implicit constant alternative (1) in the statement of this theorem becomes vacuous.

Apply Theorem~\ref{Kleinbock-Tomanov theorem}.
Then either alternative (1) in the theorem holds, or
there exists some primitive integral vector $\vpz \in \wedge^r\gfrak(\zeds)$ so that
\[\max_{u \in \mathsf B_U(e)}\cfun(\la_k(u)g\vpz)\leq \eta.\]

By Lemma~\ref{lem:nilp-subalg}, we conclude that there is some unipotent $\bbq$-group $\mathbf W < \mathbf G$ so that
\be\label{eq:nondiv-w-almost-fixed 2}
\max_{u \in \mathsf B_U(e)} \cfun(\la_k(u)g\vpz_W)\ll \eta^{\dim(W)/r}
\end{equation}

Applying~\eqref{eq:nondiv-w-almost-fixed 2} with $u=e$ we get that $\height({\bf W})\ll |g|^\star \eta^{\dim(W)/r}$.
Let $\bf M$ be as in (2) in the statement of this theorem. Then $\height({\bf M})\ll \height({\bf W})^\star \ll |g|^\star \eta^{\star}$, see~\eqref{eq:ht-H-Hcal}.
Moreover, if $\eta$ is small enough, then~\eqref{eq:nondiv-w-almost-fixed 2} (say for $u=e$)
implies that
\[
\cfun(g\vpz_W)<1/2.
\]
As $\cfun(\vpz_W)\geq 1$ this means that $g$ does not fix $\vpz_W$ and so (since $\mathbf G$ is of class-$\mathcal H$, hence fixes $\vpz_H$ for any normal subgroup $\mathbf H \lhd \mathbf G$) the group $\bf W$ is not a normal subgroup of $\G$. In particular, ${\bf M}\neq \G$.

Applying Proposition~\ref{prop:almost-inv-Lie} for $\H={\bf W}$, we have that parts (a) and (b) in (2) of the statement of the theorem hold, concluding the proof of this theorem.
\end{proof}

Theorem~\ref{thm:non-div-par} allows us to give a new, and arguably more elementary, proof to the main result of \cite{DM-Nondiv} (though the main ingredients are similar):

\begin{coro} \label{cor:DM} Suppose $\mathbf G$ is semisimple, and $g \in G$ is such that
\begin{equation}\label{enough to deduce in parabolic}
|\{u\in \mathsf B_U(e): \la_k(u )g\Gamma\not\in X_{\eta}\}|> E \eta^{1/F} \qquad \text{for infinitely many $k$.}
\end{equation}
 Then $U g \subset g\mathbf P (\bbq _ \Sigma)$ for some parabolic proper $\bbq$-subgroup of~$\mathbf G$.
\end{coro}


\begin{lemma}\label{lem:Jacob-Moro-Par}
Assume $\G$ is semisimple. 
Let $W\subset\gfrak$ be a rational subspace which generates a unipotent subalgebra. 
There exists a $\bbq$-parabolic subgroup ${\bf P}(W)$ so that $\height({\bf P}(W))\ll \height(W)^\star$ and 
$W\subset\Lie(\rad_u({\bf P}(W))$.
\end{lemma}

\begin{proof}
Let $\hat W$ be the algebra generated by $W$ and let $U_0=\{\exp(\hat W)\}$. 
Let ${\bf U}_0$ denote the corresponding algebraic group. 
Define inductively ${\bf U}_i=\rad_u(N_\G({\bf U}_{i-1}))$. 
Then ${\bf U}_i\subset N_{\G}({\bf U}_{i-1})$ and ${\bf U}_{i-1}\subset{\bf U}_i$.
This process terminates after $d\leq \dim\G$ number of steps and gives a unipotent 
subgroup ${\bf U}_d$ so that ${\bf U}_d=\rad_u(N_\G({\bf U}_d))$.
Therefore, $N_\G({\bf U}_d)$ is a parabolic subgroup, see~\cite{BT}; the claim holds with ${\bf P}(W)=N_\G({\bf U}_d)$.
\end{proof}

\begin{proof}[Proof of Corollary~\ref{cor:DM}] 
Suppose \eqref{enough to deduce in parabolic} holds along some sequence, say $\ell _ 1, \ell _ 2, \dots$ of $k$s (to avoid confusion with $k _ i$ of Theorem~\ref{fancy nondivergence theorem} we use $\ell$ rather than $k$).
Then by Theorem~\ref{thm:non-div-par} there exists for every $j$ a unipotent $\bbq$-subgroup $\mathbf W _ j$ and $\mathbf M_j = \mathbf M _ {\mathbf W _ j}$ with $\mathbf M _ j \ne \mathbf G$ so that the heights of $\mathbf W _ j$ and $\mathbf M _ j$ are bounded uniformly in $j$ and so that
\be \label{almost U invariant}
\max_{\zpz \in \mathcal B_U} \| \zpz \wedge{}{\eta_{M}(g)}\| \leq E |g|^F \eta^{1/F}e^{-\ell _ j/F}.
\end{equation}
since the only finitely many $\bbq$-subgroups of a given height, passing to a subsequence if necessary, we may assume that $\mathbf W _ j = \mathbf W$ and $\mathbf M _ j = \mathbf M$ for all $j$, hence from \eqref{almost U invariant} it follows that $\zpz \wedge{}{\eta_{M}(g)}=0$ for all $\zpz \in \mathcal B_U$, hence $U g \subset g\mathbf M (\bbq _ \Sigma)$. By Lemma~\ref{lem:Jacob-Moro-Par}, $\mathbf M$ is contained in some nontrivial $\bbq$-parabolic subgroup $\mathbf P < \mathbf G$ (indeed, with $\height(\mathbf P) \ll \absolute g^\star$).
\end{proof}

\section{Proof of Theorem~\ref{thm:eff-linearization-general}}\label{sec:proof-main}

For every $\Gamma$-invariant subcollection $\mathcal{F} \subset \mathcal H$, \ $t \in \bbr^+$,
 and $\vare: \bbr^+ \to (0, 1)$, let $\exceptional (\vare, t, \fcal)$ be the set
\begin{equation*}
\exceptional (\vare, t, \fcal) = \left\{ u \in \mathsf B _ U (e): \ \text {$\lambda _ k (u) g \Gamma$ is not $(\vare, t,\mathcal{F})$-Diophantine} \right\}
.\end{equation*}
For every $\epsilon \in (0, 1)$, we let $\underline \epsilon$ denote the constant function $\underline \epsilon:s \mapsto \epsilon$.
For every  $1 \leq r \leq \dim \mathbf G$, we let $\mathcal{F} _r$ denote the collection of class-$\hcal$ subgroups of $\Gbf$ of dimension $\leq r$.
For notational simplicity, let $\mathcal{F}_0=\emptyset$.

The bulk of the proof of Theorem~\ref{thm:eff-linearization-general} is the following estimate:

\begin{lemma}\label{main lemma} 
There are $\consta\label{first constant} $, $\consta\label{second constant}>2$, and $D\geq 1$ depending on $N$, and $0<c_0\leq 1$ depending on $N$, $\#\places$, and polynomially on the primes in $\places$ so that the following holds. For $1 \leq r \leq \dim \mathbf G$ and $\eta,\beta,\tau \in (0, 1)$, $n \in \bbr^+$  with 
\begin{equation}\label{condition for main lemma}
\eta^{1/\ref{second constant}} \leq c_0 \cdot (\tau\beta e^{-n} \tilde E_\Gbf^{-1})^{\ref{first constant}}
\end{equation}
at least one of the following holds:
\begin{enumerate}
\item 
$
\absolute {\left (\exceptional (\underline \eta,n,\mathcal{F} _ r )  \setminus \exceptional (\underline\eta',n',\mathcal{F} _ {r-1})\right )\cap \{ u \in \mathsf B _ U (e): \lambda _ k (u) g \Gamma\in X_\tau\}}
 \ll \beta^{1/D}
$
\ for
\[n'= \ref{first constant} (n + \log (1/\tau)+ \log(1/\beta)+\log \tilde E_\Gbf) \qquad\eta' = e^{n'} \eta^{1/ \ref{second constant}}, \ \text{or}\]

\item
 for some $\mathbf H $ of dimension $r$ 
\begin{enumerate}
\item For all $u\in\mathsf B_U(e)$ we have 
\[
\cfun(\eta_{H}(\la_k(u)g))\leq \beta^{-1}e^{n}. 
\]
\item For all $u\in\mathsf B_U(e)$ we have  
\[
\max_{\zpz\in\mathcal B_U} \Bigl\|\zpz\wedge \eta_{H}(\la_{k}(u)g)\Bigr\|\leq \eta^{1/2}.
\]
\end{enumerate}
\end{enumerate}
\end{lemma}

\begin{proof}
Recall from \S\ref{sec:lambda-n} the definition
\[
\mathsf B_\ufrak(0,\delta)=\Bigl\{\sum_{\zpz\in\mathcal B_U} r_\zpz\zpz: |r_\zpz|_\places\leq \delta\Bigr\}\]
so that
$
\mathsf B_U(e):=\exp\Bigl(\mathsf B_\ufrak(0,1)\Bigr)$.

\medskip

We will cover the set
\begin{equation*}
\log (\exceptional (\underline \eta,n,\mathcal{F} _ r )) \subset \mathsf B = \mathsf B_\ufrak(0,1)
\end{equation*}
by a collection of balls $\mathcal{E} = \left\{ \mathsf B _ i =\mathsf B_\ufrak(0,\rho_i)+ \upz _ i \right\} _ {i \in I}$ and for each such ball attach a class-$\mathcal H$ group $\mathbf H _ i\in \mathcal{F} _ r $ so that
\begin{enumerate}
\item[($ \mathcal{E}1$)] $\sum_ {i \in I} \absolute {\mathsf B _ i} \ll 1$.

\item [($ \mathcal{E}2$)] for every $\upz \in \mathsf B_i$ and $u = \exp(\upz)$
\begin{align}
\cfun(\eta_{H_i}(\la_k(u)g))&\leq \beta^{-1}e^{n} \label{half good estimate 1}\\
\max_{\mathcal B_U}\|\zpz\wedge{\eta_{H_i}(\la_k(u)g)}\|&\leq \eta^{1/2}
 \label{half good estimate 2}
\end{align}
\item [($ \mathcal{E}3$)] For every $i \in I$, for some $\upz \in \mathsf B _ i$, equality holds in at least one of \eqref{half good estimate 1} or \eqref{half good estimate 2}.
\end{enumerate}
\medskip

\noindent
More precisely, we will \emph{try} to construct a cover $\mathcal E$ with these properties, and \emph{if} we fail this will establish that part (2) of Lemma~\ref{main lemma} holds.

Assuming we succeed, we will show that these properties imply, for a suitable choice of constants $\ref{first constant},\ref{second constant},\kappa$ that for $n', \eta'$ as above 
\begin{multline} \label{exceptional inclusion equation}
\left (\exceptional (\underline \eta,n,\mathcal{F} _ r )  \setminus \exceptional (\underline\eta',n',\mathcal{F} _ {r-1})\right )\cap \{ u \in \mathsf B _ U (e): \lambda _ k (u) g \Gamma\in X_\tau\} \subset\\
 \bigcup_ {i \in I} \left\{ u \in \exp (\mathsf B _ i): \text {$\lambda _ k (u) g$ satisfies $\eqref{half good estimate 1}'$ and $\eqref{half good estimate 2}'$} \right\}
\end{multline}
where $\eqref{half good estimate 1}'$ and $\eqref{half good estimate 2}'$ denote inequalities \eqref{half good estimate 1} and \eqref{half good estimate 2} but with $e^n$ and $\eta$ instead of $\beta^{-1}e^{n}$ and $\eta^{1/2}$, respectively.
Once~\eqref{exceptional inclusion equation} is established, we get from Lemma \ref{l;polygood} that
\begin{multline*}
\absolute {\left (\exceptional (\underline \eta,n,\mathcal{F} _ r )  \setminus \exceptional (\underline\eta',n',\mathcal{F} _ {r-1})\right )\cap \{ u \in \mathsf B _ U (e): \lambda _ k (u) g \Gamma\in X_\tau\}}\\ 
\ll \max(\eta^{1/2},\beta)^{1/D}.
\end{multline*} 
The desired estimate in part~(1) follows from this bound and~\eqref{condition for main lemma}.

The construction of the open cover is straightforward. For every $\upz \in \mathsf B_\ufrak(0,1)$ for which $u=\exp (\upz) \in \exceptional (\underline \eta,n,\mathcal{F} _ r ) $, there is (by definition) a $\bbq$-group $\mathbf H _ {\upz} \in \mathcal H$ of dimension $\leq r$  so that $\eqref{half good estimate 1}'$ and $\eqref{half good estimate 2}'$ holds (with $\mathbf H _ {\upz}$ replacing the yet undefined $\mathbf H _ i$).

For each such $\upz$, let $\mathsf B ( {\upz})$ denote the set
$\mathsf B ( {\upz})=\mathsf B_\ufrak(0,\rho_{\upz})+ \upz$
with $\rho _ {\upz}$ chosen to be as small as possible so that for some
$\vpz \in \mathsf B ( {\upz})$ either $\cfun(\eta_{H_\upz}(\la_k(\exp \vpz)g))= \beta^{-1}e^{n}$ or
$
\max_{\mathcal B_U}\norm{\zpz\wedge\eta_{H_\upz}(\la_k(\exp\vpz)g)} =\eta^{1/2} $. 
Unless the estimates~(2)(a) and~(2)(b) of the statement of this theorem holds for some $r$-dimensional $\mathbf H \in \mathcal H$, for any $\upz \in \mathsf B_\ufrak(0,1)$ it holds that $\mathsf B( {\upz}) \subset \mathsf B_\ufrak(0,3)$. Note that the estimate of (2)(a) together with \eqref{eq:def-k-adjoint} gives that $\height(\mathbf H)\ll |g|^\star e^{2n}$.

Assuming there is no such $\mathbf H$, the Vitali covering argument allows us to find a subcollection $\mathcal E=\left\{ \mathsf B _ i = \mathsf B_\ufrak(0,\rho_{i})+ \upz_i \right\} _ {i \in I}$ of $\left\{ \mathsf B ( {\upz}) \right\}$ so that the collection of smaller balls $\left\{ \mathsf B_\ufrak(0,\rho_{i}/3)+ \upz_i \right\} _ {i \in I}$ is a disjoint collection of subsets of $\mathsf B_\ufrak(0,3)$ but
\begin{equation}\label{union contains exceptional set}
\bigcup_ {i \in I}\mathsf B _ i \supset \log (\exceptional (\underline \eta,n,\mathcal{F} _ r ))
.\end{equation}
The resulting collection $\mathcal E $ clearly satisfies ($ \mathcal{E}1$)--($ \mathcal{E}3$).

\medskip

It remains to establish \eqref{exceptional inclusion equation}. Fix some 
\[
u \in \exceptional (\underline \eta,n,\mathcal{F} _ r ) \cap \{ u \in \mathsf B _ U (e): \lambda _ k (u) g \Gamma\in X_\tau\}.
\]
Then by \eqref{union contains exceptional set} there is an $i \in I$ so that $u \in \exp (\mathsf B _ i)$; put $\H_i=\H_{u_i}$. By the definition of $\mathsf B _ i$  estimates \eqref{half good estimate 1} and \eqref{half good estimate 2} hold, while by the definition of $\exceptional (\underline \eta,n,\mathcal{F} _ r )$ there is an $\mathbf H \in \mathcal H$ of dimension $\leq r$ so that
$\cfun(\eta_{H}(\la_k(u)g))\leq e^{n}$ and $\max_{\mathcal B_U}\| \zpz \wedge{\eta_{H}(\la_k(u)g)}\|\leq \eta$.
There are now two possibilities: either $\mathbf H = \mathbf H _ i$, in which case $u$ is contained in the set on the right hand side of \eqref{exceptional inclusion equation}, or $\mathbf H \neq \mathbf H _ i$.

Thus suppose that $\mathbf H \neq \mathbf H _ i$. 
By Lemma~\ref{prop:eff-red-theory} there is a $\gamma \in \Gamma$ so that $| \lambda _ k (u) g\gamma |\leq E_\Gbf \tau^{-F}$. Since $\eta _ {\gamma^{-1} \bullet \gamma }(\la_k(u)g\gamma) = \eta _ \bullet (\la_k(u)g)$, we have
\begin{equation*}
\begin{aligned}
\cfun(\eta_{L}(\la_k(u)g\gamma))&\leq \beta^{-1}e^{n} \\
\max_{\mathcal B_U}\| \zpz \wedge{\eta_{L}(\la_k(u)g\gamma)}\|&\leq \eta^{1/2} 
\end{aligned}
\qquad\text{for $\mathbf L=\gamma^{-1} \Hbf \gamma,\gamma^{-1} \Hbf_i \gamma $}.
\end{equation*}
Since $\Hbf \neq \Hbf_i$, we have that 
\[
\tilde {\mathbf H}= (\gamma ^{-1} \mathbf H \gamma \cap \gamma ^{-1} \mathbf H _ i \gamma) ^ {\mathcal H}
\] 
is of dimension $\leq r -1$. Applying Lemma~\ref{l;lojas-ineq} with $\epsilon = \eta ^ {1/2}$ and $r= \beta^{-1}e ^ { n}$ we get that if $\eta \ll E_\Gbf ^{\star} \beta^{-\star}\tau^{-\star} e^{\star n }$, the group $\tilde {\mathbf H}$ is nontrivial and
\begin{align*}
\cfun(\eta_{\tilde H}(\la_k(u)g \gamma))&\ll (E_\Gbf\beta^{-1}\tau^{-1} e^{n})^{\ref{first constant}} \\
\max_{\mathcal B_U}\| \zpz \wedge{\eta_{\tilde H}(\la_k(u)g \gamma)}\|&\ll (E_\Gbf\beta^{-1}\tau^{-1} e^{n})^{\ref{first constant}} \eta ^ {1/ \ref{second constant}}.
\end{align*}
Therefore (in view of our convention regarding implicit constants, and recalling that $\tilde E_\Gbf = E_\Gbf \height(\Gbf)$) we have that $u \in \exceptional (\underline{\eta'}, n',\mathcal F_{r-1})$ for $n'= \ref{first constant} (n + \log (1/\beta)+ \log (1/\tau)+ \log \tilde E_\Gbf)$ and $\eta' = e^{n'} \eta^{1/ \ref{second constant}}$.
\end{proof}

\begin{proof}[Proof of Theorem~\ref{thm:eff-linearization-general}]
We need to estimate the size of the set
\[
\Bad = \biggl\{u\in\mathsf B_U(e):\la_k(u)g\Gamma\not\in X_\eta\ \text{ or is not $(\vare,t)$-Diophantine}\biggr\}.
\]
Set 
\begin{align*}
\Bad _ {\eta} &= \left\{ u\in\mathsf B_U(e):\la_k(u)g\Gamma\not\in X_\eta \right\}\\
\Bad _ {\exceptional} &= \left\{ u\in\mathsf B_U(e):\text{$\la_k(u)g\Gamma$ is not $(\vare,t)$-Diophantine but is in $X_\eta$} \right\}
.\end{align*}
Clearly $\Bad = \Bad _ \eta \cup \Bad _ {\exceptional}$.

\medskip

We can bound the size of $\Bad _ \eta$ using Theorem~\ref{thm:non-div-par}, obtaining that
\be \label{size of bad set sub Greek eta}
\absolute {\Bad _ \eta} \leq E \eta ^ {1/F}
\end{equation}
unless there exists a group ${\bf M} \neq \G$,
\[
\height(\lplus)\leq E |g|^{F}\eta^{1/F},
\]
so that for all $u \in \mathsf B_U(e)$
\begin{align*}
\cfun(\eta_{M}(\la_{k}(u)g))&\leq E |g|^F \eta^{1/F}\\
\max_{\zpz \in \mathcal B_U} \| \zpz \wedge{}{\eta_{M}(\la_{k-1}(u)g)}\| &\leq E |g|^F \eta^{1/F}e^{-k/F}.
\end{align*}
This clearly implies that (2) of Theorem~\ref{thm:eff-linearization-general} holds (if we choose $A$ large enough).

\medskip

Assume therefore for the remainder of the proof that \eqref{size of bad set sub Greek eta} holds.
Let $d = \dim\Gbf$.
It follows from Definition~\ref{def:Diophantine-intro} that in the notations of Lemma~\ref{main lemma}
\begin{multline}\label{bad is contained equation}
\Bad _ {\exceptional} \subset \\\left (\bigcup_ {\ell=1}^{\lceil t \rceil} \exceptional \left (\underline{\vare (e^{\ell})},\ell+1, \mathcal{F} _ {d-1}\right ) \cup \exceptional \left (\underline{\vare (0)},1, \mathcal{F} _ {d-1}\right )\right ) \cap \Bigl(\mathsf B_U(e)\setminus \Bad _\eta\Bigr)
.\end{multline}
Fix $\ell$. To estimate $\absolute{\exceptional \left (\underline{\vare (e^{\ell})},\ell+1, \mathcal{F} _ {d-1}\right )\cap \Bigl(\mathsf B_U(e)\setminus \Bad _\eta\Bigr)}$, define iteratively, starting with $r=d-1$, \ 
$n_{d-1}=\ell +1$ and $\eta _ {d-1}=\vare (e ^ \ell)$.
Proceed by induction to define
\begin{equation*}
n_{r-1}=\ref{first constant} (n_r +\ell+2\log (1/\eta)+ \log \tilde E_\Gbf) \qquad\eta_{r-1} = e^{n_{r-1}} \eta_r^{1/ \ref{second constant}}
.\end{equation*}
Then
\begin{multline}\label{excluded sum equation}
\absolute {\exceptional \left (\underline{\vare (e^{\ell})},\ell+1, \mathcal{F} _ {\dim \mathbf G}\right )\cap \Bigl(\mathsf B_U(e)\setminus \Bad _\eta\Bigr)} \\
\leq \sum_ {r=1}^ {d-1} \absolute { \left (\exceptional(\underline{\eta_r},n_r, \mathcal{F} _ {r}) \setminus (\exceptional(\underline{\eta_{r-1}},n_{r-1}, \mathcal{F} _ {r-1})\right )\cap \Bigl(\mathsf B_U(e)\setminus \Bad _\eta\Bigr)}
.\end{multline}

We want to apply Lemma~\ref{main lemma} (with $\tau=\eta$  and $\beta=\eta e^{-\ell}$ where $\eta$ is as in this theorem). In order to apply Lemma~\ref{main lemma}, our choice of parameters needs to satisfy condition~\eqref{condition for main lemma}, with the critical case being that of $r=1$. In this case \eqref{condition for main lemma} becomes
\begin{equation}
\label{nuisance equation}
\eta_1^{1/\ref{second constant}} \leq c_0 \cdot (\eta^2 e^{-\ell} e^{-n_1} \tilde E_\Gbf^{-1})^{\ref{first constant}}.
\end{equation}
Iteratively working through the constants, there are $\consta\label{first' constant}>\ref {first constant}$, $\consta\label{second' constant}>\ref{second constant}$, easily explicated in terms of $d$ and $\ref{first constant}$ and $\ref{second constant}$ so that
\begin{equation*}
e^{n_1} \leq \left (\frac {2 \tilde E _ {\mathbf G} e^\ell }{ \eta^2 e^{-\ell} }\right) ^ {\ref{first' constant}}
\qquad
\eta_1 \leq e ^ {2 n _ 1} \vare(e^ \ell) ^ {1/\ref{second' constant}}
.\end{equation*}
Then assuming \eqref{condition on vare} with $A$ large enough and suitable choice of constant $E _ 1$ we can ensure that $\eta _ 1 \leq \vare (e ^ \ell) ^ {1/2 \ref{second' constant}}$ and that \eqref{nuisance equation} holds.

By Lemma~\ref{main lemma} and \eqref{condition on vare} (for $A$ large enough), for every $r$
\begin{multline}\label{bound using main lemma}
\absolute { \left (\exceptional(\underline{\eta_r},n_r, \mathcal{F} _ {r}) \setminus (\exceptional(\underline{\eta_{r-1}},n_{r-1}, \mathcal{F} _ {r-1})\right )\cap \Bigl(\mathsf B_U(e)\setminus \Bad _\eta\Bigr)} \\\ll \eta^{1/D}e^{-\ell/D}
\end{multline}
unless for some $\mathbf H $ of dimension $r$, for all $u \in \mathsf B_U(e)$ and some constant $F$ depending only on $N$ and $E_1'$ depending polynomially on $\tilde E_\Gbf$ and on $N$,
\begin{align}\label{exceptional case in proof 1}
\cfun(\eta_{H}(\la_k(u)g))&\leq \eta^{-1}e^{\ell+n_r}\leq E_1' e^{F \ell}\eta^{- F}
\\
\label{exceptional case in proof 2}
\max_{\zpz \in \mathcal B_U} \Bigl\| \zpz \wedge \eta_{H}(\la_{k}(u)g)\Bigr\|&\leq  \eta_r^{1/2}
\leq E_1' e^{F \ell}\eta^{- F} \vare  (e ^ \ell) ^ {1/F} .
\end{align}

\medskip

If equations \eqref{exceptional case in proof 1} and \eqref{exceptional case in proof 2} hold for all $u \in \mathsf B _ U (e)$, there are two cases. Firstly, it may happen that $\mathbf H \lhd \mathbf G$ in which case 
\[\height (\mathbf H) = \cfun(\eta_{H}(\la_k(u)g)) \leq E_1' e^{F \ell}\eta^{- F} .
\]
Then as we assumed $\vare(s) \leq  \eta^A s^{-A} /E_1$ (for the constants $A$ and $E_1$ of the theorem we are proving, which are yet to be fixed) if $A$ was chosen large enough, by \eqref{exceptional case in proof 2}
\begin{equation}\label{exceptional case in proof 2'}
\max_{\zpz \in \mathcal B_U} \Bigl\| \zpz \wedge \eta_{H}(\la_{k}(u)g)\Bigr\| \leq \vare  (e ^ \ell) ^ {1/2F}
.\end{equation}
For a given $\Hbf$, the value of $\ell$ has to be large enough so that \eqref{exceptional case in proof 1} holds, namely 
\begin{equation*}
e^\ell \geq (\height (\H)\eta^{F}/E_1')^{1/F}
,\end{equation*}
so 
\begin{equation*}
\max_{\zpz \in \mathcal B_U} \Bigl\| \zpz \wedge \eta_{H}(\la_{k}(u)g)\Bigr\| \leq  \vare  (\height (\H)^{1/F}\eta/ E_1') ^ {1/2F}
\end{equation*}
and (3) of the statement of Theorem~\ref{thm:eff-linearization-general} is satisfied.

\medskip
If equations \eqref{exceptional case in proof 1} and \eqref{exceptional case in proof 2} hold for all $u \in \mathsf B _ U (e)$, but $\mathbf H$ is not a normal  subgroup of $\mathbf G$ we apply Proposition~\ref{prop:almost-inv-Lie} and conclude that $\mathbf M = \mathbf M _ {\mathbf H}$ satisfies for all $u \in \mathsf B _ U (e)$
\begin{equation}\label{third case bound}
\begin{aligned}
\cfun (\eta _ {M} (\la _ k (u) g)) &
\ll {E_\Gbf}^{\star} e^{\star \ell} \eta^{-\star}\\
\max_ {\zpz \in \mathcal{B}_U} \norm {\zpz \wedge \eta _ {M} (\la _ k (u) g)}&  \ll {E_\Gbf}^{\star} e^{\star \ell} \eta^{-\star} e^{-k/D'}
.\end{aligned}
\end{equation}
(the dependence of the upper bounds in \eqref{third case bound} on $|g|$ can be eliminated as in the proof of Proposition~\ref{prop:almost-inv-Lie} by using the fact that by~\eqref{size of bad set sub Greek eta} for most $u \in \mathsf B_U(e)$ there is a $\gamma_u \in \Gamma$ so that $\absolute {\la_k(u)g\gamma_u} \ll\eta^{-\star}$).
In this case, (2) of Theorem~\ref{thm:eff-linearization-general} holds.

\medskip 

The only remaining case is if \eqref{bound using main lemma} holds for every $r$ and $\ell$ (as well as the analogous estimates for $\exceptional (\underline {\vare (0)}, 1, \mathcal{F} _ {d -1})$, for which we omit the details, but is handled similarly), in which case it follows from~\eqref{bad is contained equation} and~\eqref{excluded sum equation} that $\absolute {\Bad _ {\exceptional}}  \ll \eta ^ {1/D}$, establishing (1) of Theorem~\ref{thm:eff-linearization-general}.
\end{proof}

Let us record the following corollary of the proofs of Lemma~\ref{main lemma} and Theorem~\ref{thm:eff-linearization-general}. 

\begin{coro}\label{cor: main lemma}
Let the notation be as in Theorem~\ref{thm:eff-linearization-general}. In particular, $A, D$ and $E_1$ are as in loc.\ cit. Let $g\in G$, $t>0$, $k\geq 1$, and $0<\eta<1/2$. Assume 
\[
0<\vare \leq  \eta^A e^{-tA} /E_1.
\]
Suppose there exists 
\[
\exceptional_r\subset \{ u \in \mathsf B _ U (e): \lambda _ k (u) g \Gamma\in X_\eta\}
\]
with $|\exceptional_r|>E_1\eta^{1/D}$ so that 
and for every $u\in\exceptional_r$ there is a $\bbq$-group $\mathbf H_u \in \mathcal H$ of dimension $\leq r$  satisfying both of the following 
\begin{align*}
\cfun(\eta_{H_u}(\la_k(u)g))&\leq e^{t} \\
\max_{\mathcal B_U}\|\zpz\wedge{\eta_{H_u}(\la_k(u)g)}\|&\leq \vare.
\end{align*} 

Then Theorem~\ref{thm:eff-linearization-general}(2) holds with a subgroup $\H\in\mathcal H$ which is also contained in $\gamma{\bf H}_u\gamma^{-1}$ for some $u\in\exceptional_r$ and some $\gamma\in\Gamma$.  
\end{coro}

\section{Some corollaries of Theorem~\ref{thm:eff-linearization-general}}\label{sec:cor-main}

In this section we discuss some of the consequences of Theorem~\ref{thm:eff-linearization-general}.
Recall that for any $T>0$, we put
\be\label{eq:def-sigmaT}
\sigma(T)=\min\left (\{1\}\cup \biggl\{\|\zpz\wedge\vpz_H \|:\begin{array}{c} \Hbf\in\Hcal, \Hbf\lhd\Gbf,\\  
\height(\H)\leq T, \{1\}\neq\Hbf\neq \G\end{array}\biggr\}\right ).
\ee

\begin{thm}\label{thm:linearization-noneff-eff-S-arith}
Let $A$, $D$, and $E_1$ be as in Theorem~\ref{thm:eff-linearization-general}. 
There exists some $\vartheta$ depending only on $N$ so that the following holds.
Let $0<\eta<1/2$. Let $\{x_m\}$ be a sequence in $X$ and let $k_m\to\infty$ be a sequence of natural numbers. 
For each $m$ let $\mathsf V_m\subset\mathsf B_U(e)$ be a measurable set with measure $>\vartheta E_1\eta^{1/D}$.  
Let
\[
Y=\bigcap_{\ell\geq 1} \overline{\bigcup_{m\geq \ell} \{\la_{k_m}(u)x_m: u\in \mathsf V_m\}}.
\]
Then exactly one of the following holds.
\begin{enumerate}
\item $Y$ contains an $\vare$-Diophantine point for 
\[
\vare(s)=\Bigl(\eta s^{-1} \sigma(E_1^A\eta^{-A}s^A)/2E_1\Bigr)^A.
\]
\item There exists 
\begin{enumerate}
\item a finite collection  $\{(\H_i, L_i): 1\leq i \leq \ell\} \subset \hcal\times\R^+$, and 
\item a countable (possibly finite) collection
\[
\mathcal W=\{({\bf W}_j, R_j, r_j): j\in J\}\subset \hcal\times \R^+\times \bbr^+
\]
where ${\bf W}_j$ is a non-normal unipotent subgroup for all $j\in J$, and $r_j\to0$
\end{enumerate}
so that if we put
\[
Y_i = \{g \in N_G(U,H_i):\cfun(\eta_{H_i}(g))\leq L_i\}\Gamma/\Gamma
\]
and
\[
Z_j=\biggl\{g \in N_G(U,M_{W_j}):\begin{array}{c}\cfun(\eta_{M_{W_j}}(g))\leq R_j\; \&\\ \cfun(\eta_{W_j}(g))\leq r_j\end{array}\biggr\}\Gamma/\Gamma,
\]
then 
\[
Y \subset \biggl(\bigcup_{i =1}^\ell Y_i\biggr) \bigcup \biggl(\bigcup_{j \in J} Z_j\biggr).
\]
\end{enumerate}
\end{thm}

\medskip
\noindent
This theorem implies Theorem~\ref{thm:linearization-noneff-eff} since each $Z_j$ is contained in 
\[
\bigl\{g \in N_G(U,M_{W_j}):\cfun(\eta_{M_{W_j}}(g))\leq R_j\bigr\}\,\Gamma/\Gamma,
\]
and as $r_j \to 0$ for any $\beta$ only finitely many of the $Z_j$ can interset $X_\beta$.
Recall that for every $i$ the sets $Y_i$ above are closed subsets of $X$ (see Corollary~\ref{lem:tube-closed}), and the same proof gives that the sets $Z_j$ are closed as well.

\medskip

We first prove a special case of Theorem~\ref{thm:linearization-noneff-eff-S-arith}.

\begin{lemma}\label{lem:linearization-noneff-eff-S-arith-t-fixed}
Let $A$, $D$, and $E_1$ be as in Theorem~\ref{thm:eff-linearization-general}. 
There exists some $\vartheta$ depending only on $N$ so that the following holds.
Let $0<\eta<1/2$ and let $t\in\bbr^+$. Let $\{x_m\}$ be a sequence in $X$ and let $k_m\to\infty$ be a sequence of natural numbers. 
For each $m$ let $\mathsf V_m\subset\mathsf B_U(e)$ be a measurable set with measure $>\vartheta E_1\eta^{1/D}$.  
Let
\[
Y=\bigcap_{\ell\geq 1} \overline{\bigcup_{m\geq \ell} \{\la_{k_m}(u)x_m: u\in \mathsf V_m\}}.
\]
Then at least one of the following holds.
\begin{enumerate}
\item $Y\cap X_\eta$ contains an $(\vare,t)$-Diophantine point for 
\[
\vare(s)=\Bigl(\eta s^{-1} \sigma(E_1^A\eta^{-A}s^A)/2E_1\Bigr)^A.
\]
\item There exists 
\begin{enumerate}
\item a finite collection  $\{(\H_i, L_i): 1\leq i \leq \ell\} \subset \hcal\times\R^+$, and 
\item a countable (possibly finite) collection
\[
\mathcal W=\{({\bf W}_j, R_j, r_j): j\in J\}\subset \hcal\times \R^+\times \bbr^+
\]
where ${\bf W}_j$ is a non-normal unipotent subgroup for all $j\in J$, and $r_j\to0$
\end{enumerate}
so that if we put
\[
Y_i = \{g \in N_G(U,H_i):\cfun(\eta_{H_i}(g))\leq L_i\}\Gamma/\Gamma
\]
and
\[
Z_j=\biggl\{g \in N_G(U,M_{W_j}):\begin{array}{c}\cfun(\eta_{M_{W_j}}(g))\leq R_j\; \&\\ \cfun(\eta_{W_j}(g))\leq r_j\end{array}\biggr\}\Gamma/\Gamma,
\]
then 
\[
Y \subset \biggl(\bigcup_{i =1}^\ell Y_i\biggr) \bigcup \biggl(\bigcup_{j \in J} Z_j\biggr).
\]
\end{enumerate}
\end{lemma}

\medskip

We need the following lemma.

\begin{lemma}\label{lem:cusp-or-not-limsup}
Let $E$ and $F$ be as in Theorem~\ref{thm:non-div-par}. 
Let the notation be as in Theorem~\ref{thm:linearization-noneff-eff-S-arith}.
Then one of the following holds.
\begin{enumerate}
\item There exists some $\beta_0>0$ and subsequence $m_i\to\infty$ so that
\[
|\{u\in\mathsf B_U(e):\la_{k_{m_i}-1}(u)x_{m_i}\not\in X_\beta\}|\leq E\beta^{1/F}.
\]
for all $\beta\leq\beta_0$, or
\item $Y=\emptyset$.
\end{enumerate}
\end{lemma}

\begin{proof}
In view of Theorem~\ref{thm:non-div-par} it suffices to show that there exists some $\beta'$ and a subsequence $\{m_i\}$
so that
\[
\{u\in\mathsf B_U(e):\la_{k_{m_i}-1}(u)x_{m_i}\}\cap X_{\beta'}\neq\emptyset.
\]
Indeed if this is established, then Theorem~\ref{thm:non-div-par}(2) cannot hold for any $\eta=\beta\leq \beta'^F/2E$, $x_{m_i}$ and $k_{m_i}-1$, hence Theorem~\ref{thm:non-div-par}(1) holds which implies part~(1) in this lemma with $\beta_0=\beta'^F/2E$.

Assume contrary to the above claim that for every $\beta$ there exists some $m_\beta$ so that for all $m\geq m_\beta$ 
\[
\{u\in\mathsf B_U(e):\la_{k_m-1}(u)x_{m}\}\cap X_\beta =\emptyset.
\]  
Then by Theorem~\ref{thm:non-div-par} applied with $\eta=\beta$ and the point $x_m=h_m\Gamma$, we thus get that 
there exists a unipotent $\bbq$-subgroup ${\bf W}$ which is not normal in $\G$ so that
\[
\cfun\Bigl(\eta_W(\la_{k_m-1}(u)h_m)\Bigr)\leq E \beta^{1/F}\;\text{ for all $u\in\mathsf B_U(e)$};
\]
see~\eqref{eq:nondiv-v-almost-fixed-statement} also Lemma~\ref{lem:nilp-subalg}. 
This and Lemma~\ref{l;polygood} implies that
\be\label{eq:def-E'-poly}
\cfun\Bigl(\eta_W(\la_{k_m}(u)h_m)\Bigr)\leq E' \beta^{1/F}\;\text{ for all $u\in\mathsf B_U(e)$}.
\ee
Therefore, we get that 
\be\label{eq:orbits-misses-X-beta}
\{\la_{k_m}(u)x_m:u\in\mathsf B_U(e)\}\cap X_{E' \beta^{1/F}}=\emptyset\;\;\text{ for all $m>m_\beta$.}
\ee
Hence the claim in part~(2) holds.
\end{proof}

\begin{proof}[Proof of Lemma~\ref{lem:linearization-noneff-eff-S-arith-t-fixed}]
The proof is based on applying Theorem~\ref{thm:eff-linearization-general} to the pieces of the orbits 
\[
\{\la_{k_m}(u)x_m: u\in\mathsf V_m\}.
\]
We show that Theorem~\ref{thm:eff-linearization-general}(3) cannot hold for the choice of $\vare$ we made in the lemma.
Further, we show that if there are infinitely many $m$ so that Theorem~\ref{thm:eff-linearization-general}(1) holds, then part~(1) in the lemma holds. In consequence, we are reduced to the case that for all but finitely many $m$ Theorem~\ref{thm:eff-linearization-general}(2) holds. In this case we use Lemma~\ref{lem:almost-tube} to conclude that part~(2) above holds.

\medskip

We begin by replacing $x_m$ with a possibly different point in the orbit which is chosen to have a representative of a controlled size. 

Assuming $Y\neq\emptyset$ and repeatedly applying Lemma~\ref{lem:cusp-or-not-limsup}, we may find natural numbers $\{n_i: i\in I\}$ with $|\log\beta_0|<n_1<n_2<\cdots$ so that if we put  
\[
J_i:=\{m\in\bbn\setminus J_{i-1}:|\{u\in\mathsf B_U(e):\la_{k_m-1}(u)x_m\not\in X_{2^{-n_i}}|\leq E2^{-n_i/F}\},
\]
for all $i\in I$ and $J_0=\emptyset$, then the following hold:
\begin{itemize}
    \item for all $i\in I$, $J_i$ is an infinite set, and
    \item for every $x\in Y$, there exists an $i\in I$, a sequence ${m_p}\to\infty$ in $J_i$, and for any $m_p$ there is some $u_{m_p}\in\mathsf B_U(e)$ so that 
\[
\la_{k_{m_p}}(u_{m_p})x_{m_p}\to x.
\]
\end{itemize}
We remark that by Lemma~\ref{lem:cusp-or-not-limsup} if $Y\neq\emptyset$, then $I\neq \emptyset$, but it may well be finite: for instance if $\{x_m\}$ is a bounded sequence, then we may choose $n_1$ large enough so that $I=\{1\}$. 

Recall the constants $E_\G$ and $F$ from Lemma~\ref{prop:eff-red-theory}.
For every $i\in I$ and $m\in J_i$ fix some $g_m\in\G$ so that 
\be\label{eq:Mi-gm-bound}
|g_m|\leq E_\G 2^{n_{i}F}=:T_i
\ee 
and $g_m\Gamma=\la_{k_m-1}(u_m)x_m\in X_{2^{-n_i}}$ for some $u_m\in\mathsf B_U(e)$. 

Recall also from~\eqref{assumption on Greek lambda} that for some $\kappa>0$ we have
\be\label{eq:doubing-section-cor}
\la_{k_m-\kappa}(\mathsf B_U(e))\la_{k_m-1}(u_m)\subset\la_{k_m}(\mathsf B_U(e)).
\ee
Let
\[
\vartheta=2|\mathsf B_U(e)|/|\la_{-\kappa}(\mathsf B_U(e))|.
\]

Apply Theorem~\ref{thm:eff-linearization-general} with $g_m$, $k_m-\kappa$, $\eta$, $t$, and
\be\label{eq:def-vare-sigma}
\vare(s)=\Bigl(\eta s^{-1} \sigma(E_1^A\eta^{-A}s^A)/2E_1\Bigr)^A;
\ee
note that $\vare$ satisfies the condition in~\eqref{condition on vare}.

We first argue that Theorem~\ref{thm:eff-linearization-general}(3) cannot hold.
Indeed, assume contrary to this claim that there exists some $\H\lhd\G$ satisfying Theorem~\ref{thm:eff-linearization-general}(3). 
That is: 
$
\height(\H)\leq E_1(e^t\eta^{-1})^A
$
and
\[
\max_{\zpz\in\mathcal B_U}\Bigl\|\zpz\wedge\vpz_H\Bigr\|\leq \vare\Bigl(\height(\H)^{1/A}\eta/E_1\Bigr)^{1/A}.
\]
In view of~\eqref{eq:def-vare-sigma} we thus get that  
\[
\max_{\zpz\in\mathcal B_U}\Bigl\|\zpz\wedge\vpz_H\Bigr\|\leq \height(\H)^{-1/A}\sigma(\height(\H))/2< \sigma(\height(\H)).
\]
However, this contradicts the definition of $\sigma$, see~\eqref{eq:def-sigmaT}.

Assume now that the conclusion in Theorem~\ref{thm:eff-linearization-general}(1) holds for a subsequence $m_i\to\infty$. 
Then, since
\[
\la_{k_{m_i}-\kappa}(\mathsf B_U(e))\la_{k_{m_i}-1}(u_{m_i})\subset \la_{k_{m_i}}(\mathsf B_U(e)),
\]
$|\la_{k_{m_i}-\kappa}(\mathsf B_U(e))|/ |\la_{k_{m_i}}(\mathsf B_U(e))|\geq 2/\vartheta$,
and $|\mathsf V_{m_i}|>\vartheta E_1\eta^{1/D}$, we have that 
\[
\{\la_{k_{m_i}}(u)x_{m_i}: u\in\mathsf V_m\}\cap \{x\in X_\eta: \text{$x$ is $(\vare,t)$-Diophantine}\}\neq\emptyset
\]
for all $m_i$. Hence $Y\cap \{x\in X_\eta: \text{$x$ is $(\vare,t)$-Diophantine}\}\neq\emptyset$ and part~(1) in the lemma holds.

Altogether, we are reduced to the case that Theorem~\ref{thm:eff-linearization-general}(2) holds for all but finitely many $m$. Dropping the first few terms, which does not effect $Y$, we assume that Theorem~\ref{thm:eff-linearization-general}(2) holds for all $m$, or more precisely that Theorem~\ref{thm:eff-linearization-general}(2) holds for $g_m$, $k_m-\kappa$, $\eta$, and $(\vare,t)$.
In particular, we have the following: For every $m\in J_1$ there exists a nontrivial proper subgroup $\Hbf_m\in\Hcal$ with 
\[
\height(\Hbf_m)\leq (E|g_m|^A+E_1e^{At})\eta^{-A}\leq (ET_1^A + E_1e^{At})\eta^{-A}=:L, 
\]
so that the following hold.
\begin{enumerate}
\item[$(\dagger)_1$] For all $u\in\mathsf B_U(e)$ we have
\[
\cfun(\eta_{H_m}(\la_{k_{m}-\kappa}(u)g_{m}))\leq L.
\]
\item[$(\ddagger)_1$] For every $u\in\mathsf B_U(e)$ we have
\[
\max_{\zpz\in\mathcal B_U}\Bigl\|\zpz\wedge{\eta_{H_m}(\la_{k_{m}-\kappa}(u)g_{m})}\Bigr\|\leq Le^{-k_{m}+\kappa/D}.
\]
\end{enumerate}

Let $\mathcal F=\{(\H,L): \height(\H)\leq L\}$. 
In view of~\eqref{eq:hcal-ht-T}, $\mathcal F$ is a finite family.

\medskip

Let now $i\in I$ and $i\geq 2$ --- we note again that it is possible that $I=\{1\}$ and this case is empty.

Arguing as in Lemma~\ref{lem:cusp-or-not-limsup}, see in particular~\eqref{eq:orbits-misses-X-beta},
for all $m\in J_i$ we have 
\[
\{\la_{k_m}(u)x_m:u\in\mathsf B_U(e)\}\cap X_{\theta_i}=\emptyset.
\]
where $\theta_i=E'2^{-n_{i-1}/F}$.
This, in view of~\eqref{eq:doubing-section-cor} implies that
\be\label{eq:Ji-bound-holds-for-gm}
\{\la_{k_m-\kappa}(u)g_m\Gamma:u\in\mathsf B_U(e)\}\cap X_{\theta_i}=\emptyset.
\ee
Therefore, by Theorem~\ref{thm:non-div-par}, for every $m\in J_i$
there exists some unipotent subgroup ${\bf W}_m$ with 
\[
\height(\mathbf{W}_m) \leq ET_i^{F} \theta_i^{1/F}=:S_i
\] 
so that
\be\label{eq:nondiv-v-almost-fixed-statement-cor}
{\cfun\Bigl(\eta_{W_m}(\la_{k_m-\kappa}(u)g)\Bigr)\leq E \theta_i^{1/F}=:s_i\;\;\text{ for all $u\in\mathsf B_U(e)$}}.
\ee
Moreover, if we put ${\bf M}_m={\bf M}_{{\bf W}_m}$, then ${\bf M}_m\neq \G$,
\[
\height(\lplus_m)\leq S_i,
\] 
and the following hold: 
\begin{enumerate}
\item[$(\dagger)_i$] For all $u\in\mathsf B_U(e)$ we have
\[
\cfun(\eta_{M}(\la_{k_m-\kappa}(u)g))\leq S_i. 
\]
\item[$(\ddagger)_i$] For all $u\in\mathsf B_U(e)$ we have
\[
\max_{\zpz\in\mathcal B_U} \|\zpz\wedge{}{\eta_{M}(\la_{k_m-\kappa}(u)g)}\|\leq S_ie^{-k_m+\kappa/F}.
\]
\end{enumerate}

Let $\mathcal E_i=\{({\bf W}, S_i, s_i): \height({\bf W})\leq S_i, {\bf W}\in\hcal \text{ is unipotent and not normal}\}$.
Then $\mathcal E_i$ is a finite family for each $i$.

\medskip

We now show that the claim in part~(2) holds with $\mathcal F$ and $\mathcal E_i,$ $i\geq 2$.
Let $x=g\Gamma\in Y$. 
Then there exists an $i\in I$ and a sequence ${m_p}\to\infty$ in $J_i$ so that the following holds.
For any $m_p$ there is some $u_{m_p}\in\mathsf V_{m_p}$ so that 
\[
\la_{k_{m_p}}(u_{m_p})g_{m_p}\Gamma\to g\Gamma.
\]
Assume first that $i=1$. Then passing to a subsequence we may assume that $(\dagger)_1$ and $(\ddagger)_1$ hold 
with $\H_{m_p}=\H$ for all $p$. Hence by Lemma~\ref{lem:almost-tube} we have
\[
g\in\{g' \in N_G(U,M_W):\cfun(\eta_{H}(g'))\leq L\}\Gamma.
\]

Similarly, if $i\geq 2$ we may pass to a subsequence and assume that ${\bf W}_{m_p}={\bf W}_m$ for all $p$. 
One then argues as in Lemma~\ref{lem:almost-tube} and gets that
\[
g\in \{g' \in N_G(U,M_W):\cfun(\eta_{M_W}(g'))\leq S_i, \cfun(\eta_W(g'))\leq s_i\}.
\]
The proof is complete.
\end{proof}

\begin{proof}[Proof of Theorem~\ref{thm:linearization-noneff-eff-S-arith}]
Let $0<\eta<1/2$ and define $\vare$ as in part~(1).

 

Recall from Definition~\ref{def:Diophantine-intro} that
\[
\{x\in X: \text{$x$ is $\vare$-Diophantine}\}=\bigcap_t \{x\in X: \text{$x$ is $(\vare,t)$-Diophantine}\}.
\]
Moreover, 
$\{x\in X_\eta: \text{$x$ is $(\vare,t)$-Diophantine}\}$
is a nested family of compact sets. 
Therefore, if Lemma~\ref{lem:linearization-noneff-eff-S-arith-t-fixed}(1) holds for all $t$, then Theorem~\ref{thm:linearization-noneff-eff-S-arith}(1) holds.  
Therefore, we may assume there exists some $t$ so that Lemma~\ref{lem:linearization-noneff-eff-S-arith-t-fixed}(2) holds.
This implies that Theorem~\ref{thm:linearization-noneff-eff-S-arith}(2) holds and completes the proof.
\end{proof}

We now state and prove analogue of Theorem~\ref{thm:linearization-noneff} in the more general $\places$-arithmetic setting. 

\begin{thm}\label{thm:linearization-noneff-S-arith}
Let $\alpha>0$. Let $\{H_i:1\leq i\leq r\}\subset\Hcal$ be a finite subset consisting of proper subgroups, and for each $1\leq i\leq r$
let $\mathcal C_i\subset N_G(U,H_i)$ be a compact subset. 
There exists an open neighborhood $\mathcal O=\mathcal O(\alpha,\{H_i\}, \{\mathcal C_i\})$ 
so that $X \setminus \mathcal O$ is compact and disjoint from
$\cup_i \mathcal C_i\Gamma/\Gamma$ so that the following holds.
For every $x\in\mathcal G(U)$ there exists some $k_0=k_0(\alpha,\{H_i\}, \{\mathcal C_i\}, x)$ 
so that for all $k\geq k_0$ we have 
\[
|\{u\in\mathsf B_U(e): \la_k(u)x\in\mathcal O\}|< \alpha 
\] 
\end{thm}

\begin{proof}
Let $\eta=(\alpha/E_1)^D$ where $D$ and $E_1$ are as in Theorem~\ref{thm:eff-linearization-general}. 

Let $x\in\mathcal G(U)$ and let $g\in G$ be so that $x=g\Gamma$.
Define
\be\label{eq:def-vare-tilde-vare}
\vare(s)=\Bigl(\eta s^{-1} \sigma(E_1^A\eta^{-A}s^A)/2E_1\Bigr)^A
\ee
where $\sigma(T)$ is defined as in~\eqref{eq:def-sigmaT}.

Let $t\in\bbr^+$ be so that $\height(\H_i)\leq e^t$ and $\cfun(\eta_{H_i}(h))<e^t$ for all $1\leq i\leq r$ and all $h\in\mathcal C_i$. 
We will show that the theorem holds with 
\[
\mathcal O=\{x\in X: \text{$x\not\in X_\eta$ or $x$ is {\em not} $(\vare,t)$-Diophantine}\}.
\]

First note that for any $i$ and any $h\in\mathcal C_i$ we have $\cfun(\eta_{H_i}(h))<e^t$ and $\zpz\wedge\eta_{H_i}(h)=0$ for all $\zpz\in\mathcal B_U$. Therefore,
\[
\cup_i \mathcal C_i\Gamma/\Gamma\subset \mathcal O.
\]

We claim there exists some $k_0$ so that for all $k\geq k_0$, Theorem~\ref{thm:eff-linearization-general}(1) holds
for $g$, $k$, and $(\vare, t)$. First note that this claim in view of the assertion in Theorem~\ref{thm:eff-linearization-general}(1) implies that 
\[
|\{u\in\mathsf B: \la_k(u)x\in\mathcal O\}|\leq E_1\eta^{1/D}=\alpha
\]
and the theorem follows. 

Let us now prove the claim. Assume contrary to the claim that Theorem~\ref{thm:eff-linearization-general}(2) or (3) holds for $g$, a sequence $k_n\to\infty$, and $(\vare, t)$.
We first show that Theorem~\ref{thm:eff-linearization-general}(3) cannot hold. Indeed, if Theorem~\ref{thm:eff-linearization-general}(3) holds, then there is some $\H\lhd\G$ with
$
\height(\H)\leq E_1(e^t\eta^{-1})^A
$
so that
\[
\max_{\zpz\in\mathcal B_U}\Bigl\|\zpz\wedge\vpz_H\Bigr\|\leq E_1\eta^{-A}\vare\Bigl(\height(\H)^{1/A}\eta^{A}/E_1\Bigr)^{1/A}.
\]
In view of~\eqref{eq:def-vare-tilde-vare} we thus get that  
\[
\max_{\zpz\in\mathcal B_U}\Bigl\|\zpz\wedge\vpz_H\Bigr\|\leq \vare(\height(\H)^{1/A})^{1/A}\sigma(\height(\H))/2<\sigma(\height(\H)).
\]
This contradicts the definition of $\sigma$ in~\eqref{eq:def-sigmaT}.

Hence we may reduce to the case that Theorem~\ref{thm:eff-linearization-general}(2) 
holds for $g$, $k_n\to\infty$, and $(\vare, t)$.
Let $L:=(E|g|^A+E_1e^{At})\eta^{-A}$. 
Then in view of Theorem~\ref{thm:eff-linearization-general}(2), for every $n$ there exists  
a proper subgroup $\H_n\in\hcal$ with $\height(\H_n)\leq L$ so that for all $u\in\mathsf B_U(e)$ we have 
\[
\max_{\zpz\in\mathcal B_U} \Bigl\|\zpz\wedge{\eta_{H_n}(\la_{k_n}(u)g)}\Bigr\|\leq  e^{-k_n/D}L.
\]
Since there are only finitely many subgroups $\H\in\Hcal$ with $\height(\H)\leq L$, see~\eqref{eq:hcal-ht-T}, 
passing to a subsequence we assume $\H_n=\H$ for all $n$. Hence 
\[
\max_{\zpz\in\mathcal B_U} \Bigl\|\zpz\wedge{\eta_{H}(\la_{k_n}(u)g)}\Bigr\|\leq  e^{-k_n/D}L.
\]
Applying this with $u=e$ and passing to the limit we get that 
\[
\zpz\wedge{\eta_{H}(g)}=0\quad\text{for all $\zpz\in\mathcal B_U$}.
\]
This contradicts the fact that $g\Gamma\in\mathcal G(U)$ and completes the proof.
\end{proof}

\section{Friendly measures}\label{sec:friendly}

In this section we discuss generalizations of our main theorems to the class of friendly measures which were studied in~\cite{KLW}, see~\S\ref{sec:friendly-intro} for the definition.

Let the notation be as in \S\ref{sec:eff-linearization}; in particular,
\[
U=\{u(t)=\exp(t\zpz): t\in\R\}.
\]
for some nilpotent element $\zpz\in\mathfrak{g}$ with $\|\zpz\|=1$.

In \cite{KLW}, an extension of Theorem \ref{Kleinbock-Tomanov theorem} for $\Sigma=\{ \infty \}$ was presented where the Haar measure on $U$
is replaced by a (uniformly) friendly measure $\mu$. While for simplicity of notation we keep our treatment of friendly measures to this case, Kleinbock and Tomanov wrote in \cite{KT:Nondiv} the $\Sigma$-arithmetic nondivergence results also for the case of friendly measures. The only difference between the statement of Theorem \ref{Kleinbock-Tomanov theorem} and the analogous statement for uniformly friendly measures (other than the obvious difference of how the size of subsets of $\lambda _ k (\mathsf B _ U)$ are measured) is that the exponent $1/D$ of the theorem is allowed to depend on the doubling constant for $\mu$. Theorem~\ref{fancy nondivergence theorem} also holds for uniformly friendly measures with the same modification.
We also note (and use below) that in view of~\cite[Prop.~7.33]{KLW}, an analogue of Lemma~\ref{l;polygood} holds true for $\mu$ in place of the Haar measure on $U$ (with a different $c$ and exponent). 

Repeating the the proof of Theorem~\ref{thm:non-div-par} but with the (uniformly) friendly versions of Theorems~\ref{Kleinbock-Tomanov theorem} and~~\ref{fancy nondivergence theorem}, we obtain the following:

\begin{thm}\label{thm:non-div-par-friendly}
Let $\mu$ be a uniformly friendly measure on $\R$. 
There exists a constant $F$ depending on $N$ and $\mu$ 
so that for any $g\in G$, $k\geq 1$, and any $0<\eta\leq 1/2$ small enough at least one of the following holds.
\begin{enumerate}
\item 
\[
\mu\Bigl(\{t\in[-1,1]: u(e^kt )g\Gamma\not\in X_{\eta}\}\Bigr)\ll \eta^{1/F}.
\]

\item There exists a unipotent $\bbq$-subgroup ${\bf W}$ so that
\[
{\|(\eta_W(u(e^kt )g)\|\ll\eta^{1/F}\;\text{ for all $t\in[-1,1]$}}.
\]
Moreover, if we put ${\bf M}={\bf M}_{\bf W}$, then ${\bf M}\neq \G$,
\[
\height(\lplus)\ll |g|^{F}\eta^{1/F},
\] 
and we have: 
\begin{enumerate}
\item For all $t\in[-1,1]$ we have
\[
\|\eta_{M}(u(e^kt )g)\|\ll |g|^F\eta^{1/F}. 
\]
\item For all $t\in[-1,1]$ we have
\[
\max_{\zpz\in\mathcal B_U} \|\zpz\wedge{\eta_{M}(u(e^kt )g)}\|\ll |g|^F\eta^{1/F}e^{-k/F}.
\]
\end{enumerate}


\end{enumerate}
\end{thm}

Similarly, the proof of Theorem~\ref{thm:eff-linearization-general-real} is easily adapted to the friendly case, giving:

\begin{thm}\label{thm:eff-linearization-friendly}
Let $\mu$ be a uniformly friendly measure on $\bbr$. There are constants $A, D$ depending only on $N$ and $\mu$, and $E_1$ depending on $N$, $G$, $\Gamma$, and $\mu$ so that the following holds. Let $g\in G$, $k\geq 1$, and $0<\eta<1/2$. Assume $\vare:\bbr^+ \to (0,1)$ satisfies 
for any $s \in \bbr^+$ that
\[
\vare(s) \leq  \eta^A s^{-A} /E_1.
\]
Then at least one of the following three possibilities holds.
\begin{enumerate}
\item 
\begin{equation*}
\mu\biggl(\biggl\{\xi \in [-1,1]:
\begin{array}{c} 
u(e^k \xi)g\Gamma\not\in X_\eta\text{ or }\\
 \text{$u(e^k \xi)g\Gamma$ is not $(\vare,t)$-Diophantine}
\end{array}
\biggr\}\biggr)
< E_1\eta^{1/D}
\end{equation*} 

\item There exist a nontrivial proper subgroup $\Hbf\in\Hcal$ of \[\height(\Hbf)\leq E_1 ( |g|^A+ e^{At})\eta^{-A}\] 
so that the following hold for all $\xi \in [-1,1]$:
\begin{align*}
\|\eta_{H}(u(e^k\xi)g)\|&\leq E_1 ( |g|^A+ e^{At})\eta^{-A}\\
\Bigl\|\zpz\wedge{\eta_{H}(u(e^k\xi)g)}\Bigr\|&\leq E_1 e^{-k/D} ( |g|^A+ e^{At})\eta^{-A}
\end{align*}
where $\zpz$ is as in~\eqref{def of z in intro}.
\item There exist a nontrivial proper normal subgroup $\Hbf \lhd \Gbf$ of  \[\height(\Hbf)\leq E_1  e^{At}\eta^{-A}\]
so that 
\[
 \Bigl\|\zpz\wedge\vpz_{H}\Bigr\|\leq E_1 \eta^{-A}  \vare(\height(\Hbf)^{1/A} \eta^{A}/ E_1 )^{1/A}.
\]

\end{enumerate}
\end{thm}

As a consequence Theorem~\ref{thm:linearization-noneff-eff-friendly} follows, see the proof of Theorem~\ref{thm:linearization-noneff-eff-S-arith}. We also get the following analogues of Theorem~\ref{thm:linearization-noneff} whose proof is mutatis mutandis the same as the proof of Theorem~\ref{thm:linearization-noneff-S-arith}.

\begin{thm}\label{thm:linearization-noneff-friendly}
Let $\mu$ be a uniformly friendly measure on $\R$. 
Let $\eta>0$. Let 
\[
\{H_i:1\leq i\leq r\}\subset\Hcal
\] 
be a finite subset consisting of proper subgroups, and for each $1\leq i\leq r$
let $\mathcal C_i\subset N_G(U,H_i)$ be a compact subset. 
There exists an open neighborhood $\mathcal O=\mathcal O(\alpha,\{H_i\}, \{\mathcal C_i\})$ 
so that $X \setminus \mathcal O$ is compact and disjoint from
$\cup_i \mathcal C_i\Gamma/\Gamma$ so that the following holds. 
For every $x\in\mathcal G(U)$ there exists some $k_0=k_0(\mu,\eta,\{H_i\}, \{\mathcal C_i\}, x)$ 
so that for all $k\geq k_0$ we have 
\[
\mu\Bigl(\{t\in[-1,1]: u(e^kt)x\in\mathcal O\}\Bigr)< \eta  
\] 
\end{thm}

\appendix
\section{Proof of Theorem~B}\label{sec:thmB}

In this section we prove Theorem~B.
In qualitative form, this is proved by Greenberg in~\cite{GrBerg-Loj-1} and~\cite{GrBerg-Loj-2}. 
We reproduce the argument here to make the estimates explicit.

\begin{proof}[Proof of Theorem~B]
Let $\bbc_p$ denote the completion of the algebraic closure of $\bbq_p$ for all $p\in\places_f$;
as {\em abstract} fields $\bbc$ and $\bbc_p$, for any $p\in\places_f$, are isomorphic.
Therefore, $\bbc^m$ in {\em Effective Nullstellensatz} theorem of~\S\ref{sec:Loj-inequality} may be replaced by $\bbc_p^m$ for any $p\in\places_f$. 

As in~\cite[p.~59--60]{GrBerg-Loj-1} and~\cite[Steps 1 and 2]{GrBerg-Loj-2} we begin with some reductions. 

Let $\mathcal I\subset\bbz[t_1,\ldots,t_m]$ be the ideal generated by $\{f_i\}$,
and let ${\bf Y}$ be the variety defined by $\mathcal I$ in $\bbc_p^m$.

Put $\mathcal J:=\mathcal I\bbq[t_1,\ldots,t_m]$.
The radical and the primary decomposition of $\mathcal J$ in $\bbq[t_1,\ldots,t_m]$ 
can be computed, see~\cite[Chap.~8.7]{BeWei-Grobner}; 
this computation uses Gr\"{o}bner basis and yields the following.  There exists a computable 
constant $s=s(m, n, D_0)$ so that
\begin{itemize}
\item $\Bigl(\sqrt\mathcal J\Bigr)^s\subset\mathcal J$,
\item $\sqrt\mathcal J=\cap_{1}^b\mathcal P_j$ where $\mathcal P_j$ is a prime ideal for all $1\leq j\leq b\leq s,$ and
\item $\mathcal P_j$ is generated by $\{f_{j,\ell}:1\leq \ell\leq s\}$ where the total degree of $\{f_{j,\ell}\}$ is bounded by $s$ 
and the logarithmic height of the numerators and denominators of each $f_{j,\ell}$ is controlled by $s\mathsf h$.
\end{itemize}

Moreover, by~\cite[Cor.\ 3.8]{GiaTrZa}, we may replace $s$ with $s'\geq s$,
which is again computable and depends only on $m$, $n$, and $D_0$, so that the following holds.
For every $1\leq j\leq b$ the ideal $\mathcal P_j\cap\bbz[t_1,\ldots,t_m]$ is generated by 
$\{g_{j,k}:1\leq k\leq s'\}$, furthermore, the total degree of $\{g_{j,k}\}$ is bounded by $s'$ 
and for every $j$ and $k$ the logarithmic height of $g_{j,k}$ is bounded by $s'\mathsf h$.

Altogether, we may assume that 
$\mathcal I$ is a prime ideal, i.e.\ $\mathbf Y$ is $\bbq$-irreducible. 

We now use induction on $u:=\dim\mathbf Y$ to prove the claim, 
see~\cite[Case 1, p.~60]{GrBerg-Loj-1} and~\cite[Step 3]{GrBerg-Loj-2}.

The base case is when $u=-1$, that is: when $\mathcal I$ contains a nonzero constant.
In this case we use the {\em effective nullstellensatz theorem} above and find 
some $a\in\bbz$ with 
\[
\log|a|\leq (8D_0)^{4M-1}(\mathsf h+8D_0\log(8D_0))
\]
where $M=2^{m-1}$ so that $a=\sum_i q_if_i$.

This implies the claim in the theorem when $u=-1$.

Assume now that $\mathbf Y$ is nonempty and that the theorem is 
established in dimensions less than $u$. 
Let ${\rm Jac}$ be the Jacobian matrix of $\{f_i\}$ and let $\Delta$ be the system of minors 
of order $m-u$ taken from ${\rm Jac}$.
Since ${\rm char}(\bbq)=0$, the locus of common zeros of $\{\Delta,\{f_i\}\}$ 
is a proper $\bbq$-subvariety of ${\bf Y}$. 
By inductive hypothesis, thus, there exists some $\dpz'$ depending on $m$, $n$, and $D$ 
which satisfies the claim in the theorem for $\{\Delta, \{f_i\}\}$. 

For any $1\leq\alpha_1<\cdots<\alpha_{m-u}\leq n$, put $(\alpha)=(\alpha_1,\ldots,\alpha_{m-u})$ 
and set $f_{(\alpha)}=\{f_{\alpha_1},\ldots,f_{\alpha_{m-u}}\}$. 
Let $\mathbf Y_{(\alpha)}$ be the variety defined by $f_{(\alpha)}$.  
Let $\mathbf Z_{(\alpha)}=\cup_{j=1}^c\mathbf Z_{(\alpha),j}$ where for all $1\leq j\leq b$, 
we have $\mathbf Z_{(\alpha),j}\subset\mathbf Y_{(\alpha)}$,  $\mathbf Z_{(\alpha),j}$ is $\bbq$-irreducible with
$\dim\mathbf Z_{(\alpha),j}=u$, and $\mathbf Z_{(\alpha),j}\neq\mathbf Y$.

Let $\mathcal I_{(\alpha),j}\subset \bbz[t_1,\ldots,t_m]$ be the ideal corresponding to $\mathbf Z_{(\alpha),j}$.  
Since $\mathbf Y_{(\alpha)}$ is defined by $f_{(\alpha)}$, a similar argument as above 
implies that there exists a computable constant $r=r(m,n,D_0)$ so that
\begin{itemize}
\item $c\leq r$, and
\item for every $1\leq j\leq c$, there exists $\{g_{(\alpha),j,k}:1\leq k\leq r\}$ so that $\mathcal I_{(\alpha),j}$ 
is generated by $\{g_{(\alpha),j,k}\}$, further, the total degree of $\{g_{(\alpha),j,k}\}$ is bounded by $r$ 
and for every $j$ and $k$ the logarithmic height of $g_{(\alpha),j,k}$ is controlled by $r\mathsf h$
\end{itemize}

Since $\mathbf Z_{(\alpha),j}\neq\mathbf Y$ for all $j$, by inductive hypothesis, 
there exists $\dpz_{(\alpha)}'$ depending on $m$, $n$, and $D_0$ which satisfies the claim in the theorem for
\[
\Bigl\{\{g_{((\alpha),j,k)}:1\leq k\leq r\}, \{f_i:1\leq i\leq n\}\Bigr\},
\]
for all $1\leq j\leq c.$

Given $(\alpha)=(\alpha_1,\ldots,\alpha_{m-u})$ and $(\beta)=(\beta_1,\ldots,\beta_{m-u})$
let $\Delta_{(\alpha),(\beta)}$ denote the corresponding minor from ${\rm Jac}$. 
By the implicit function theorem, if $z\in\mathbf Y_{(\alpha)}$ is such that 
$\Delta_{(\alpha),(\beta)}(z)\neq 0$ for some $(\beta)$, then $z$ lies on exactly one component of 
$\mathbf Y_{(\alpha)}$, moreover, that component has dimension $u$.

Let $r'=r'(m,n,D_0)$ be so that the logarithmic height of $\{\Delta,\{f_i\}\}$ is bounded by $r'\mathsf h$.  
Define 
\[
\dpz:=2r'\dpz'+r\sum_{(\alpha)}\dpz'_{(\alpha)}.
\] 
We claim that the theorem holds with $\ref{k:GrBr-Loj}=\dpz$.

Let $w=(w_1,\ldots,w_m)$ be as in the statement of the theorem. If either
\begin{enumerate}
\item $v_p(\Delta_{(\alpha,(\beta))}(w))> 2r'\dpz'\mathsf h$ for all $(\alpha)$ and $(\beta)$, or
\item $v_p(g_{(\alpha), j, k}(w))> 2r\dpz'_{(\alpha)}\mathsf h$ for some $(\alpha)$, some $j$, and all $k$,
\end{enumerate}
then we get the claim from the inductive hypothesis.

Therefore, we may assume that there are $(\alpha)$ and $(\beta)$ so that 
\be\label{eq:Hensel-1}
v_p(\Delta_{(\alpha),(\beta)}(w))\leq 2r'\dpz'\mathsf h,
\ee
and for every $(\theta)$ and every $j$ there exists some $k$ so that 
\be\label{eq:Hensel-2}
v_p(g_{(\theta), j, k}(w))\leq 2r\dpz'_{(\theta)}\mathsf h.
\ee

Now a suitable version of Hensel's Lemma, see~\cite[Note 1]{GrBerg-Loj-2}, 
implies that there exists some $y\in\bbz_p^m$ so that $f_{(\alpha)}(y)=0$ and
\be\label{eq:Hensel-y-w}
v_p(y-w)>C_2-2r'\dpz'\mathsf h.
\ee
The theorem follows if we show that $y\in\mathbf Y$.

Let us recall that
\be\label{eq:Hensel-C-big}
C_2>4r'\dpz'\mathsf h+2r\Bigl(\max\{\dpz'_{(\theta)}:(\theta)\}\Bigr)\mathsf h.
\ee
Then,~\eqref{eq:Hensel-2},~\eqref{eq:Hensel-y-w}, and~\eqref{eq:Hensel-C-big}, 
imply that $v_p(g_{(\theta), j, k}(w))=v_p(g_{(\theta), j, k}(y))$. 
In particular, $y\not\in\mathbf Z_{(\theta),j}$ for all $(\theta)$ and all $j$. 
Similarly,~\eqref{eq:Hensel-1},~\eqref{eq:Hensel-y-w}, and~\eqref{eq:Hensel-C-big}, imply 
that $\Delta_{(\alpha),(\beta)}(y)\neq0.$

Thus, the implicit function theorem implies that $y$ belongs to $\mathbf Y$.
\end{proof}

\end{document}